%% file: reAPCG.tex
\documentclass[11pt,letterpaper]{article}
\usepackage{wrapfig}
\usepackage{graphicx}
\usepackage{array}
\usepackage{booktabs}
\usepackage{makecell}
\input{header.tex}

\DeclareMathOperator{\Adj}{Adj}
\usepackage[inline, shortlabels]{enumitem}
\def\var #1#2{#1^{(#2)}}
\def\dotprod #1#2{\langle #1, #2 \rangle}
\usepackage{multicol}
\def\grad #1#2{\nabla_{#1} #2}
\def\E #1#2{\mathbb{E}_{#1} \left[ #2\right]}

\def\smallspace{\hspace*{0.2in}}
\def\largespace{\hspace*{0.8in}}

\def\aalpha{\phi}
\def\bbeta{\varphi}
\def\ggamma{\psi}

\def\zzeta{\zeta}
\def\hatvar #1#2{\hat{#1}^{(#2)}}

\def\tildevar #1#2{\tilde{#1}^{(#2)}}

\def\algmu {\tilde{\mu}}
\def\para {\eta}
\def\parb {\tau}

\def\var #1#2{#1^{(#2)}}

\newcommand*{\tran}{^{\mkern-1.5mu\mathsf{T}}}

\def\kk{t}
\def\tt{l}
\def\iit{k}
\def\tFE{\text{FE}}
\def\ivec{\mathsf{1}}

\title{
An Analysis of Asynchronous Stochastic Accelerated Coordinate Descent \thanks{This work was supported in part by NSF Grant CCF-1527568.}
}
\author{Richard Cole \\ 		Courant Institute, NYU
\and
Yixin Tao
\\		Courant Institute, NYU
}
\date{}

\begin{document}

\maketitle

\input{abstract}

\newpage
\thispagestyle{plain}\setcounter{page}{1}

\input{introduction}
\input{prelim-rjc-copy}

\input{Algorithm-rjc-copy}
\input{AnalysisOutline}

\newpage

\appendix
\input{AppThmProof}

\input{app-th-alg-equiv}
\newpage
\input{Progress-Lemma}
\newpage

\input{app-FE-Delta-bounds}

\input{OtherProofsFEDeltaBounds}

\input{Diff-w-z}
\input{Diff_y}
\newpage
\input{TechnicalLemmas}
\input{app-Lres-Lresbar-diff}

\input{appendix-counter}
\input{amortization}

\newpage
\bibliographystyle{plain}
\bibliography{bib}

\end{document}

%% file: header.tex
\usepackage{amsmath,amsthm,amsfonts,amssymb,color,hyperref,anysize,enumerate,graphicx,epstopdf,cleveref,multirow}
\usepackage[usenames,dvipsnames]{xcolor}
\usepackage[margin=1.0in]{geometry}
\usepackage[ruled,vlined,linesnumbered]{algorithm2e}

\DeclareMathOperator*{\res}{res}
\newcommand{\rr}{\mathbb{R}}

\newcommand{\la}{\leftarrow}
\newcommand{\ra}{\rightarrow}

\newcommand{\red}[1]{{\color{red} #1}}

\newenvironment{pfof}[1]{\begin{proof}[\emph{\textbf{Proof of #1: }}]}{\end{proof}}

\usepackage{tcolorbox}

\newcommand{\rjc}[1]{{\color{blue}{#1}}}
\newcommand{\RJC}[1]{{\textbf{\color{Orange}RJC: #1}}}

\newcommand{\yxt}[1]{{\color{red}{#1}}}
\newcommand{\YXT}[1]{\hide{\textbf{\color{red}YXT: #1}}}

\tcbuselibrary{theorems}

\newtcbtheorem[auto counter]{coloredDEF}{Definition}
{colback=blue!5,colframe=blue!40!black,fonttitle=\bfseries,before={\vspace{0.3cm}},after={\vspace{0.3cm}}}{de}

\newtcbtheorem[auto counter]{coloredNOT}{Notation}
{colback=blue!5,colframe=blue!40!black,fonttitle=\bfseries,before={\vspace{0.3cm}},after={\vspace{0.3cm}}}{no}

\newtcbtheorem[auto counter]{assume}{Assumption}
{colback=green!5,colframe=green!60!black!80,fonttitle=\bfseries,before={\vspace{0.3cm}},after={\vspace{0.3cm}}}{assume}

\newtcbtheorem[auto counter]{coloredLEM}{Lemma}
{colback=red!5,colframe=red!60!black,fonttitle=\bfseries,before={\vspace{0.3cm}},after={\vspace{0.3cm}}}{lm}

\newtcbtheorem[auto counter]{coloredTHM}{Theorem}
{colback=purple!5,colframe=purple!50!black,fonttitle=\bfseries,before={\vspace{0.3cm}},after={\vspace{0.3cm}}}{th}

\newtheorem{theorem}{Theorem}
\newtheorem{lemma}[theorem]{Lemma}

\newtheorem{obs}[theorem]{Observation}

\newtheorem{defn}{Definition}

\newtheorem{remark}[theorem]{Remark}

\newcommand{\veca}{\vec{a}}
\newcommand{\vecb}{\vec{b}}

\newcommand{\barT}{T}

\newcommand\numberthis{\addtocounter{equation}{1}\tag{\theequation}}

\newcommand{\hide}[1]{}

%% file: abstract.tex
\begin{abstract}
\thispagestyle{empty}

Gradient descent, and coordinate descent in particular, are core tools in machine learning and elsewhere.
Large problem instances are common. 
To help solve them, two orthogonal approaches are known: acceleration and parallelism.
In this work, we ask whether they can be used simultaneously.
The answer is ``yes''.

More specifically, we consider an asynchronous parallel version of the accelerated coordinate descent algorithm
proposed and analyzed by Lin, Liu and Xiao~\cite{LinLuXiao2015}. 
We give an analysis based on the efficient implementation of this algorithm. 
The only constraint  is a standard bounded
asynchrony assumption, namely that each update can overlap with at most $q$ others.
($q$ is at most the number of processors times the ratio in the lengths
of the longest and shortest updates.)
We obtain the following three results:
\begin{itemize}
\item
A linear speedup for strongly convex functions so long as $q$ is not too large.
\item
A substantial, albeit sublinear, speedup for strongly convex functions for  larger $q$.
\item
A substantial, albeit sublinear, speedup for convex functions.
\end{itemize}

\end{abstract}

%% file: introduction.tex
\newcommand{\Lresbar}{L_{\overline{\res}}}

\section{Introduction}
\label{sec::introduction}
\newcommand{\tx}{\tilde x}

We consider the problem of finding an (approximate) minimum point of a convex function $f:\rr^n\ra\rr$ having a Lipschitz
bound on its gradient.
\hide{As shown in Lin et al.~\cite{LinLuXiao2015}, this is a dual of the Empirical Risk Minimization (ERM) problem, a central
tool in machine learning (ERM includes the following methods: SVM, Lasso, ridge regression and logistic regression). \YXT{Lin et al. may not be a good example as they need non-smooth part.}
\RJC{Good point. This was copied from one of the papers by
Sun et al. We need to find something else if we can.}}

Gradient descent is the standard solution approach for huge-scale problems of this type.
Broadly speaking, gradient descent proceeds by moving iteratively in the direction of the negative gradient of a convex function.
Coordinate descent is a commonly studied version of gradient descent.
It repeatedly selects and updates a single coordinate of the argument to the convex function.
Stochastic versions are standard: 
at each iteration the next coordinate to update is chosen uniformly at random\footnote{There are also versions
in which different coordinates can be selected with different probabilities.}.

\textbf{Speed up by acceleration}~
Acceleration is a well-known technique for improving the rate of convergence.
It improves the rate from
$\Theta([1 - \Theta(\mu)]^T)$ to
$\Theta([1 - \Theta(\sqrt{\mu}])^T)$
on strongly convex functions with
strong convexity parameter $\mu$ (defined in Section~\ref{sec::prelim}), and from
$\Theta(1/T)$ to $\Theta(1/T^2)$ on convex functions (here $\mu = 0$). \hide{\YXT{I don't understand what this $\Theta(1/\sqrt{\mu})$ means.}\RJC{Reworded.}}
So the gains are most significant when $\mu$ is small.

\textbf{ Speed up by parallelism}~
Another way to achieve speedup and thereby solve larger problems is parallelism. There have been multiple analyses of various parallel implementations of coordinate descent~\cite{LWRBS2015,LiuW2015,MPPRRJ2015,Sun2017,2016arXiv161209171K,CCT2018,richtarik2016parallel,bradley2011parallel}.

One important issue in parallel implementations is whether the different processors
are all using up-to-date information for their computations.
To ensure this requires considerable synchronization, locking, and consequent waiting.
Avoiding the need for the up-to-date requirement, i.e.\ enabling asynchronous updating, was a significant advance. The advantage of asynchronous updating is it reduces and potentially eliminates the need for waiting.
At the same time, as some of the data being used in calculating updates will be
out of date, one has to ensure that the out-of-datedness is bounded in some fashion. In this paper, we ask the following question:
\vspace*{0.1in}

\centerline{\emph{\textbf{Can acceleration and asynchronous parallel updating be applied}}}

\centerline{\emph{\textbf{simultaneously and effectively to coordinate descent?}}}\vspace*{0.1in}

It was an open question whether the errors introduced by parallelism and asynchrony would preclude the
speedups due to acceleration. For Devolder et al.~\cite{DevolderGN2014} have shown that with arbitrary errors in the computed
gradients $g$ that are of size $\Theta(\epsilon g)$ for some constant $\epsilon > 0$, in general,
speedup due to acceleration cannot be maintained for more than a bounded number of steps. More specifically,
they observe that the superiority
of fast gradient methods over classical ones is no longer absolute when an
inexact gradient is used. They show that, contrary to simple gradient schemes,
fast gradient methods must necessarily suffer from error accumulation.
In contrast, although the ``errors'' in the gradient values in our algorithm may be of
size $\Theta(\epsilon g)$, or even larger,
it turns out there is sufficient structure to enable both the speedup due to acceleration and
a further speedup due to parallelism. 

\textbf{Modeling asynchrony}~
The study of asynchrony in parallel and distributed computing
goes back to Chazen and Miranker~\cite{ChazenMir1969} for linear systems
and to Bertsakis and Tsitsiklis for a wider range of computations~\cite{BertsakisTsiTsi1969}.
They obtained convergence results for both deterministic and stochastic algorithms along with rate of convergence results for deterministic algorithms.
The first analyses to prove rate of convergence bounds for stochastic asynchronous 
computations
were those
by Avron, Druinsky and Gupta~\cite{avron2015revisiting} (for the Gauss-Seidel algorithm),
and Liu and Wright~\cite{LiuW2015} (for coordinate descent);
they called this the ``inconsistent read'' model.%
\footnote{``Consistent reads''
mean that all the coordinates a core read may have some delay, but they must appear simultaneously at some moment.
Precisely, the vector of $\tx$ values used by the update at time $t$ must be $x^{t-c}$ for some $c\ge 1$.
``Inconsistent reads'' mean that the $\tx$ values used by the update at time $t$ can be any of the $(x_1^{t-c_1},\cdots,x_n^{t-c_n})$,
where each $c_j\ge 1$ and the $c_j$'s can be distinct.}
We follow this approach; and also, following Liu and Wright, we assume there is a bounded amount of overlap between the various updates, but this is the only assumption we use in our analysis.
\hide{
\RJC{We need to check what Bertsakis and Tsitsiklis assumed.
They certainly had an asynchronous model and so must have been using out-of-date information in some fashion.
This is much older work.}\YXT{It seems that they use the message, and the communication delay is bounded.}
\RJC{It seems to me they must have out-of-date information.
What I recall is that they had a strong assumption on the form
of the function which is how they were able to show convergence.
But I don't remember anything about their read model.}\YXT{It seems to me that their algorithm more like an elimination on disagreement of different processors.}
\RJC{I think the distinction to make is between stochastic and
non-stochastic analyses. The work by Bertsakis and Tsitsiklis
appears to be non-stochastic.}
}

\textbf{Analysis and results}
Our analysis has two starting points: the analysis of the sequential accelerated stochastic coordinate descent 
by Lin et al.~\cite{LinLuXiao2015} and the analysis of the stochastic asynchronous coordinate descent
by Cheung et al.~\cite{CCT2018}. We now state our results informally.
\begin{theorem}[Informal]
Let $q$ be an upper bound on how many other updates a single update can overlap.
$L_{\overline{res}}$ is a Lispschitz parameter defined in
Section~\ref{sec::prelim}.

\noindent
(i)
Let $f$ be a strongly convex function with strongly convex parameter $\mu$.
If $q = O\left(\frac{\sqrt{n} \mu^{\frac{1}{4}}}{\Lresbar}\right)$, then $f(x^{T}) - f^*  = \Theta\left(\left(1 - \frac{1}{5} \frac{\sqrt{\mu}}{n}\right)^{T} \right)$ (linear speedup).
While if $q = O\left(\frac{\sqrt{n}}{L_{\overline{res}}}\right)$, then $f(x^{T}) - f^* =  \Theta\left(\left(1 - \frac{1}{8} \frac{\mu^{\frac{2}{3}}}{n}\right)^{T} \right)$
(sublinear speedup).

\noindent
(ii) Let $f$ be a (non-strongly) convex function.
If $q = O\left(\frac{\sqrt{\rjc{\epsilon}n}}{L_{\overline{res}}}\right)$, then $f(x^{T+1}) - f^* =  \Theta\left(\left(\frac{n}{T}\right)^{\frac{3}{2} - \epsilon} \right)$ (sublinear speedup).
\end{theorem}

\textbf{Comparison to prior works}~
Fang et al.~\cite{fang2018accelerating} and Hannah et al.~\cite{hannah2018a2bcd} have also recently analyzed asynchronous accelerated coordinate descent.
However, there are substantial differences.
First, their analyses at best only partially account for the efficient implementation of this algorithm. 
Second, their analyses only consider the strongly convex case. Finally, they use the \emph{Common Value} assumption which significantly limits the possible asynchrony
(a discussion of this issue phrased in terms of ``delay sequences'' can be found in~\cite{Sun2017}).

The \emph{Common Value} assumption strikes us as somewhat unnatural.
It states that the random choice of coordinate by one core does not affect the values it reads, and also does not affect the overlapping computations performed by other cores. For a more detailed discussion, please see \cite{CCT2018}.
However, it simplifies the analysis of asynchronous coordinate descent, which may explain why it has been used in multiple papers.

\textbf{Related work}~
Coordinate Descent is a method that has been widely studied; see Wright for a recent survey \cite{wright2015coordinate}.

Acceleration, in the spirit of accelerated gradient descent \cite{nesterov1983method}, has been used to achieve a faster rate of convergence for coordinate descent. In particular, Nesterov \cite{nesterov2012efficiency} proposed an accelerated version of coordinate descent. Lee and Sidford \cite{lee2013efficient} also developed an accelerated coordinate descent and focused on its application to solving linear systems. Xiao et al.~\cite{LinLuXiao2015} developed an accelerated proximal coordinate descent method for minimizing convex composite functions. Fercoq et al.~\cite{fercoq2015accelerated} gave a generalized version of accelerated coordinate descent. Zhu et al.~\cite{allen2016even} developed a faster accelerated coordinate descent by updating coordinates with different probabilities.

Considerable attention has also been given to 
applying asynchronous updates to coordinate descent in order to achieve quicker convergence.  
There have been multiple analyses of various asynchronous parallel implementations of coordinate descent~\cite{LWRBS2015,LiuW2015,MPPRRJ2015,Sun2017,2016arXiv161209171K,CCT2018}, 
several demonstrating linear speedup, with the
best bound, in~\cite{LWRBS2015,CCT2018}, showing linear speedup for up to
$\Theta(\sqrt{n})$ processors,\footnote{The achievable 
speedup depends on the Lipschitz parameters of the convex function $f$, 
as well as on the relative times to compute the different updates.}
\footnote{The result in~\cite{CCT2018} is under a much more general model of asynchrony.}
along with a matching lower bound~\cite{CCT2018} showing this is 
essentially the maximum available
speedup in general.

%% file: prelim-rjc-copy.tex
\section{Preliminaries}\label{sec::prelim}

\newcommand{\Ap}{A^+}
\newcommand{\Am}{A^-}
\newcommand{\Dx}{\Delta x}
\newcommand{\G}{\Gamma}
\newcommand{\hWj}{\widehat{W}_j}
\newcommand{\hWk}{\widehat{W}_k}
\newcommand{\hd}{\widehat{d}}
\newcommand{\hdj}{\widehat{d}_j}
\newcommand{\hdk}{\widehat{d}_k}
\newcommand{\kt}{k_{t}}
\newcommand{\Ljj}{L_{jj}}
\newcommand{\Lkj}{L_{kj}}
\newcommand{\Ljk}{L_{jk}}
\newcommand{\Lmax}{L_{\max}}
\newcommand{\Lmin}{L_{\min}}
\newcommand{\Lres}{L_{\res}}
\newcommand{\muf}{\mu_f}
\newcommand{\pkt}{x_{k_t}}
\newcommand{\pktaum}{x_k^{\tau-1}}
\newcommand{\ptauo}{x^{\tau-1}}
\newcommand{\ptauprm}{x^{\tau+n-1}}
\newcommand{\pt}{x^{t}}
\newcommand{\ptone}{x^{t-1}}
\newcommand{\ptp}{x^{t+1}}
\newcommand{\tg}{\tilde{g}}
\newcommand{\ve}{\vec{e}}
\newcommand{\Wj}{W_j}
\newcommand{\Wk}{W_k}


We consider the problem of finding an (approximately) minimum point of a convex function $f:\rr^n\ra\rr$.
Let $X^*$ denote the set of minimum points of $f$; we use $x^*$ to denote a minimum point of $f$.
Without loss of generality, we assume that $f^*$, the minimum value of $f$, is $0$.

We recap a few standard terminologies. Let $\ve_j$ denote the unit vector along coordinate $j$.

\begin{defn}\label{def:Lipschitz-parameters}
The function $f$ is $L$-Lipschitz-smooth if for any $x,\Dx\in\rr^n$, $\|\nabla f(x+\Dx) - f(x)\| ~\le~ L\cdot\|\Dx\|$.
For any coordinates $j,k$, the function $f$ is $\Ljk$-Lipschitz-smooth if for any $x\in\rr^n$ and $r\in\rr$,
$|\nabla_k f(x+r\cdot \ve_j) - \nabla_k f(x)| ~\le~ \Ljk\cdot |r|$;
it is $\Lres$-Lipschitz-smooth if $||\nabla f(x+r\cdot \ve_j) - \nabla f(x)|| ~\le~ \Lres\cdot |r|$.
Finally, $\Lmax ~:=~ \max_{j,k} \Ljk$ and  $\Lresbar ~:=~ \max_k \left(\sum_{j=1}^n (\Lkj)^2\right)^{1/2}$.
\end{defn}

$\Lresbar$ was introduced in~\cite{CCT2018} to account for the effect of No Common Value 
on the averaging in our analysis. If Common Value was assumed, 
the parameter $\Lres \le \Lresbar$ would suffice. 
We recap the discussion of the difference between $\Lres$ and $\Lresbar$ from~\cite{CCT2018} in Appendix~\ref{app::Lres-Lresbar-diff}.
We note that if the convex function is $s$-sparse, meaning that each term $\nabla_k f(x)$ depends on at most $s$ variables, then $\Lresbar \le \sqrt s \Lmax$. When $n$ is huge, this would appear to be the only feasible case.

Next, we define strong convexity.
\begin{defn}
\label{def::str-conv}
Let $f:\rr^n\ra \rr$ be a convex function. 
$f$ is strongly convex with parameter $0 < \muf \le 1$, if for all 
$x,y$, $f(y) - f(x) \ge \langle \nabla f(x), y-x \rangle + \frac 12 \muf ||y-x||_L^2$,
where $||y-x||_L^2 = \sum_j \Ljj(y_j -x_j)^2$. 
\end{defn}

By a suitable rescaling of variables, we may assume that $\Ljj$ is the same for all $j$ and equals 1.
This is equivalent to using step sizes proportional to $\Ljj$ without rescaling, a common practice.
Note that rescaling leaves the strong convexity parameter $\mu$ unchanged. 
Since we measure distances by $\| \cdot \|_L$ 
(this is the same as rescaling and measuring distances by $\| \cdot \|_2$)
and as our strongly convex parameter is defined with respect to $\| \cdot \|_L$, choosing the coordinate to update uniformly is the best choice. This is also the case for the accelerated algorithm analyzed in  the work of Zhu et al.~\cite{allen2016even}, where uniform sampling is the best choice with the measure $\| \cdot \|_L$.
\cite{allen2016even} also consider other measures,
including the measure $\| \cdot \|_2$ without rescaling,
where non-uniform sampling is a better choice.
We note that the accelerated asynchronous algorithm
in~\cite{hannah2018a2bcd}
analyzes the non-uniform sampling case (which includes
uniform sampling as a special case).

{\bf The update rule}~
The basic time $t$ iteration for our accelerated coordinate descent is shown in Algorithm~\ref{alg::apcg::asyn} (we will explain the meaning of the term $\pi$ in the superscripts shortly --- for now simply ignore this term and view the superscript as indicating a time).
The values $ \ggamma_{\kk}, \bbeta_{\kk},  \aalpha_{\kk}$
are suitable parameters satisfying
$0 \le  \ggamma_{\kk}, \bbeta_{\kk}, \aalpha_{\kk} \le 1$.

\begin{algorithm}
\caption{The basic iteration}
\label{alg::apcg::asyn}
	Choose ${\iit}_{\kk} \in \{1, 2, \cdots, n\}$ uniformly at random\;
	$\var{y}{\kk, \pi} = \ggamma_{\kk} \var{x}{\kk, \pi} + (1 - \ggamma_{\kk}) \var{z}{\kk, \pi}$\;
	$\var{w}{\kk, \pi} = \bbeta_{\kk} \var{z}{\kk, \pi} + (1 - \bbeta_{\kk}) \var{y}{\kk, \pi}$\;
	$\var{z}{\kk+1, \pi} = \arg\min_x \{ \frac{\rjc{\Gamma_{\kk}}}{2} \|x - \var{w}{\kk, \pi}\|^2 + \dotprod{\tilde{g}^{\kk, \pi}_{\iit_{\kk}}}{x_{\iit_{\kk}}}\}$\;
	$\var{x}{\kk+1, \pi} = \var{y}{\kk, \pi} + n \aalpha_{\kk} (\var{z}{\kk+1, \pi} - \var{w}{\kk, \pi})$\; 
\end{algorithm}

In a sequential implementation $\tilde{g}^{\kk, \pi}_{\iit_{\kk}}$ is the gradient $\nabla_{\kk}f(y^{\kk,\pi})$.
In parallel implementations, 
in general, as we will explain,
we will not have the coordinates
$y^{\kk,\pi}$ at hand, and instead we will compute
a possibly different gradient 
$\tilde{g}^{\kk, \pi}_{\iit_{\kk}} = \nabla_{\kk}f(\tilde{y}^{\kk,\pi})$.
The challenge for the analysis is to bound the effect
of using $\tilde{y}^{\kk,\pi}$ rather than $y^{\kk,\pi}$.

We note that in general Algorithm~\ref{alg::apcg::asyn} is not efficiently
implementable, and we will be using a more efficient implementation
which we will describe in Section~\ref{sec::results}.

{\bf Bounded asynchrony}~
Following~\cite{LiuW2015}, we assume that each update overlaps the computation of at most $q$ other updates.
$q$ is a function of the number of processors and the variation in the runtime of the different
updates. This variation can be due to both variance in the inherent length of the updates,
and variations coming from the computing environment such as communication delays,
interrupts, processor loads, etc.

{\bf Time in the analysis}~
Our analysis will be comparing the performance of our parallel
implementation to that of the sequential algorithm.
In order to do this, we need to impose an order on the
updates in the parallel algorithm.
As the algorithms are stochastic, we want an order
for which at each step each coordinate is equally likely
to be updated. While using commit times for
the orderings seems natural, it does not ensure this
property, so as in~\cite{MPPRRJ2015,CCT2018}, we instead use the ordering
based on start times.

Suppose there are a total of $\barT$ updates. We view the whole stochastic process as a branching tree of height $\barT$.
Each node in the tree corresponds to the \emph{moment} when some core randomly picks a coordinate to update,
and each edge corresponds to a possible choice of coordinate.
We use $\pi$ to denote a path from the root down to some leaf of this tree.
A superscript of $\pi$ on a variable will denote the instance of the variable on path $\pi$.
A double superscript of $(t,\pi)$ will denote the instance of the variable at time $t$ on path $\pi$,
i.e.\ following the $t$-th update.

\hide{
Recall that in a standard coordinate descent, be it sequential or parallel and synchronous,
the update rule, applied to coordinate $j$,
first computes the \emph{accurate} gradient $g_j := \nabla_j f(\ptone)$,
and then performs the update given below. 
\[
 \pt_j \leftarrow \ptone_j + \arg\min_d \{g_j\cdot d ~+~ \G d^2 / 2\}  \equiv \ptone_j + \hd(g,x)~~~~\text{and}~~~~\forall k\neq j,~\pt_k \leftarrow \ptone_k, 
\]
where $\G\ge \Lmax$ is a parameter controlling the step size.
However, in an asynchronous environment, an updating core (or processor) might retrieve outdated information $\tx$ instead of $\ptone$,
so the gradient the core computes will be $\tg_j^t \equiv \tg_j := \nabla_j f(\tx)$, instead of the accurate value $\nabla_j f(\ptone)$.
Our update rule, which is naturally motivated by its synchronous counterpart, is
\begin{equation}\label{eq:update-rule}
\pt_j \leftarrow \ptone_j + \hd(\tg_j,\ptone_j,\G)~=~\ptone_j+\Delta \pt_j~~~~~~~~\text{and}~~~~~~~~\forall k\neq j,~\pt_k \leftarrow \ptone_k.
\end{equation}

We let $\hWj(g,x,\G) = -[g\hd ~+~ \G \hd^2 / 2~+ d) ]$; it is well known that $\hWj(g,x,\G)$ is a lower bound on the reduction in the value of $f$,
which we treat as the progress.
}

{\bf Notation}~
We let $k_t$ denote the coordinate selected at time $t$,
which we call the coordinate being updated at time $t$,
as in our efficient implementation, it will be the only coordinate
being updated at time $t$.
We let $\Delta z_{k_t}^{\kk,\pi}=z_{k_t}^{\kk+1, \pi} - w_{k_t}^{\kk, \pi}$ (note this is not the increment to $z_{k_{t}}^{\rjc{\kk}, \pi}$).
Also, we let $\Delta x_{k_t}^{\kk,\pi}=n \phi_t \Delta z_{k_t}^{\kk,\pi}$. \footnote{In the appendix, we also use the notation $\Delta z_{k}^{\kk, \pi}$, which is equal to $0$ if $k \neq k_t$.}
Note that the computation starts at time $t=0$,
which is the ``time'' of the first update.


%% file: Algorithm-rjc-copy.tex
\section{The Algorithm and its Performance}
\label{sec::results}

Algorithm~\ref{alg::apcg::asyn}
updates every coordinate in each iteration.
This is unnecessary and could be very inefficient.
Instead, we follow the approach taken in~\cite{LinLuXiao2015},
which we now explain.
Observe that the update rule could be written as follows.

\begin{align*}
\left[\begin{array}{ll}\var{y_k}{\kk+1, \pi}\\ \var{z_k}{\kk+1, \pi}\end{array}\right]
=
\left[\begin{array}{cc}
1 - \bbeta_{\kk}(1 - \ggamma_{\kk+1}) &  \bbeta_{\kk}(1 - \ggamma_{\kk+1})\\
1- \aalpha_{\kk} & \aalpha_{\kk}
\end{array}\right]
\left[ \begin{array}{c}\var{y_k}{\kk, \pi}\\ \var{z_k}{\kk, \pi}\end{array}\right] 
+ \left\{
   \begin{array}{cl}
     0 &\text{if $k \ne k_t$} \\ 
     \Delta z_{k_t}^{\kk,\pi} \cdot
     \left[\begin{array}{c}
                  1 - \ggamma_{\kk+1}(1 - n\aalpha_{\kk}) \\ 1
            \end{array}
     \right]
    &\text{if $k = k_t$}
   \end{array}
\right.
\end{align*}

For short, we write
 $\left[\begin{array}{c}\var{y_k}{\kk+1, \pi}\\ \var{z_k}{\kk+1, \pi}
 \end{array}\right]
 = A^t \left[\begin{array}{c} \var{y_k}{\kk, \pi} \\ \var{z_k}{\kk, \pi}\end{array}\right]$ if $k\ne k_t$ and
$\left[\begin{array}{c}\var{y_k}{\kk+1, \pi}\\ \var{z_k}{\kk+1, \pi}
\end{array}\right] = A^t \left[\begin{array}{c} \var{y_k}{\kk, \pi} \\ \var{z_k}{\kk, \pi}\end{array}\right]+\Delta z_{k_t}^{\kk,\pi}  D^t $ if $k= k_t$.

We let $B^{\kk} = A^{\kk} A^{\kk-1}\cdots A^1$.
Instead of storing the values $\var{y_k}{\kk, \pi}$ and $\var{z_k}{\kk, \pi}$,
we store the values 
$ \left(\begin{array}{c}\var{u_k}{\kk, \pi}\\ \var{v_k}{\kk, \pi}\end{array}\right)
= (\var{B}{t})^{-1}\left(\begin{array}{c}\var{y_k}{\kk, \pi}\\ \var{z_k}{\kk, \pi}\end{array}\right)$,
since for $k\ne k_t$, 
$\left(\begin{array}{c}\var{u_k}{\kk+1, \pi}\\ \var{v_k}{\kk+1, \pi}\end{array}\right)
=\left(\begin{array}{c}\var{u_k}{\kk, \pi}\\ \var{v_k}{\kk, \pi}\end{array}\right)$,
and so we need to update only $\var{u_{k_t}}{\kk, \pi}$ and $\var{v_{k_t}}{\kk, \pi}$ at time $t$.
This  leads to Algorithm~\ref{alg::apcg},
an efficient version of Algorithm~\ref{alg::apcg::asyn}.

\begin{algorithm}
\caption{An efficient implementation of APCG}
\label{alg::apcg}
Let  $\var{B}{0} =   \left(\begin{matrix} 
1 & 0 \\
0 & 1
\end{matrix} \right)$\;
Let $\var{u}{0} = \var{v}{0} = \var{x}{0} = \var{y}{0} = \var{z}{0}$\;
 \For{$\kk$ $=$ $0$ \textsf{\textbf{to}}\ $T-1$ }
{
Choose $\iit_{\rjc{\kk}} \in {1,\cdots, n}$ uniformly at random\;
$(\var{y_{k_t}}{\kk}, \var{z_{k_t}}{\kk})\tran \la \var{B}{\kk} (\var{u_{k_t}}{\kk}, \var{v_{k_t}}{\kk})\tran$\;
$\Delta \var{z}{t}_{k_t} \la \arg \min_{h} \{ \frac{\Gamma_{\kk}}{2} \|h\|^2 + \langle \nabla_{k_{\kk}} f(\var{y}{k}), h \rangle \}$\;
$\var{A}{\kk} \la \left(\begin{matrix} 
1 - \bbeta_{\kk}(1 - \ggamma_{\kk+1}) &  \bbeta_{\kk}(1 - \ggamma_{\kk+1})\\
1- \bbeta_{\kk} & \bbeta_{\kk}
\end{matrix} \right)$\;
$\var{D}{\kk} \la \left(\begin{matrix} 
 1 - \ggamma_{\kk+1}(1 - n\aalpha_{\kk}) \\ 1
\end{matrix} \right)$\;
$\var{B}{\kk+1} \la \var{A}{\kk} \var{B}{\kk}$\;
$(\var{u_{k_t}}{\kk+1}, \var{v_{k_t}}{\kk+1})\tran \la (\var{u_{k_t}}{\kk}, \var{v_{k_t}}{\kk})\tran + {\var{B}{\kk+1}}^{-1} \var{D}{\kk}  \Delta  z_{k_t}^{\kk,\pi}$\;
}
\end{algorithm}

\hide{
In this section, we describe an asynchronous version of accelerated coordinate descent. WLOG, we assume the diagonal entries $L_{\iit,\iit}$ of the Lipschitz matrix  of the function $f$ are $1$ for all $i$\footnote{For a function with  different $L_{i,i}$, we can just rescale  the coordinates.}. Additionally, we assume the strongly convex parameter is $\mu$.
}

\hide{
In particular, we are looking at a revised version of the APCG method \cite{NIPS2014_5356}. The original APCG method \cite{NIPS2014_5356} modifies all the coordinates at each round.  Lin, Lu and Zhao \cite{NIPS2014_5356} gave an  alternative method which, at each round, changes only one coordinate. In their paper, for simplicity, they only present the algorithm in the strongly convex case. Here, we generalize their algorithm to a new algorithm, Algorithm~\ref{alg::apcg}, which  can be used in both the strongly convex and the convex setting.
}

\hide{
\textbf{Notes for Algorithm~\ref{alg::apcg}}. $\alpha_{\kk}$, $\beta_{\kk}$, $\gamma_{\kk}$ are series generated by $\algmu$, and $\Gamma_{\kk}$ is a decreasing series. The way to set $\algmu$ and $\Gamma_{\kk}$ will be in Theorem~\ref{thm::apcg::final}. Also, $\ivec_{\iit}$ is a $1 \times n$ matrix whose elements equal  $0$ on coordinates other than $\iit$ and equals  $1$ on coordinate $\iit$, and $\var{D}{k}$ is a $2 \times n$ matrix which has non-zero entries only in column $\iit_{\kk}$.
}

Now, we describe an asynchronous version of Algorithm~\ref{alg::apcg} precisely. In this version there is a global counter which starts from $0$. At each time, one core makes a request to the counter, the counter returns its current value and immediately increments its value by $1$.\footnote{we interpret the value of the counter as being the time.}

The initial values of $\var{u}{0}$ and $\var{v}{0}$ are in the shared memory.  Each core iteratively performs the following tasks.

\noindent
\textbf{ Asynchronous Implementation of the loop in Algorithm \ref{alg::apcg}}
\begin{enumerate}[itemsep=0pt, font=\scriptsize\bfseries]
\item Makes a request to the counter and receives an integer $\kk$ as its rank order, or rank for short. The assigned values are successive integers.
\item Chooses a random coordinate $\iit_{\rjc{\kk}}$ uniformly\footnote{This might be a serial computation.
We discuss how to mitigate its effect if needed in
Appendix~\ref{app::counter-effect}.}.
\item Retrieves values $\tildevar{u}{\kk}$ and $\tildevar{v}{\kk}$ ̃from the shared memory.
\footnote{In many scenarios (e.g., problems involving sparse matrices), there is no need to read \underline{all} the coordinates.}
\item Calculates $\var{B}{\kk}$ and ${\var{B}{\kk}}^{-1}$.
\item Sets $(\tildevar{y}{\kk},\tildevar{z}{\kk})\tran = \var{B}{\kk} (\tildevar{u}{\kk}, \tildevar{v}{\kk})\tran$.
\item Computes $\tilde{g}^{\kk}_{\iit_{\kk}} = \nabla_{\iit_{\kk}} f(\tildevar{y}{\kk})$.
\item Computes $\Delta \var{z}{t}_{k_t} = \arg \min_{h} \{ \frac{\Gamma_{\kk} }{2} ||h||^2 + \langle \tilde{g}^{\kk}_{\iit_{\kk}}, h \rangle \}$.
\item Computes $\left  (\begin{matrix} 
D_u\\
D_v
\end{matrix} \right) =  {\var{B}{\kk + 1}}^{-1} \var{D}{\kk}\Delta z_{k_t}^{\kk,\pi}$.
\begin{multicols}{2}
\item Requests a lock on $u_{\iit_{\kk}}$.
\item Reads  the current value of $u_{\iit_{\kk}}$.
\item Sets $\var{u}{\kk+1}_{\iit_{\kk}} = \var{u}{\kk}_{\iit_{\kk}} + D_u$.
\item Writes $\var{u}{\kk+1}_{\iit_{\kk}}$ to $u_{\iit_{\kk}}$  in the shared memory.
\item Releases the lock.
\item Requests a lock on  $v_{\iit_{\kk}}$.
\item Reads the current value of $v_{\iit_{\kk}}$.
\item Sets $\var{v}{\kk+1}_{\iit_{\kk}} = \var{v}{\kk}_{\iit_{\kk}} + D_v$.
\item Writes $\var{v}{\kk+1}_{\iit_{\kk}}$ to $v_{\iit_{\kk}}$ in the shared memory.
\item Releases the lock.
\end{multicols}
\end{enumerate}
\paragraph{Results}
We show the following two theorems.
The first theorem gives a linear convergence result for strongly convex functions,
while the second result allows for a larger number of processors and achieves sublinear convergence for
both strongly convex and functions that are just convex.
\begin{theorem}\label{thm::asyn::compare}
Suppose that $f$ is a strongly convex function with convex parameter $\mu$ and dimension $n \geq 19$.
Suppose we set $\phi_t = \phi = \frac{\sqrt{3 \mu}}{\sqrt{20} n}$, $\Gamma_{\kk} = \sqrt{\frac{20}{3} \mu}$, $\varphi_t = 1 - \frac{\sqrt{3 \mu}}{\sqrt{20} n}$, and $\psi_t = \frac{1}{1 + \frac{\sqrt{3 \mu}}{\sqrt{20} n}}$
and $q \leq  \min\left\{ \frac{1}{{38}} \frac{\sqrt{n} \mu^{\frac{1}{4}}}{\Lresbar},\frac{1}{{37}} \frac{\sqrt{n}}{\Lresbar}, \frac{\sqrt{n}}{{17}}, \frac{n}{{50}}\right\}$.
Then,
\begin{align*}
\E{}{f(\var{x}{T}) - f^*} \leq \left(1 - \sqrt{\frac{3}{80}} \frac{\sqrt{\mu}}{n}\right)^{T} \left[f(\var{x}{0}) - f^* + \left(1 - \frac{\sqrt{\mu}}{ \sqrt{240}}\right) \frac{\mu}{2} \|x^* - \var{x}{0}\|^2\right].
\end{align*}
\end{theorem}
Note that here $\var{A}{t} =  \left (\begin{matrix} 
\frac{1 + \phi^2}{1+ \phi} & 1 - \frac{1 + \phi^2}{1+ \phi}\\
\phi & 1 - \phi
\end{matrix} \right)$ and $\var{B}{t} = \left (\begin{matrix} 
\frac{1 + \phi^2}{1+ \phi} & 1 - \frac{1 + \phi^2}{1+ \phi}\\
\phi & 1 - \phi
\end{matrix} \right)^{t}$.
\begin{theorem} \label{thm::apcg::final}
Suppose that $\epsilon < \frac{1}{3}$,  $n \ge 19$, $\Gamma_t = \frac{20}{3} \sqrt{n \phi_t}$,
and $q \leq \min\left\{  \frac{\sqrt{\epsilon} }{{17} \Lresbar}, \frac{\sqrt{n} }{{37} \Lresbar},  \frac{\sqrt{n}}{{17}}, \frac{n }{{50}}\right\}$. 
\begin{enumerate}
\item Suppose that $f$ is a strongly convex function with strongly convex parameter $\mu$, and we set $\phi_t = \phi = \frac{(\frac{3}{20} \mu)^{\frac{2}{3}}}{n}$, $\varphi_t = 1 -  \frac{(\frac{3}{20} \mu)^{\frac{2}{3}}}{n}$, and $\psi_t = \frac{1}{1 + \frac{(\frac{3}{20} \mu)^{\frac{2}{3}}}{n}}$. Then
\begin{align*}
&\E{}{f(\var{x}{T})  - f^*} \leq \left(1 - (1 - \epsilon) \frac{(\frac{3\mu}{20})^\frac{2}{3}}{n}\right)^{T}\left[ f(\var{x}{0})  - f^*+\frac{10}{3} \|x^* - \var{x}{0}\|^2\right].
\end{align*}
\item  While if $f$ is a convex function and we set $\phi_t = \frac{2}{2n  + t + 2}$, $\varphi_t = 1$, and $\psi_t = \frac{2n + t }{2n + t + 2}$, then
\begin{align*}
&\E{}{f(\var{x}{T})  - f^*} \leq \left(\frac{(2n) (2n + 1)}{(2n + T )(2n + T + 1)}\right)^{\frac{n (\frac{3}{4}  - \epsilon - \frac{1}{4n})}{n+1}}\left[ f(\var{x}{0})  - f^*+\frac{10}{3} \|x^* - \var{x}{0}\|^2\right].
\end{align*}
\end{enumerate} 
\end{theorem}
Here, in the strongly convex case, 
$\var{A}{t} =  \left(\begin{matrix} 
\frac{1 + \phi^2}{1+ \phi} & 1 - \frac{1 + \phi^2}{1+ \phi}\\
\phi & 1 - \phi
\end{matrix} \right)$ and $\var{B}{t} = \left(\begin{matrix} 
\frac{1 + \phi^2}{1+ \phi} & 1 - \frac{1 + \phi^2}{1+ \phi}\\
\phi & 1 - \phi
\end{matrix} \right)^{t}$; in the non-strongly convex case, $\var{A}{t} =  \left(\begin{matrix} 
\frac{t+2n + 1}{t+2n + 3} & \frac{2}{t + 2n + 3}\\
0  & 1
\end{matrix} \right)$ and $\var{B}{t} =  \left(\begin{matrix} 
\frac{(2n + 2)(2n + 1)}{(2n + t + {2})(2n + t + {1})} & 1 -\frac{(2n + 2)(2n + 1)}{(2(2n + t + {2})(2n + t + {1})} \\
0  & 1
\end{matrix} \right)$.

\begin{remark}%
\emph{
\begin{enumerate}[(i)]
\item Computing $\var{B}{t}$: In the strongly convex case, one simple observation is that it can be computed in $O(\log t)$ time. However, by assumption, each update can overlap at most \rjc{$q$} other updates. If a process  remembers its current $\var{B}{t}$, then calculating the $\var{B}{t}$ for its next update is just an $O(\log q)$ time calculation.
\item In the non-strongly convex case, $\var{B}{t}$ can be calculated in $O(1)$ time.
\item Our analysis is for an efficient asynchronous implementation, in contrast to prior work. 
\item The result in Hannah et al.~\cite{hannah2018a2bcd} is analogous to Theorem~\ref{thm::asyn::compare}, except that the constraints on $q$ replace $\Lresbar$ by $L$. 
But, by allowing for non-uniform sampling
of the coordinates,
their bound also optimizes for non-scale-free measures
$\mu$ of the strong convexity.
However, as already noted, their analysis uses the Common
Value assumption.
\footnote{\cite{hannah2018a2bcd} state that they
can remove the Common Value assumption as in their
earlier work on non-accelerated coordinate descent~\cite{Sun2017}.
However, in this earlier work, as noted in~\cite{CCT2018},
this comes at the cost of having no parallel speedup.
}
In addition, their amortization does not account for the
error magnification created by multiplying by $\var{B}{t}$
to go from an out-of-date $(\var{\tilde{u}}{t}, \var{\tilde{v}}{t})$
to the corresponding $(\var{\tilde{y}}{t}, \var{\tilde{z}}{t})$.
Rather, it appears to assume out-of-date
$(\var{\tilde{y}}{t}, \var{\tilde{z}}{t})$ values can be
read directly.
\item 
The result in Fang et al.~\cite{fang2018accelerating}
does not consider the efficient implementation of the
accelerated algorithm, 
and as noted in~\cite{hannah2018a2bcd} it does not appear to
demonstrate a parallel speedup.
\end{enumerate}
}
\end{remark}

%% file: AnalysisOutline.tex
\section{The Analysis}
\subsection{The Series of Points Analyzed}

We note that the commit time ordering of the
updates need not be the same as their start time ordering.
Nonetheless, as in~\cite{MPPRRJ2015,CCT2018},
we focus on the start time ordering as this guarantees
a uniform distribution over the coordinates at each time step.
For the purposes of our analysis, 
we suppose the updates 
are applied sequentially according to their start time ordering;
so the time $t$ update updates is treated
as if it updates the time $t-1$ variables.
These need not be the same as the values 
the asynchronous algorithm encounters, because
the algorithm encounters new values only when they commit.
Recall that the updates are the values ${\var{B}{\kk}}^{-1} \var{D}{\kk}\Delta \var{z}{t, \pi}_{k_t} $.

\hide{
In our analysis, we focus on a particular series of points.
We order the updates, which are the values ${\var{B}{\kk}}^{-1} \var{D}{\kk}$, according to their starting times (rather than their commit times). \yxt{Then, at time $t$, we look at the points which accumulates first $t$ updates. 
}
}

The precise definition follows. We first define  $\hatvar{y}{\kk, \pi}$ and $\hatvar{z}{\kk, \pi}$ to be:
\begin{align*} \label{apcg::async::reorder}
(\var{\hat{y}}{\kk, \pi}, \var{\hat{z}}{\kk, \pi}) \tran = \var{B}{\kk} \left[(\var{u}{0}, \var{v}{0}) + \sum_{l = 0}^{\kk - 1} {\var{B}{l+1}}^{-1}  \var{D}{l}\Delta \var{z}{t, \pi}_{k_t} \ivec_{\iit_{\kk}} )\right]  \tran, \numberthis
\end{align*} where
$\var{D}{\kk} = \Bigg  (\begin{matrix} 
[n \psi_{t+1} \phi_t + (1 - \psi_{t+1})]  \\
{1}
\end{matrix} \Bigg)$,
and $\ivec_{\iit_{\kk}}$ is the vector with one non-zero unit
entry at coordinate $\kk$.
Note that $\left[(\var{u}{0}, \var{v}{0})  + \sum_{l = 0}^{\kk - 1} {\var{B}{l + 1}}^{-1}  \var{D}{l}\Delta \var{z}{t, \pi}_{k_t} \ivec_{\iit_{\kk}} \right]\tran$ may not appear in the memory at any time in the asynchronous computation, which is why we use the notation $\hat{y}$, $\hat{z}$. The key exception is that after the final update, at time $T$, the term $\left[(\var{u}{0}, \var{v}{0})  + \sum_{l = 0}^{T - 1} {\var{B}{l + 1}}^{-1}  \var{D}{l}\Delta \var{z}{t, \pi}_{k_t} \ivec_{\iit_{\kk}} )\right]  \tran$ will be equal to the final  $(u, v) \tran$ in the shared memory.

\begin{align*}
&\text{In addition, we define $\hatvar{x}{\kk, \pi}$ to satisfy }~~&
\var{\hat{y}}{\kk, \pi} &= (1 - \psi_t) \var{\hat{z}}{\kk, \pi} + \psi_t \var{\hat{x}}{\kk, \pi}, \hspace*{4in}
\\
&\text{and $\hatvar{w}{\kk, \pi}$ as }& \hatvar{w}{\kk, \pi} &= \varphi_t \hatvar{z}{\kk, \pi} + (1 - \varphi_t) \hatvar{y}{\kk, \pi}.\hspace*{4in}
\end{align*}

Conveniently, as we shall see, for any fixed path $\pi$, the asynchronous version is equivalent to  Algorithm~\ref{alg::apcg::asyn}.
This will allow us to carry out much of the analysis
w.r.t.\ the simpler Algorithm~\ref{alg::apcg::asyn} rather
than the asynchronous version of Algorithm~\ref{alg::apcg}.
The only place we need to work with the latter algorithm
is in bounding the differences
$\tilde{g}^{\kk, \pi}_{\iit_{\kk}} - \nabla f(y^{\kk, \pi}_{\iit_{\kk}})$, the differences between
the computed gradients and the ``correct'' values.

\hide{
\begin{algorithm}
\caption{Simple Algorithm for Analysis}
\label{alg::apcg::asyn}
	Choose ${\iit}_{\kk} \in \{1, 2, \cdots, n\}$ uniformly at random\;
	$\var{y}{\kk, \pi} = \ggamma_{\kk} \var{x}{\kk, \pi} + (1 - \ggamma_{\kk}) \var{z}{\kk, \pi}$\;
	$\var{w}{\kk, \pi} = \bbeta_{\kk} \var{z}{\kk, \pi} + (1 - \bbeta_{\kk}) \var{y}{\kk, \pi}$\;
	$\var{z}{\kk+1, \pi} = \arg\min_x \{ \frac{\Gamma_{\kk}}{2} \|x - \var{w}{\kk, \pi}\|^2 + \dotprod{\tilde{g}^{\kk, \pi}_{\iit_{\kk}}}{x_{\iit_{\kk}}}\}$\;
	$\var{x}{\kk+1, \pi} = \var{y}{\kk, \pi} + n \aalpha_{\kk} (\var{z}{\kk+1, \pi} - \var{w}{\kk, \pi})$\; 
\end{algorithm}
\RJC{The above algorithm seems to be off by one on values of $t$ compared to Algorithm 1; e.g.\ it computes $\gamma_{t+1}$ where Algorithm 1 computes $\gamma_{t+2}$.}
We note that Algorithm~\ref{alg::apcg::asyn} is not efficiently implementable in general. However, it is quite useful for the analysis.}
The following theorem says that the $\var{x}{\kk, \pi}$, $\var{y}{\kk, \pi}$, $\var{z}{\kk, \pi}$ and $\var{w}{\kk, \pi}$ in Algorithm~\ref{alg::apcg::asyn}  are equal to $\var{\hat{x}}{\kk, \pi}$, $\var{\hat{y}}{\kk, \pi}$, $\var{\hat{z}}{\kk, \pi}$ and $\hatvar{w}{\kk, \pi}$.
\begin{theorem} \label{alg::apcg::equl} 
For a given path $\pi$, the values $\{\var{x}{\kk, \pi},\var{y}{\kk, \pi},\var{z}{\kk, \pi}, \var{w}{\kk, \pi}\}_{\kk = 0,\cdots, T} $ in Algorithm~\ref{alg::apcg::asyn} are equal to the values $\{\hatvar{{x}}{\kk, \pi},\hatvar{{y}}{\kk, \pi},\hatvar{z}{\kk, \pi}, \hatvar{w}{\kk, \pi}\}_{\kk = 0,\cdots, T} $ in the asynchronous version of Algorithm~\ref{alg::apcg} if each value $\{\tilde{g}^{\kk, \pi}_{\iit_{\kk}}\}_{\kk = 0,\cdots, T}$,  and the starting points are the same in both algorithms.
\end{theorem}

\hide{
\yxt{Note that our Theorem \ref{alg::apcg::equl} allow us to express our lemma more clearly, while it doesn't say that we can do the analysis only on Algorithm~\ref{alg::apcg}. The series $\{\tilde{g}^{\kk, \pi}_{\iit_{\kk}}\}_{\kk = 0,\cdots, T}$ is defined in asynchronous version of Algorithm~\ref{alg::apcg}. In order to analyze its difference with the right gradient (as we will see in the Lemma~\ref{lem::apcg::error::non::str}), we need to go back to asynchronous version of Algorithm~\ref{alg::apcg}.}
}
\begin{center}
\begin{tabular}{ c p{3.5in} }
     \raisebox{-\totalheight}{\includegraphics[scale=0.3]{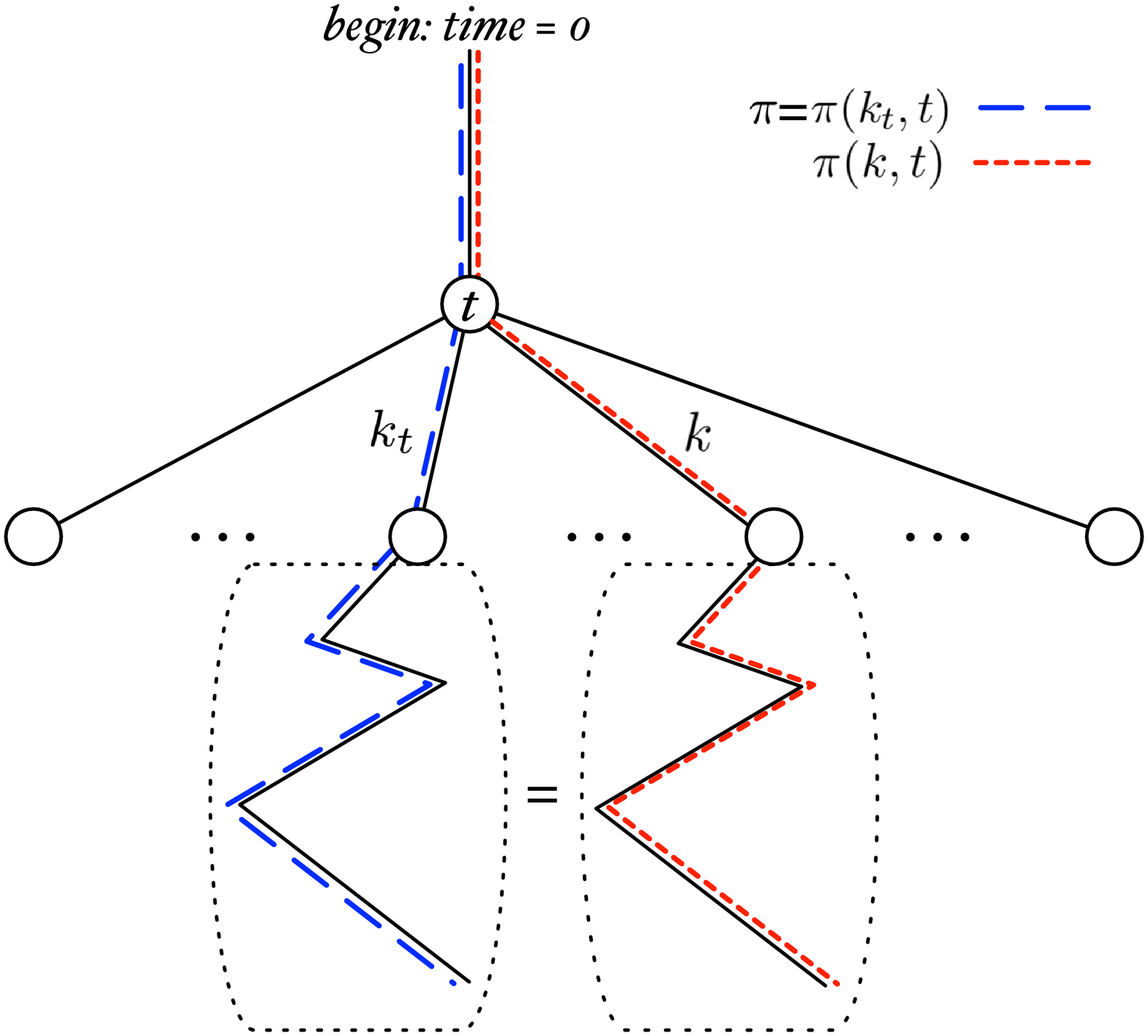}} &  \parbox[t]{3.5in}{\hspace*{0.2in}In the following analysis, with a slight abuse of notation, we let $\{\var{x}{\kk, \pi},\var{y}{\kk, \pi},\var{z}{\kk, \pi}, \var{w}{\kk, \pi}\}_{\kk = 0,\cdots, T}$ denote both $\{\var{x}{\kk, \pi},\var{y}{\kk, \pi},\var{z}{\kk, \pi}, \var{w}{\kk, \pi}\}_{\kk = 0,\cdots, T} $ in Algorithm~\ref{alg::apcg::asyn} and  $\{\var{\hat{x}}{\kk, \pi},\var{\hat{y}}{\kk, \pi},\var{\hat{z}}{\kk, \pi}, \hatvar{w}{\kk, \pi}\}_{\kk = 0,\cdots, T} $. 
 
\hspace*{0.2in}Given a path $\pi$, we will be considering alternate paths in which edge $\kt$ is changed to $k$ but every other
choice of variable is unchanged; $\pi(k, t)$ denotes this alternate path. 
Then, $\{\var{x}{\kk, \pi(k,t)},\var{y}{\kk, \pi(k,t)},\var{z}{\kk, \pi(k,t)}, \var{w}{\kk, \pi(k,t)}\}$ are the instances
of variables $\{\var{x}{\kk},\var{y}{\kk},\var{z}{\kk}, \var{w}{\kk}\}$ on path $\pi(k, t)$ at time $t$. Remember that $k_t$ denotes the coordinate selected at time $t$, which implies $\pi(\kt,t) = \pi$.  The relationship between $\pi(\kt,t) = \pi$ and $\pi(k,t)$ is shown in the figure on the left.}   \\ 
\end{tabular}
\end{center}

\subsection{Starting Point: The Progress Lemma (see  Appendix~\ref{app::progress::lemma} for the full version)} 
The starting point for the bounds in Theorem~\ref{thm::asyn::compare} and \ref{thm::apcg::final}
is the following lemma.

Let $0 < \parb < 1$, $\zeta_{\kk}$, $\mathcal{P}_t$, $\mathcal{Q}_t$, $\mathcal{R}_t$, $\mathcal{S}_t$, and $\mathcal{T}_t$ denote some parameters. Then we show the following lemma.
\renewcommand{\arraystretch}{1.5}
\begin{lemma}[Informal Progress Lemma] \label{lem::apcg::progress::informal}
Define the potential function
$\var{\mathcal{F}}{\kk} = f(\var{x}{\kk})  - f^*+\zeta_{\kk} \|x^* - \var{z}{\kk}\|^2$. Then,
\begin{align*} 
\mathbb{E}_{\pi} \left[\var{\mathcal{F}}{\kk + 1}\right] &\leq  (1 - \tau \phi_t)  \mathbb{E}_{\pi}\left[\var{\mathcal{F}}{\kk} \right] ~-~ \Adj_{\kk},\hspace*{2in}
\end{align*}
\begin{align*}
\text{where}\hspace*{0.25in}\Adj_{\kk} ~= &~ \left. \mathbb{E}_{\pi} \left[ \mathcal{P}_t \left(\Delta z_{\iit_t}^{t, \pi}  \right)^2\right]   \largespace \largespace \largespace \largespace \hspace*{0.25in} \right\}  
\mbox{progress term} \\
&\left.\begin{array}{ll}
&\smallspace \smallspace - \mathbb{E}_{\pi} \Bigg[ \sum_{ k'} \Bigg( \mathcal{Q}_t\left(\grad{\iit'}{f(\var{y}{\kk, \pi})} - \tilde{g}^{t, \pi(k', t)}_{\iit'}\right)^2  \\
&\largespace \largespace +  \mathcal{R}_t\left(\grad{\iit'}{f(\var{y}{\kk, \pi})} - \grad{\iit'}{f(\var{y}{\kk, \pi(k', t)})}\right)^2 \Bigg) \Bigg] \\
&\smallspace \smallspace -\mathbb{E}_{\pi} \left[ \mathcal{S}_t\sum_{k'} \|\var{w}{\kk, \pi}_{k'} - \var{w}{\kk, \pi(k', t)}_{k'}\|^2 \right] \\
&\smallspace \smallspace -\mathbb{E}_{\pi} \left[ \mathcal{T}_t \sum_{ k'} \| \var{z}{\kk, \pi}_{k'} - \var{z}{\kk, \pi(k', t)}_{k'}\|^2 \right].
\end{array} \right\} \mbox{error terms}
\end{align*}

\end{lemma}\renewcommand{\arraystretch}{1}

\begin{remark}
\label{rem::prog-lem}
\emph{
\begin{enumerate}[(i)]
\item The sequential analysis in \cite{LinLuXiao2015,allen2016even}
considers a single randomized update based on the same $x^t$, $y^t$ and $z^t$
and obtains the bound
$\mathbb{E}_{\pi} \left[\var{\mathcal{F}}{\kk + 1}\right] \leq  (1 - \phi_t)  \mathbb{E}_{\pi}\left[\var{\mathcal{F}}{\kk} \right]$.
In our asynchronous version, since we drop the common value assumption, $x^t$, $y^t$ and $z^t$ may be different when different choices of coordinates to update are made. 
This makes the analysis quite challenging.
\item Our analysis extracts progress and error terms. Our goal is to use the progress term
to consume the error terms.
We obtain the progress term by damping the reduction in the
potential function obtained in the sequential analysis
by a factor $\tau$.
\hide{
In order to do that, instead of getting $\left[\var{\mathcal{F}}{\kk + 1}\right] \leq  (1 -  \phi_t)  \mathbb{E}_{\pi}\left[\var{\mathcal{F}}{\kk} \right]$ in sequential analysis, we damped the improvement of the potential function $\var{\mathcal{F}}{\kk}$ by a factor $\tau$.
}
\end{enumerate}
}
\end{remark}

\subsection{Amortization between Progress terms and Error terms in Progress Lemma}
In order to prove the convergence result $\E{}{\var{\mathcal{F}}{T }} \leq \Big(\prod_{\kk = 0 \cdots T-1} (1 - \parb \phi_{\kk})\Big) \var{\mathcal{F}}{0}$, it suffices to show that $\sum_t \prod_{l = t+1}^{T-1} (1 - \parb \phi_{l}) \Adj_t \ge 0$. 

$\Adj_t$ consists of the progress term, $\left(\Delta z_{\iit_t}^{t, \pi}  \right)^2$, and the error terms, $\left(\grad{\iit'}{f(\var{y}{\kk, \pi})} - \tilde{g}^{t, \pi(k', t)}_{\iit'}\right)^2$, $\left(\grad{\iit'}{f(\var{y}{\kk, \pi})} - \grad{\iit'}{f(\var{y}{\kk, \pi(k', t)})}\right)^2 $, $\|\var{w}{\kk, \pi}_{k'} - \var{w}{\kk, \pi(k', t)}_{k'}\|^2$ and $\| \var{z}{\kk, \pi}_{k'} - \var{z}{\kk, \pi(k', t)}_{k'}\|^2$. It's hard to suitably bound the error terms by the progress term.
To do this, in the spirit of \cite{CCT2018}, we introduce the new terms $(\Delta_{s}^{\tFE})^2$ and $\E{\pi}{(g_{\max, \iit_{\kk}}^{\pi, \kk} - g_{\min, \iit_{\kk}}^{\pi, \kk})^2}$ as a bridge, to connect the progress and error terms. 

Roughly speaking, $(\Delta_{t}^{\tFE})^2$ is the expectation, over all paths $\pi$, of the difference between the maximal and minimal possible updates at time $t$ on path $\pi$, and  $\E{\pi}{(g_{\max, \iit_{\kk}}^{\pi, \kk} - g_{\min, \iit_{\kk}}^{\pi, \kk})^2}$ is the expectation of the difference between the maximal and minimal possible gradients at time $t$. For more precise definitions, please see Appendix~\ref{Appendix::C}. For simplicity, let $E_t^{\Delta}$ denote the expected value of $\left(\Delta z_{\iit_t}^{t, \pi}  \right)^2$ at time $t$. We also suppose the $\var{B}{t}$ are \emph{good}, which roughly speaking means that $\var{A}{t}$ is close to the identity matrix (the precise definition can be found in Definition \ref{defn::good} in Appendix~\ref{Appendix::D}).  We show in Lemma~\ref{lem::B::good} that the $\var{B}{t}$ are good for the  choices of parameters in Theorems~\ref{thm::asyn::compare} and \ref{thm::apcg::final}.
We show the following bounds on 
$(\Delta_{\kk}^{\tFE})^2$,
$\E{\pi}{(g_{\max, \iit_{\kk}}^{\pi, \kk} - g_{\min, \iit_{\kk}}^{\pi, \kk})^2}$,
and the error terms. 
\hide{
\begin{defn}
We say $\var{B}{t}$ are good if the absolute value of the first elements of these four matrix:
\begin{align*}
&\var{B}{t} {\var{B}{s+1}}^{-1} \Bigg  (\begin{matrix} 
[n \psi_{s+1} \phi_s + (1 - \psi_{s+1})] \\
1
\end{matrix} \Bigg), &\var{B}{t} {\var{B}{s+1}}^{-1} \Bigg  (\begin{matrix} 
[n \psi_{s+1} \phi_s + (1 - \psi_{s+1})] \\
0
\end{matrix} \Bigg),  \\
&\var{B}{t} {\var{B}{s+1}}^{-1} \Bigg  (\begin{matrix} 
0 \\
1
\end{matrix} \Bigg), &~~~~\mbox{and}~~~~~~~~~~~\var{B}{t} {\var{B}{s+1}}^{-1} \Bigg  (\begin{matrix} 
0 \\
0
\end{matrix} \Bigg) .
\end{align*}
are smaller than $\frac{3}{2} n \phi_t$ and the second element of \begin{align*}
&\var{B}{t} {\var{B}{s+1}}^{-1} \Bigg  (\begin{matrix} 
[n \psi_{s+1} \phi_s + (1 - \psi_{s+1})] \\
1
\end{matrix} \Bigg)
\end{align*} is smaller than $2$,  for any $s$ and $t$ such that $|s - t| \leq 2q$.
\end{defn}}
\begin{lemma}[Informal Amortization Lemma; full version in Appendix~\ref{Appendix::C}]
\label{lem:FEandDeltaBounds:informal}
Let $I=[0,T-1]$. If the $\var{B}{t}$ are good, then:
\begin{align*}
(\Delta_{\kk}^{\tFE})^2 &\leq \Theta\left(\frac{1}{\Gamma_t^2}\right) \E{\pi}{(g_{\max, \iit_{\kk}}^{\pi, \kk} - g_{\min, \iit_{\kk}}^{\pi, \kk})^2}; \numberthis \label{ineq::tran::1::informal} \\
\E{\pi}{(g_{\max, \iit_{\kk}}^{\pi, \kk} - g_{\min, \iit_{\kk}}^{\pi, \kk})^2} &\leq \Theta\left( q \Lresbar^2 n \phi_{\kk}^2\right) \sum_{s \in I \cap [\kk - 2q, \kk + 2q] \setminus \{\kk\}} \Big[ (\Delta_{s}^{\tFE})^2+ E_s^{\Delta}\Big]; \numberthis \label{ineq::tran::2::informal}\\
\E{\pi}{\sum_{k'} \left(\var{w}{\kk, \pi}_{k'} - \var{w}{\kk, \pi(k', t)}_{k'}\right)^2} &\leq  \Theta(q) \sum_{s \in I \cap [\kk - q - 1, \kk - 1]} (\Delta_{s}^{\tFE})^2; \numberthis \label{ineq::tran::3::informal}\\
\E{\pi}{\sum_{k'} \left(\var{z}{\kk, \pi}_{k'} - \var{z}{\kk, \pi(k', t)}_{k'}\right)^2} &\leq  \Theta(q) \sum_{s \in I \cap [\kk - q - 1, \kk - 1]} (\Delta_{s}^{\tFE})^2; \numberthis \label{ineq::tran::4::informal}\\
\E{\pi}{\sum_{k'} \left(\grad{\iit'}{f(\var{y}{\kk, \pi})} - \tilde{g}^{t, \pi(k', t)}_{\iit'}\right)^2} & \leq  \Theta(n \cdot n^2 \phi_t^2 q \Lresbar^2) \sum_{s \in I \cap [t - 3q, t +q ]}\left((\Delta_{s}^{\tFE})^2 + E^{\Delta}_{s}\right)\\
&\largespace + \Theta(n \cdot n^2 \phi_t^2) \left((\Delta_{t}^{\tFE})^2 + E^{\Delta}_{t}\right)  \\
&\largespace+ \Theta(n) \E{\pi}{(g_{\max, \iit_{\kk}}^{\pi, \kk} - g_{\min, \iit_{\kk}}^{\pi, \kk})^2} \numberthis \label{ineq::tran::5::informal} \\
\E{\pi}{\sum_{k'} \left(\grad{k'}{f(\var{y}{\kk, \pi})} - \grad{k'}{f(\var{y}{\kk, \pi(k', t)})}\right)^2} &\leq \Theta(n \cdot n^2 \phi_t^2 q \Lresbar^2) \sum_{s \in I \cap [t - 3q, t +q ]}\left((\Delta_{s}^{\tFE})^2 + E^{\Delta}_{s}\right)\\
&\largespace + \Theta(n \cdot n^2 \phi_t^2) \left((\Delta_{t}^{\tFE})^2 + E^{\Delta}_{t}\right) . \numberthis\label{ineq::tran::6::informal}
\end{align*}
\end{lemma}

Using \eqref{ineq::tran::1::informal} and \eqref{ineq::tran::2::informal}, we obtain the following lemma, which bounds the sum of the series of $(\Delta_{s}^{\tFE})^2$ by $E_s^{\Delta}$.

\begin{lemma} \label{lem::delta::act}
Let $\{a_t\}$ be a series of non-negative numbers, let $\Xi =  \max_{\kk} \frac{216  q^2 \Lresbar^2 n^2 \phi_{\kk}^2}{n (\Gamma_{\kk})^2}$,
and let
\begin{align*}
\Phi_a &= \min_{\kk \in [0, \cdots, T-1]} \min_{s \in [\kk - 2q, \kk + 2q] \cap [0, T-1]} \left\{\frac{a_s \Big(\prod_{l = s + 1}^{T-1 }  (1 - \parb \phi_l)\Big)}{a_t \Big(\prod_{l = \kk + 1}^{T-1}  (1 - \parb \phi_l)\Big)}\right\}.
\end{align*}
If the $\var{B}{t}$ are good, then
\begin{align*}
\sum_t a_t \Big(\prod_{l = \kk + 1}^{T-1}  (1 - \parb \phi_l)\Big)(\Delta_{\kk}^{\tFE})^2 \leq   \frac{\Xi}{\Phi_a - \Xi} \sum_{t} a_t \Big(\prod_{l = \kk + 1}^{T-1}  (1 - \parb \phi_l)\Big) E_t^{\Delta}.
\hspace*{1in}
\end{align*}
\end{lemma}

Using Lemma~\ref{lem:FEandDeltaBounds:informal} and~\ref{lem::delta::act}, we can bound the error terms by the progress term by using the bridges $(\Delta_{s}^{\tFE})^2$ and $\E{\pi}{(g_{\max, \iit_{\kk}}^{\pi, \kk} - g_{\min, \iit_{\kk}}^{\pi, \kk})^2}$. By choosing the parameters carefully, we can deduce Theorems~\ref{thm::asyn::compare} and~\ref{thm::apcg::final}.

\subsection{Note Regarding the Appendix}
In Appendix~\ref{app:proofs-of-thms}, we give a more general theorem, Theorem~\ref{thm::overall}, that subsumes
Theorems~\ref{thm::asyn::compare} and~\ref{thm::apcg::final}. In Appendix~\ref{app::proofs-of-thms}, we show that Theorems~\ref{thm::asyn::compare} and~\ref{thm::apcg::final} follow by carefully choosing the parameters in Theorem~\ref{thm::overall}. In order to obtain this more general theorem, as in the main part, we demonstrate the full version of the Progress Lemma in Appendix~\ref{app::progress::lemma} and then show the full version of the Amortization Lemma in Appendix~\ref{Appendix::C}. Finally, in Appendix~\ref{app:proofs-of-thms}, we give the proof of this general theorem.

%% file: AppThmProof.tex
\section{Proof of the general theorem: Theorem~\ref{thm::overall}
}
\label{app:proofs-of-thms}

We state and prove the general theorem: Theorem~\ref{thm::overall}.

\begin{theorem} \label{thm::overall}
Suppose that $0 < \tau \leq 1$, $0 < \tilde{\tau} \leq 1$, $\phi_t $, $\varphi_t$, $\psi_t$, $\Gamma_t$, $q$, $r  = \max_t \left\{ \frac{36 (3q)^2 L_{\overline{res}}^2 n^2 \xi \phi_{t}^2}{ \Gamma_{t}^2 n}\right\} \leq \frac{1}{{32}}$,  $n$ and $\zeta_{\kk + 1}$ satisfy the following conditions.
\begin{itemize}
\item Strongly convex case: $\phi_t = \phi$, $\varphi_t = 1 - \phi$, $\psi = \frac{1}{1 + \phi}$, $\Gamma_t = \Gamma$, and $\zeta_{\kk + 1} = \frac{n \phi \Gamma}{2}\left(1 - \frac{n \phi (1 - \tau)}{3} \right)$.
\item Non-strongly convex case: $\phi_t = \frac{2}{t + t_0}$ for some $t_0 \geq 2(n+1)$, $\varphi_t = 1$, $\psi = 1 - \phi_t$, and $\zeta_{\kk + 1} = \frac{n \phi_t \Gamma_t}{2} \left( 1 - \frac{n \phi_t \left( 1 - \tilde{\tau} - \frac{1}{4t_0}\right)}{3} \right)$.
\end{itemize} 
Let
\begin{align*}
\xi & =  \max_t \max_{s \in [ t - 3q, t + {2q}]} \frac{\phi_{s}^2 \Gamma_{t}^2}{\phi_{t}^2 \Gamma_{s}^2} \leq \frac{6}{5};\\
\Xi &=  \max_{\kk} \frac{216 n^2 \phi_{\kk}^2 q^2 \Lresbar^2}{n(\Gamma_{\kk})^2 }; \\
\Phi_b &=  \min_{\kk \in [1, \cdots, T]} \min_{s \in [\kk - 4q, \kk + 4q]} \left\{\frac{\frac{\frac{3}{ \Upsilon} + 2 n \phi_s }{2  \Gamma_s}  \frac{540 q^2 \Lresbar^2 n^2 \phi_s^2}{n} \Big(\prod_{l = s + 1}^{T}  (1 - \parb \phi_l)\Big)}{\frac{\frac{3}{ \Upsilon} + 2 n \phi_t }{2  \Gamma_t}  \frac{540 q^2 \Lresbar^2 n^2 \phi_t^2}{n} \Big(\prod_{l = \kk + 1}^{T}  (1 - \parb \phi_l)\Big)}\right\}; \\
\Phi_c &=\min_{\kk \in [1, \cdots, T]} \min_{s \in [\kk - 2q, \kk + 2q]} \left\{\frac{12 \Gamma_s \phi_s \Big(\prod_{l = s + 1}^{T}  (1 - \parb \phi_l)\Big)}{12 \Gamma_t \phi_t \Big(\prod_{l = \kk + 1}^{T}  (1 - \parb \phi_l)\Big)}\right\}.
\end{align*}
in both cases. Suppose $\frac{(q)^2 L_{\overline{res}}^2 }{n} \leq 1$, the $\var{B}{t}$ are good, and the following constraints hold:

\smallskip
\noindent
\emph{
Strongly convex case:}
\begin{enumerate}[(i)]
\begin{multicols}{2}
\item$\phi_t \leq \frac{1}{n+1}$; \label{lab::enu::str::1}
\item$n \phi_t \Gamma_t \leq \mu$;\label{lab::enu::str::2}
\item$\frac{4}{5} \Gamma_t \geq n \phi_t$; \emph{(subsumed by~\ref{lab::enu::str::7})} \label{lab::enu::str::3}
\item $\frac{1}{(\Phi_b - \Xi)} \frac{3}{ 2 \Gamma_t \Upsilon}\frac{{({896} +  {\frac{397}{1-r}}) \xi}   q^2 \Lresbar^2 n^2 \phi_t^2}{n} \leq \frac{1}{5} n \phi_t \Gamma_t$; \label{lab::enu::str::4}
\item$\frac{1}{(\Phi_b - \Xi)} \frac{2 n \phi_t}{ \Gamma_t} \frac{{({896} +  {\frac{397}{1-r}}) \xi}   q^2 \Lresbar^2 n^2 \phi_t^2}{n} \leq \frac{3}{50} n \phi_t \Gamma_t$; \label{lab::enu::str::5}
\item$\frac{\Xi}{\Phi_c(\Phi_c - \Xi)}112 q^2 \Gamma_t \phi_t \leq \frac{1}{50} n \phi_t \Gamma_t$; \label{lab::enu::str::6}
\item$\frac{3}{20} \Gamma_t \geq n \phi_t$;\label{lab::enu::str::7}
\item$q \leq \min\left\{\frac{n-8}{12}, \frac{2n - 4}{10}, \frac{n}{20}\right\}$;\label{lab::enu::str::8}
\item $q \le \frac {n}{25}$;\label{lab::enu::str::9}
\item $n \ge 19$. \label{lab::enu::str::10}
\end{multicols}
\end{enumerate}

\noindent
\emph{
Non-strongly convex case:}
\begin{enumerate}[(i)]
\begin{multicols}{2}
\item$\tilde{\tau} \geq \frac{1}{2}$; \label{lab::enu::non::1}
\item$(1 - \tilde{\tau}\phi_t) \frac{\Gamma_t}{n \phi_t} \leq (1 - \tau \phi_t) \frac{\Gamma_{t-1}}{n \phi_{t-1}}$; \label{lab::enu::non::2}
\item$\frac{4}{5} \Gamma_t \geq n \phi_t$; \emph{(subsumed by~\ref{lab::enu::non::7})}  \label{lab::enu::non::3}
\item $\frac{1}{(\Phi_b - \Xi)} \frac{3}{ 2 \Gamma_t \Upsilon}\frac{{({896} +  {\frac{397}{1-r}}) \xi}   q^2 \Lresbar^2 n^2 \phi_t^2}{n} \leq \frac{1}{5} n \phi_t \Gamma_t$; \label{lab::enu::non::4}
\item$\frac{1}{(\Phi_b - \Xi)} \frac{2 n \phi_t}{ \Gamma_t} \frac{{({896} +  {\frac{397}{1-r}}) \xi}   q^2 \Lresbar^2 n^2 \phi_t^2}{n} \leq \frac{3}{50} n \phi_t \Gamma_t$; \label{lab::enu::non::5}
\item$
\frac{\Xi}{\Phi_c(\Phi_c - \Xi)} 112 q^2 \Gamma_t \phi_t \leq \frac{1}{50} n \phi_t \Gamma_t$; \label{lab::enu::non::6}
\item$\frac{3}{20} \Gamma_t \geq n \phi_t$;\label{lab::enu::non::7}
\item$q \leq \min\left\{\frac{n-8}{12},\frac{2n - 4}{10}, \frac{n}{20} \right\}$;\label{lab::enu::non::8}
\item $q \le \frac {n}{25}$;\label{lab::enu::non::9}
\item $n \ge 19$. \label{lab::enu::non::10}
\end{multicols}
\end{enumerate}
\emph{
Then,
\begin{align*}
&
\E{}{f(\var{x}{T})  - f^*+\zeta_{T+1} \|x^* - \var{z}{T}\|^2} 
\leq \Big(\prod_{t = 0 \cdots T-1} (1 - \parb \phi_{\kk})\Big)\Bigg[ f(\var{x}{0})  - f^*+\zeta_0 \|x^* - \var{z}{0}\|^2\Bigg].
\end{align*}\label{thm::asyn::apcg}
}
\end{theorem}

\begin{proof}Let $\Upsilon = 1 - \tau$ in the strongly convex case and  $\Upsilon = 1 - \tilde{\tau} - \frac{1}{4t_0}$ in the non-strongly convex case. 
By Lemma~\ref{lem::apcg::error::non::str},
which requires constraints (i)--(iii),
it suffices to show $\sum_t \prod_{l = t+1}^{T} (1 - \parb \phi_{l}) \Adj_t \ge 0$.
Our proof will also apply Lemmas~\ref{lem:FEandDeltaBounds}, \ref{lem::delta::act} and \ref{lem::B::good}, which assume constraint~\ref{lab::enu::str::8}.

Note that $r \leq \frac{{73} q^2 L_{\overline{res}}^2 \xi}{{10}n}$ as $\frac{3}{20} \Gamma_t \geq n \phi_t$. By  \eqref{ineq::tran::2}, \eqref{ineq::tran::5}, and \eqref{ineq::tran::6}
from Lemma~\ref{lem:FEandDeltaBounds} for the first inequality, and by Lemma~\ref{lem::delta::act} for the third inequality,
\begin{align*}
&\sum_{t} \Big(\prod_{l = t+1 \cdots T - 1} (1 - \parb \alpha_{l})\Big) \cdot \Bigg[\mathbb{E}_{\pi} \Bigg[ \frac{1}{n} \sum_{ k'} \Bigg(\frac{\frac{3}{\Upsilon} + 2 n \phi_t }{2 \Gamma_t}\left(\grad{\iit'}{f(\var{y}{\kk, \pi})} - \tilde{g}^{t, \pi(k', t)}_{\iit'}\right)^2  \\
&\largespace \largespace+  \frac{10 n \phi_t}{2 \Gamma_t} \left(\grad{\iit'}{f(\var{y}{\kk, \pi})} - \grad{\iit'}{f(\var{y}{\kk, \pi(k', t)})}\right)^2 \Bigg) \Bigg] \Bigg] \\
&\smallspace \leq \sum_t \Big(\prod_{l = t+1 \cdots T-1} (1 - \parb \phi_{l})\Big) \cdot \Bigg[ \frac{1}{n} \frac{\frac{3}{ \Upsilon } + 2 n \phi_t }{2 \Gamma_t} \cdot \\
&\largespace\Big[{\frac{63}{2}} n^2 \phi_t^2 q L_{\overline{res}}^2 \sum_{s \in I \cap [t - q, t - 1]} (\Delta_{s}^{\tFE})^2  + {\frac{54r}{1-r}} n  n^2 \phi_t^2   (\Delta_{t}^{\tFE})^2 + {\frac{36r}{1-r}} n n^2 \phi_t^2   E^{\Delta}_t  \\
&\largespace + \frac{r}{12q(1-r)} n n^2 \phi_t^2  \sum_{s \in  I \cap [t - 3q , t +q] \setminus \{t\}} \left((\Delta_{s}^{\tFE})^2 + E^{\Delta}_s \right) \\
&\largespace+  n  \frac{216  q \Lresbar^2 n^2 \phi_{\kk}^2}{n} \sum_{s \in I \cap [\kk - 2q, \kk + 2q] \setminus \{\kk\}} \Big[ (\Delta_{s}^{\tFE})^2+ E_s^{\Delta}\Big]  \Big]\Bigg] \\
&\smallspace \leq \sum_t \Big(\prod_{l = t+1 \cdots T-1} (1 - \parb \phi_{l})\Big) \cdot \Bigg[ \frac {1}{\Phi_b} \cdot \frac{\frac{3}{ \Upsilon} + 2 n \phi_t }{2  \Gamma_t} \cdot  \frac{({896} + {\frac{397}{1-r}} \xi) q^2 \Lresbar^2 n^2 \phi_t^2}{n} \left[(\Delta_{t}^{\tFE})^2 + E_t^{\Delta}\right] \Bigg] \\
 &\largespace \text{(using the definition of $\Phi_b$ and noting that there are at most $q$ occurences of each $(\Delta_s^\tFE)^2$}\\
&\smallspace \leq \sum_t \frac{1}{\Phi_b} \left[ \frac{\Xi}{(\Phi_b - \Xi)} + 1 \right] \Big(\prod_{l = t+1 \cdots T-1} (1 - \parb \phi_{l})\Big)  \frac{\frac{3}{\Upsilon} + 2 n \phi_t }{2 \Gamma_t} \cdot \frac{({896} +  {\frac{397}{1-r}} \xi) q^2 \Lresbar^2 n^2 \phi_t^2}{n} E_t^{\Delta} \\
& \largespace \text{by Lemma~\ref{lem::delta::act} applied to the series
$b_t = \frac{\frac{3}{\Upsilon} + 2 n \phi_t }{2 \Gamma_t} \cdot \frac{ ({896} +  {\frac{397}{1-r}} \xi) q^2 \Lresbar^2 n^2 \phi_t^2}{n}$} \\
&\smallspace \leq \sum_t \frac{1}{(\Phi_b - \Xi)}  \Big(\prod_{l = t+1 \cdots T-1} (1 - \parb \phi_{l})\Big)  \frac{\frac{3}{ \Upsilon } + 2 n \phi_t }{2 \Gamma_t} \cdot \frac{({896} +  {\frac{397}{1-r}} \xi) q^2 \Lresbar^2 n^2 \phi_t^2}{n} E_t^{\Delta}.
\end{align*}


Similarly, by \eqref{ineq::tran::3}, \eqref{ineq::tran::4}, and Lemma~\ref{lem::delta::act},
\begin{align*}
&\sum_{t} \Big(\prod_{l = t+1 \cdots T-1} (1 - \parb \phi_{l})\Big) \cdot \Bigg[\mathbb{E}_{\pi} \left[ 6 \cdot \Gamma_t \phi_t \sum_{k'} \|\var{w}{\kk, \pi}_{k'} - \var{w}{\kk, \pi(k', t)}_{k'}\|^2 \right] \\
&\smallspace +\mathbb{E}_{\pi} \left[ \frac{\phi_t \Gamma_t \varphi_t n \phi_t \Upsilon}{3}  \sum_{ k'} \| \var{z}{\kk, \pi}_{k'} - \var{z}{\kk, \pi(k', t)}_{k'}\|^2 \right] \Bigg] \\
& \largespace \leq \sum_t \Big(\prod_{l = t+1 \cdots T-1} (1 - \parb \phi_{l})\Big) \cdot \Bigg[7 \Gamma_t \phi_t \cdot 16q \sum_{s \in [t - q - 1, t -1]} (\Delta_{s}^{\tFE})^2\Bigg] ~~~\mbox{(as $n\phi_t \varphi \Upsilon\leq 1$)}\\
&\largespace \leq \sum_t \Big(\prod_{l = t+1 \cdots T-1} (1 - \parb \phi_{l})\Big) \cdot \Bigg[\frac{112 q^2 \Gamma_t \phi_t}{\Phi_c}  (\Delta_{t}^{\tFE})^2\Bigg]\\
&\largespace\smallspace \text{(using the definition of $\Phi_c$ and noting that there are at most $q$ occurences}\\
&\largespace\largespace \text{of each $(\Delta_s^\tFE)^2$ term)}\\
&\largespace \leq \sum_t \frac{\Xi}{\Phi_c(\Phi_c - \Xi)} \Big(\prod_{l = t+1 \cdots T-1} (1 - \parb \phi_{l})\Big) \cdot \Bigg[112 q^2 \Gamma_t \phi_t  E_t^{\Delta}\Bigg].
\end{align*}
\hide{
\begin{align*}
&\sum_{t} \Big(\prod_{l = t+1 \cdots T} (1 - \parb \alpha_{l})\Big) \cdot \Bigg[\mathbb{E}_{\pi} \left[ 11 \cdot \Gamma_t \alpha_t \sum_{k'} \|\var{w}{\kk, \pi}_{k'} - \var{w}{\kk, \pi(k', t)}_{k'}\|^2 \right] \\
&\smallspace +\mathbb{E}_{\pi} \left[ \frac{\alpha_t \Gamma_t (1 - \beta_t) n \alpha_t (1 - \tilde{\tau})}{ 3} \sum_{ k'} \| \var{z}{\kk, \pi}_{k'} - \var{z}{\kk, \pi(k', t)}_{k'}\|^2 \right] \Bigg] \\
& \largespace \leq \sum_t \Big(\prod_{l = t+1 \cdots T} (1 - \parb \alpha_{l})\Big) \cdot \Bigg[12 \Gamma_t \alpha_t \cdot 16q \sum_{s \in [t - q - 1, t -1]} (\Delta_{s}^{\tFE})^2\Bigg]\\
&\largespace \leq \sum_t \Big(\prod_{l = t+1 \cdots T} (1 - \parb \alpha_{l})\Big) \cdot \Bigg[\frac{192 q^2 \Gamma_t \alpha_t}{\phi_c}  (\Delta_{t}^{\tFE})^2\Bigg]\\
&\largespace \leq \sum_t \frac{\xi}{\phi_c(\phi_c - \xi)} \Big(\prod_{l = t+1 \cdots T} (1 - \parb \alpha_{l})\Big) \cdot \Bigg[192 q^2 \Gamma_t \alpha_t  E_t^{\Delta}\Bigg]\\
& \largespace\largespace \text{by Lemma~\ref{lem::delta::act} applied to the series
$c_t$ =12 \Gamma_t \alpha_t$.}
$.
\end{align*}
}

Therefore, to ensure $\sum_t \prod_{l = t+1}^{T-1} (1 - \parb \phi_{l}) \Adj_t \ge 0$, it suffices to have 
\begin{align*}
&\frac{1}{(\Phi_b - \Xi)} \frac{\frac{3}{ \Upsilon} + 2 n \phi_t }{2  \Gamma_t} \cdot \frac{ ({896} +  {\frac{{397}}{1-r}} \xi) q^2 \Lresbar^2 n^2 \phi_t^2}{n} + \frac{\Xi}{\Phi_c(\Phi_c - \Xi)}112 q^2 \Gamma_t \phi_t \leq \frac{n\phi_t(\frac{4}{5}\Gamma_t - n \phi_t)}{2},
\end{align*}
which is a consequence of constraints (iv)--(vii). 
\end{proof}

\section{Proofs of the Remaining Theorems}
\label{app::proofs-of-thms}

\subsection{Proofs of Theorems~\ref{thm::asyn::compare} and \ref{thm::apcg::final}}

Next, we obtain lower bounds on $\Phi_b$ and $\Phi_c$ under a condition that holds with our parameter choices
in Theorems~\ref{thm::asyn::compare} and~\ref{thm::apcg::final}.

\begin{lemma}
\label{lem::phi-bounds}
Suppose that $\frac{\sqrt{\alpha_t}}{\Gamma_t}$ is a constant.
Then if $n\ge {50}q$, $\Phi_b,\Phi_c \ge \frac 45$, $\xi {\leq} \frac 65$, {and if $ q \le \frac{\sqrt n}{{17}}$ the condition on $r$ holds}.
\end{lemma}
\begin{proof}
We begin with the bound on $\Phi_b$. By inspection,
\begin{align*}
\Phi_b &\ge \min_{1\le k \le {4q}} \left\{
\left[\left(\frac{\Gamma_t}{\Gamma_{t-k}}\right) \cdot \left( \frac{\phi_{t-k}} {\phi_t}\right)^3 \prod_{l=t-k+1}^t (1 - \tau \phi_l)\right],
\left[\left(\frac{\Gamma_t}{\Gamma_{t-k}}\right) \cdot \left(\frac{\Gamma_t} {\Gamma_{t-k}}\right) \cdot \left( \frac{\phi_{t+k}} {\phi_t}\right)^3 \prod_{l=t+1}^{t+k} (1 - \tau \phi_l)^{-1}\right]
\right\}\\
&\ge \min_{1\le k \le 4q} 
\left\{
\left[\left( \frac{\phi_{t-k}} {\phi_t}\right)^{5/2} \prod_{l=t-k+1}^t (1 - \tau \phi_l)\right],
\left[\left( \frac{\phi_{t+k}} {\phi_t}\right)^{5/2} \prod_{l=t+1}^{t+k} (1 - \tau \phi_l)^{-1}\right]
\right\} \\
& \ge \min\left\{ \left( 1 - \frac 1n \right)^{4q}, \left( 1 - \frac {1}{2n} \right)^{10q} \right\} ~~~~\text{(as by Lemma~\ref{lem::prop::alpha}(ii),
$\phi_l \le \frac 1n$ for all $l$)}\\
& \ge 1 - \frac {10q}{n}~~~~~~~~~\text{if~~} \frac {5q}{n} \le 1 \\
& \ge \frac 45 ~~~~~~~~~\text{if } n\ge 50q.
\end{align*}
Essentially the same argument yields the same lower bound on $\Phi_c$ and an upper bound on $\xi$.

{Finally, a direct calculation shows the final claim.}
\end{proof}

\pfof{Theorem~\ref{thm::asyn::compare}}
We first observe that the assumptions of Theorem~\ref{thm::asyn::compare} imply the
conditions for Lemma~\ref{lem::phi-bounds} and the $\var{B}{t}$ are good by Lemma~\ref{lem::prop::alpha} and \ref{lem::B::good}.
In addition, by assumption, $\frac{n\phi_t }{\Gamma_t} = \frac {3}{20}$, and  $q \le \frac{1}{25}\frac{\sqrt{n}}{\Lresbar}$;
thus it follows from the definition of $\Xi$ that $\Xi \le \frac 1{125}$.
Furthermore, the choice of $\tau = \frac 12$ and  $\frac{n\phi_t }{\Gamma_t} = \frac {3}{20}$
establishes constraints~\ref{lab::enu::str::1}, \ref{lab::enu::str::3}, and \ref{lab::enu::str::7}.
Next, the choice of $\phi_t = \frac{\sqrt{3 \mu}}{\sqrt{20} n}$ and $\Gamma_t = \frac{\sqrt{20 \mu}}{\sqrt{3}}$ establishes constraint \ref{lab::enu::str::2}.
$n\ge 19$ implies $ \frac n{25} \le \min\left\{\frac{n-8}{12}, \frac{2n - 4}{10}, \frac{n}{20} \right\}$,
so the assumption that $ q \le \frac n{25}$ establishes constraint \ref{lab::enu::str::9}.
By calculation, constraint \ref{lab::enu::str::4} becomes
\begin{align*}
\frac{1}{\frac 45 - \frac {1}{125}} \cdot \frac 3{2 \sqrt{\frac {20\mu}{3}} \frac 12}\cdot {{1388}} \cdot \left(\frac{q^2 \Lresbar^2}{n}\right) \sqrt{\frac{3\mu}{20}}\le {\frac {43}{200}} \sqrt{\frac {20\mu}{3}},
\end{align*}
i.e.\ $q \le \frac {1}{{38}} \frac{\sqrt n \mu^{1/4}}{ \Lresbar}$ suffices.

Again, constraint \ref{lab::enu::str::5} becomes
\begin{align*}
\frac{125}{99} \cdot 2\cdot \frac {3}{20}\cdot {{1388}} \cdot {\frac {3\mu}{20}} 
\frac {q^2\Lresbar^2}{n} \le \frac {3}{50}\mu,
\end{align*}
i.e.\ $q \le \frac {1}{{37}} \frac{\sqrt n}{\Lresbar}$ suffices.

Finally, constraint \ref{lab::enu::str::6} becomes
\begin{align*}
\frac{1}{125} \cdot \frac{125}{99} \cdot \frac 54 \cdot 112 \cdot q^2 \le \frac{1}{{200}} n,
\end{align*}
i.e.\  $q \le \frac{\sqrt{n}}{{17}}$ suffices.
\hide{
Since $q \leq \frac{n-1}{12}$ and $\alpha_t$, $\Gamma_t$ are fixed, $\phi_b \geq 4^{-\frac{1}{6}} \geq \frac{3}{4}$, $\phi_c \geq 4^{-\frac{1}{6}} \geq \frac{3}{4}$, and $\xi \leq \frac{1}{4}$ by $\alpha_{l} \leq \frac{1}{n}$. Therefore, if we pick $\parb = \tilde{\parb} = \frac{1}{2}$, all constraints in Theorem~\ref{thm::asyn::apcg} hold, and therefore the theorem follows.
The conditions on $q$ come from constraints (v)--(vii) in Theorem~\ref{thm::overall}.
}
\end{proof}

\pfof{Theorem~\ref{thm::apcg::final}}

{\bf Strongly convex case}~
In the strongly convex case, we choose $\tau = 1 - \epsilon$, $\Gamma_t = \frac{20}{3} \sqrt{n \phi_t}$,
which establishes constraints \ref{lab::enu::str::3} and \ref{lab::enu::str::7}, and choose $\phi_t = \frac{\left(\frac{3}{20} \mu \right)^{\frac{2}{3}}}{n}$ which establishes constraints \ref{lab::enu::str::1} and \ref{lab::enu::str::2}.
Note that the assumptions of Theorem~\ref{thm::apcg::final} imply the $\var{B}{t}$ are good by Lemma~\ref{lem::prop::alpha} and \ref{lem::B::good}.

\smallskip
\noindent
\emph{Constraint \ref{lab::enu::str::4}}.
\begin{align*}
\frac{1}{(\Phi_b - \Xi)}\cdot \frac{3}{ 2 \Gamma_t (1 - \tau)} \cdot \frac{({{1388}}) q^2 \Lresbar^2 n^2 \phi_t^2}{n} \leq {\frac{43}{200}} n \phi_t \Gamma_t.
\end{align*}
Using the assumption that $1 - \tau =  \epsilon$,
this reduces to
\begin{align*}
\frac {125}{99}\cdot \frac{3}{2} \cdot {{1388}} \cdot {\frac{200}{43}} \cdot \frac {1}{\left(\frac{20}{3}\right)^2} \left( \frac{q^2 \Lresbar^2}{n\epsilon}\right) \le 1,
\end{align*}
and so it suffices that $q \le \frac{1}{{17}} \frac {\sqrt\epsilon \sqrt n }{\Lresbar}$.
\hide{
This holds as $q \leq \epsilon \frac{3 \sqrt{n} \Gamma_{-1}}{500 \Lresbar}$, $\Gamma_t = \frac{20}{3} \sqrt{n \alpha_t}$, $\phi_b \geq \frac{3}{4}$ and $\xi \leq \frac{1}{4}$, where $\tilde{\tau} = 1 - \epsilon$.
}

\smallskip
\noindent
\emph{Constraint \ref{lab::enu::str::5}}.
\begin{align*}
\frac{1}{(\Phi_b - \Xi)} \cdot \frac{{2} n \phi_t}{ \Gamma_t}\cdot \frac{{{1388}} \cdot q^2 \Lresbar^2 n^2 \phi_t^2}{n} \leq \frac{3}{50} n \phi_t \Gamma_t.
\end{align*}
This reduces to
\begin{align*}
\frac{125}{99} \cdot 2 \cdot {{1388}} \cdot \frac{n\phi_t}{(\frac{20}{3})^2}\cdot \left(\frac{q^2 \Lresbar^2} {n} \right) \le \frac {3}{50},
\end{align*}
and as $n\alpha_t \le 1$, it suffices that $q \le \frac{1}{{37}} \frac{\sqrt n}{\Lresbar}$.

\hide{
This holds as $q \leq  \frac{3 \sqrt{n} \Gamma_{-1}}{1340 \Lresbar}$, $\Gamma_t = \frac{20}{3} \sqrt{n \alpha_t}$, $\phi_b \geq \frac{3}{4}$, $\xi \leq \frac{1}{4}$ and $\alpha_t \leq \frac{1}{n}$.
}

\smallskip
\noindent
\emph{Constraint \ref{lab::enu::str::6}}.
\begin{align*}
\frac{\Xi}{\Phi_c(\Phi_c - \Xi)}\cdot 112 \cdot q^2 \Gamma_t \phi_t \leq \frac{1}{50} n \phi_t \Gamma_t.
\end{align*}
As in the proof of Theorem~\ref{thm::asyn::compare}, $ q \le \frac {\sqrt n}{{17}}$ suffices.

\noindent
{\bf Non-strongly convex case}~
Remember that we chose $\Gamma_t = \frac{20}{3} \sqrt{n \phi_t}$ and this establishes \ref{lab::enu::non::3} and \ref{lab::enu::non::7} as $n \phi_t \leq 1$. We set $\tilde{\tau} = 1 - \frac{1}{4 t_0} - \epsilon$,
where $t_0 = 2(n+1)$. We will also choose $\tau$ so that $\tau \leq \tilde{\tau}$. Note that the assumptions of Theorem~\ref{thm::apcg::final} imply the $\var{B}{t}$ are good by Lemma~\ref{lem::prop::alpha} and \ref{lem::B::good}.
 
 Constraint~\ref{lab::enu::non::1} is implied by $\tilde{\tau} = 1 - \frac{1}{4 t_0} - \epsilon \ge \frac 12$ as $\epsilon < \frac 13$ and $n \geq 19$.

\smallskip
\noindent
\emph{ Constraint \ref{lab::enu::non::2}}.
\begin{align*}\left(1 - \tilde{\parb}\phi_{\kk}\right)\frac{\Gamma_{\kk}}{n\phi_{\kk}} \leq \left(1 - \parb\phi_{\kk}\right)\frac{\Gamma_{\kk - 1}}{n\phi_{\kk - 1}}.\end{align*} 
In order to satisfy this, we only need 
\begin{align*}
\left( \frac{1 - \tilde{\parb} \phi_{\kk}}{1 - \parb \phi_{\kk}}\right)^2 \leq \frac{\phi_{\kk + 1}}{\phi_{\kk}}.
\end{align*}

We know $\ln \left( \frac{1 - \tilde{\parb} \phi_{\kk}}{1 - \parb \phi_{\kk}}\right)^2 \leq 2\ln [1 - (\tilde{\parb} - \parb) \phi_{\kk}]\leq - 2 (\tilde{\parb} - \parb) \phi_{\kk} $, and by Lemma~\ref{lem::prop::alpha}(iii), $\ln \frac{\phi_{\kk + 1}}{\phi_{\kk}} \geq \ln \left(1 - \frac{\phi_{\kk}}{2}\right) \geq  - \frac{\phi_{\kk}}{2} - \left(\frac{\phi_{\kk}}{2}\right)^2$ as $\phi_{\kk} < 1$. Thus it suffices to have $2(\tilde{\parb} - \parb) \phi_{\kk} \geq  \phi_{\kk} [\frac{1}{2} + \frac{\phi_{\kk}}{4}]$.  Using the fact that $\phi_{\kk} \leq \frac{1}{n}$, it  suffices to let 
\begin{align*}
\tilde{\parb} - \parb = \frac{1}{4} + \frac{1}{8n}. \numberthis \label{eqn::parb::tparb}
\end{align*}

\smallskip
\noindent
\emph{Constraint \ref{lab::enu::non::4}}.
\begin{align*}
\frac{1}{(\Phi_b - \Xi)}\cdot \frac{3}{ 2 \Gamma_t (1 - \frac{1}{4 t_0} -  \tilde{\tau})} \cdot \frac{{{1388}} q^2 \Lresbar^2 n^2 \phi_t^2}{n} \leq \frac{1}{5} n \phi_t \Gamma_t.
\end{align*}
Using the assumption that $\tilde{\tau} = 1 - \frac{1}{4 t_0} - \epsilon > \frac 12$,
this reduces to
\begin{align*}
\frac {125}{99}\cdot \frac{3}{2} \cdot {{1388}} \cdot {\frac {200}{43}} \cdot \frac {1}{(\frac{20}{3})^2} \left( \frac{q^2 \Lresbar^2}{n\epsilon}\right) \le 1,
\end{align*}
and so it suffices that $q \le \frac{1}{{17}} \frac {\sqrt\epsilon \sqrt n}{\Lresbar}$.

\smallskip
\noindent
\emph{Constraint \ref{lab::enu::non::5}}.
\begin{align*}
\frac{1}{(\Phi_b - \Xi)} \cdot \frac{2 n \phi_t}{ \Gamma_t}\cdot \frac{{{1388}} \cdot q^2 \Lresbar^2 n^2 \phi_t^2}{n} \leq \frac{3}{50} n \phi_t \Gamma_t.
\end{align*}
This reduces to
\begin{align*}
\frac{125}{99} \cdot 2 \cdot {{1388}} \cdot \frac{n\phi_t}{(\frac{20}{3})^2}\cdot \left(\frac{q^2 \Lresbar^2} {n} \right) \le \frac {3}{50},
\end{align*}
and as $n\phi_t \le 1$, it suffices that $q \le \frac{1}{{37}} \frac{\sqrt n}{\Lresbar}$.

\hide{
This holds as $q \leq  \frac{3 \sqrt{n} \Gamma_{-1}}{1340 \Lresbar}$, $\Gamma_t = \frac{20}{3} \sqrt{n \alpha_t}$, $\phi_b \geq \frac{3}{4}$, $\xi \leq \frac{1}{4}$ and $\alpha_t \leq \frac{1}{n}$.
}

\smallskip
\noindent
\emph{Constraint \ref{lab::enu::non::6}}.
\begin{align*}
\frac{\Xi}{\Phi_c(\Phi_c - \Xi)}\cdot 112 \cdot q^2 \Gamma_t \phi_t \leq \frac{1}{{200}} n \phi_t \Gamma_t.
\end{align*}
As in the proof of Theorem~\ref{thm::asyn::compare}, $ q \le \frac {\sqrt n}{{17}}$ suffices.

Now, let's look at the convergence rate.
\begin{enumerate}
\item In the strongly convex case,   the convergence rate becomes
\begin{align*}
&\prod_{\kk} (1 - \parb \alpha_{\kk}) = \Bigg(1 - (1 - \epsilon) \frac{(\frac{3\mu}{20})^\frac{2}{3}}{n}\Bigg)^{T}.
\end{align*}
\item In the non-strongly convex case, recall that the convergence rate of $\prod_{k = 0, \cdots, T - 1} (1 -  \phi_{\kk})$ is $\frac{(t_0 - 2) (t_0 - 1)}{(t_0 + T- 1)(t_0 + T - 2)}$, Therefore,  by Lemma~\ref{lem::convergence}, the convergence rate of $\prod_{k = 0, \cdots, T-1} (1 -  \parb\phi_{\kk})$ is 
$$\Big(\frac{(t_0 - 2) (t_0 - 1)}{(t_0 + T-1)(t_0 + T - 2)}\Big)^{\frac{n \parb}{n+1}} = \Big(\frac{(t_0 - 2) (t_0 - 1)}{(t_0 + T-1)(t_0 + T - 2)}\Big)^{\frac{n (\frac{3}{4} -  \frac{1}{4t_0} - \epsilon - \frac{1}{8n})}{n+1}}.$$
\end{enumerate}

\end{proof}

%% file: app-th-alg-equiv.tex
\pfof{Theorem~\ref{alg::apcg::equl}}

We know that when $\kk = 0$, $\hatvar{{y}}{\kk, \pi} = \hatvar{{z}}{\kk, \pi} = \var{x}{\kk, \pi}$. We will prove by induction that for any $t > 0$, $\hatvar{{y}}{\kk, \pi} = \var{y}{\kk, \pi}$ and $\var{\hat{z}}{\kk, \pi} = \var{z}{\kk, \pi}$.

At $\kk = 0$, Algorithm~\ref{alg::apcg::asyn} sets $ \var{z}{\kk, \pi} = \var{x}{\kk, \pi}$, and for all $\kk$ sets $\var{y}{\kk, \pi} = (1 - \psi_t) \var{z}{\kk, \pi} + \psi_t \var{x}{\kk, \pi}$. Thus, $\hatvar{y}{0, \pi} = \hatvar{x}{0, \pi} = \hatvar{z}{0, \pi} = \var{{y}}{0, \pi} = \var{{z}}{0, \pi}$.

Now, suppose the hypothesis is true for all $\kk \leq \tt$. We analyze the case $\kk = \tt+1$. By \eqref{apcg::async::reorder}, and $\var{B}{\tt + 1} = \var{A}{\tt} \var{B}{\tt}$, 
\begin{align*}
(\hatvar{{y}}{\tt + 1, \pi}, \hatvar{{z}}{\tt + 1, \pi})\tran& = \var{A}{\tt} ((\hatvar{y}{\tt, \pi}, \hatvar{z}{\tt, \pi})\tran + \var{D}{\tt, \pi}).
\end{align*}  

By the definition of $\var{A}{\tt}$ and $\var{D}{\tt}$, 
\begin{align*}
&\hatvar{{y}}{\tt + 1, \pi} = ((1 - \psi_{t+1}) (1 - \varphi_t) + \psi_{t+1}) \hatvar{y}{\tt, \pi} + (1 - \psi_{t+1}) \varphi_t \hatvar{z}{\tt, \pi} + (n \psi_{t+1} \phi_t + (1 - \psi_{t+1})) \Delta z^{l, \pi}_{k_l} \ivec_{\iit_{\tt}}, \numberthis \label{eqn::apcg::equal::1}\\
&\mbox{and }~~~~~~~~~~~\hatvar{{z}}{\tt + 1, \pi} = (1 - \varphi_t) \hatvar{y}{\tt, \pi} + \varphi_t \hatvar{z}{\tt, \pi} + \Delta z^{l, \pi}_{k_l} \ivec_{\iit_{\tt}}. \numberthis \label{eqn::apcg::equal::2}
\end{align*}

We treat $y$ and $z$ separately. First, we consider $\hatvar{z}{\tt + 1, \pi}$.  It's easy to see that  if $\hatvar{z}{\tt, \pi} = \var{z}{\tt, \pi}$ and $\hatvar{y}{\tt, \pi} = \var{y}{\tt, \pi}$, then  $\hatvar{z}{\tt + 1, \pi} = \var{z}{\tt + 1, \pi}$. 

Next, we look at $\var{y}{\tt + 1, \pi}$. 
\begin{align*}
\var{y}{\kk + 1, \pi} &= \psi_{t+1} \var{x}{\kk+1, \pi} + (1 - \psi_{t+1}) \var{z}{\kk+1, \pi} \\
&= \psi_{t+1}(\var{y}{\kk, \pi} + n \phi_t \Delta z^{t, \pi}_{k_t} \ivec_{\iit_{\tt}}) + (1 - \psi_{t+1}) \varphi_t \var{z}{\kk, \pi} + (1 - \psi_{t+1})(1 - \varphi_t) \var{y}{\kk, \pi} + (1 - \psi_{t+1}) \Delta z^{t, \pi}_{k_t} \ivec_{\iit_{\tt}}.
\end{align*}
Comparing this with \eqref{eqn::apcg::equal::1}, by the induction hypothesis, we get $\var{y}{\tt + 1, \pi} = \hatvar{{y}}{\tt + 1, \pi}$.
Since $\var{y}{\tt + 1, \pi} = \hatvar{{y}}{\tt + 1, \pi}$ and $\var{z}{\tt + 1, \pi} = \hatvar{{z}}{\tt + 1, \pi}$, and as  $\hatvar{x}{\tt + 1, \pi}$  and $\hatvar{{x}}{\tt + 1, \pi}$ satisfy, respectively:
\begin{align*}
\var{y}{\tt + 1, \pi} = (1 - \psi_{t+1}) \var{z}{\tt + 1, \pi} +  \psi_{t+1} \var{x}{\tt + 1, \pi}\\
\mbox{and }~~~~~~ \hatvar{{y}}{\tt + 1, \pi} = (1 - \psi_{t+1}) \hatvar{{z}}{\tt + 1, \pi} + \psi_{t+1} \hatvar{{x}}{\tt + 1, \pi},
\end{align*} 
also, as $\var{w}{\tt+1, \pi}$ 
and $\hatvar{{w}}{\tt+1, \pi}$ satisfy, respectively:
\begin{align*}
\var{w}{\tt+1, \pi} =\varphi_t \var{z}{\tt+1, \pi} + (1 - \varphi_t) \var{y}{\tt+1, \pi}\\
\mbox{and }~~~~~~ \hatvar{w}{\tt+1, \pi} = \varphi_t \hatvar{z}{\tt+1, \pi} + (1 - \varphi_t) \hatvar{y}{\tt+1, \pi},
\end{align*} 
the theorem follows.
\end{proof}

%% file: Progress-Lemma.tex
\section{Proof of Lemma~\ref{lem::apcg::error::non::str},
the Progress Lemma}
\label{app::progress::lemma}
In this section, we prove Lemma~\ref{lem::apcg::error::non::str}.
We begin by stating the full version of Lemma~\ref{lem::apcg::error::non::str}.
\begin{lemma} \label{lem::apcg::error::non::str}
\begin{itemize}
\item Strongly convex case

Suppose that $0 < \parb \leq 1$, $\phi_t = \phi$, $\varphi_t = (1 - \phi)$, $\psi_t = \frac{1}{1+\phi}$, and that 
 $\Gamma_t = \Gamma$ and $\tilde{\mu}$
satisfy the following constraints, for all $\kk \geq 0$:   
\begin{align*}
{\bf i}.~ \phi \leq \frac{1}{n}; ~~
{\bf ii}.~ n \phi \Gamma \leq \mu;   ~~
{\bf iii}.~ \frac{4}{5}\Gamma \geq n \phi,
\end{align*}
and let $\zeta_{\kk + 1} = \frac{n \phi_t \Gamma_t}{2} \left( 1 - \frac{n \phi_t ( 1 - \tau)}{3}\right) $.
\begin{align*}
\text{Then}\hspace*{1.5in}\mathbb{E}_{\pi} \left[\var{\mathcal{F}}{\kk + 1}\right] &\leq  (1 - \tau \phi_t)  \mathbb{E}_{\pi}\left[\var{\mathcal{F}}{\kk} \right] ~-~ \Adj_{\kk},\hspace*{2in}
\end{align*}
\begin{align*}
\text{where}\hspace*{0.5in}\Adj_{\kk} ~= &~  \mathbb{E}_{\pi} \left[ \frac{n\phi_t(\frac{4}{5}\Gamma_t - n \phi_t)}{2} \left(\Delta z_{\iit_t}^{t, \pi}  \right)^2\right] \\
&\smallspace \smallspace - \mathbb{E}_{\pi} \Bigg[ \frac{1}{n} \sum_{ k'} \Bigg(\frac{\frac{3}{ (1 - \tau)} + 2 n \phi_t }{2 \Gamma_t}\left(\grad{\iit'}{f(\var{y}{\kk, \pi})} - \tilde{g}^{t, \pi(k', t)}_{\iit'}\right)^2  \\
&\largespace \largespace +  \frac{10 n \phi_t}{2 \Gamma_t} \left(\grad{\iit'}{f(\var{y}{\kk, \pi})} - \grad{\iit'}{f(\var{y}{\kk, \pi(k', t)})}\right)^2 \Bigg) \Bigg]\hspace*{0.5in} \\
&\smallspace \smallspace -\mathbb{E}_{\pi} \left[ 6 \Gamma_t \phi_t \sum_{k'} \|\var{w}{\kk, \pi}_{k'} - \var{w}{\kk, \pi(k', t)}_{k'}\|^2 \right] \\
&\smallspace \smallspace -\mathbb{E}_{\pi} \left[ \frac{\phi_t \Gamma_t \varphi_t n \phi_t (1 - \tau)}{ 3} \sum_{ k'} \| \var{z}{\kk, \pi}_{k'} - \var{z}{\kk, \pi(k', t)}_{k'}\|^2 \right].
\end{align*}
In this case, $\var{A}{t} =  \Bigg (\begin{matrix} 
\frac{1 + \phi^2}{1+ \phi} & 1 - \frac{1 + \phi^2}{1+ \phi}\\
\phi & 1 - \phi
\end{matrix} \Bigg)$ and $\var{B}{t} = \Bigg (\begin{matrix} 
\frac{1 + \phi^2}{1+ \phi} & 1 - \frac{1 + \phi^2}{1+ \phi}\\
\phi & 1 - \phi
\end{matrix} \Bigg)^{t}$.

\item Non-strongly convex case

Suppose that $0 < \tau \leq 1$, $0 < \tilde{\tau} \leq 1$, $\phi_t = \frac{2}{t_0 + t}$ for $t_0 \geq 2 (n+1)$, $ \varphi_t = 1$, $\psi_t = 1 - \phi_t$, and $\Gamma_t$ satisfy the following constraints, for all $t \geq 0$:
\begin{enumerate}[(i)]
\begin{multicols}{2}
\item
$\tilde{\tau} \geq \frac{1}{2}$;
\item
$(1 - \tilde{\tau} \phi_t) \frac{\Gamma_t}{n \phi_t} \leq (1 - \tau \phi_t) \frac{\Gamma_{t - 1}}{n \phi_{t-1}}$; 
\item $\frac{4}{5}\Gamma_t \geq n \aalpha_t$;
\end{multicols}
\end{enumerate}
and let $\zeta_{\kk + 1} = \frac{n \phi_t \Gamma_t}{2} \left(1 -  \frac{n \phi_t \left(1 - \tilde{\tau} - \frac{1}{4 t_0}\right)}{3} \right) $.
\begin{align*}
\text{Then}\hspace*{1.5in}\mathbb{E}_{\pi} \left[\var{\mathcal{F}}{\kk + 1}\right] &\leq  (1 - \tau \phi_t)  \mathbb{E}_{\pi}\left[\var{\mathcal{F}}{\kk} \right] ~-~ \Adj_{\kk},\hspace*{2in}
\end{align*}
\begin{align*}
\text{where}\hspace*{0.5in}\Adj_{\kk} ~= &~  \mathbb{E}_{\pi} \left[ \frac{n\phi_t(\frac{4}{5}\Gamma_t - n \phi_t)}{2} \left(\Delta z_{\iit_t}^{t, \pi}  \right)^2\right] \\
&\smallspace \smallspace - \mathbb{E}_{\pi} \Bigg[ \frac{1}{n} \sum_{ k'} \Bigg(\frac{\frac{3}{ (1 - \tau - \frac{1}{4t_0})} + 2 n \phi_t }{2 \Gamma_t}\left(\grad{\iit'}{f(\var{y}{\kk, \pi})} - \tilde{g}^{t, \pi(k', t)}_{\iit'}\right)^2  \\
&\largespace \largespace +  \frac{10 n \phi_t}{2 \Gamma_t} \left(\grad{\iit'}{f(\var{y}{\kk, \pi})} - \grad{\iit'}{f(\var{y}{\kk, \pi(k', t)})}\right)^2 \Bigg) \Bigg]\hspace*{0.5in} \\
&\smallspace \smallspace -\mathbb{E}_{\pi} \left[ {6} \Gamma_t \phi_t \sum_{k'} \|\var{w}{\kk, \pi}_{k'} - \var{w}{\kk, \pi(k', t)}_{k'}\|^2 \right] \\
&\smallspace \smallspace -\mathbb{E}_{\pi} \left[ \frac{\phi_t \Gamma_t \varphi_t n \phi_t (1 - \tau - \frac{1}{4t_0})}{ 3} \sum_{ k'} \| \var{z}{\kk, \pi}_{k'} - \var{z}{\kk, \pi(k', t)}_{k'}\|^2 \right].
\end{align*}
In this case, $\var{A}{t} =  \Bigg (\begin{matrix} 
\frac{t+t_0 - 1}{t+t_0 + 1} & \frac{2}{t + t_0 + 1}\\
0  & 1
\end{matrix} \Bigg)$ and $\var{B}{t} =  \Bigg (\begin{matrix} 
\frac{t_0(t_0 - 1)}{(t_0 + t \rjc{-1})(t_0 + t)} & 1 -\frac{t_0(t_0 - 1)}{(t_0 + t)(t_0 + t + 1)} \\
0  & 1
\end{matrix} \Bigg)$.
\end{itemize}
\end{lemma}
We restate Algorithm~\ref{alg::apcg::asyn} for the reader's convenience. 
We begin by determining constraints on these parameters for which Lemma~\ref{lem::apcg::error::non::str} holds. We then show the parameter choice in Algorithm~\ref{alg::apcg::asyn} satisfies these constraints.
\begin{algorithm}
\caption{The basic iteration}
\label{alg::apcg::asyn::copy}
	Choose ${\iit}_{\kk} \in \{1, 2, \cdots, n\}$ uniformly at random\;
	$\var{y}{\kk, \pi} = \ggamma_{\kk} \var{x}{\kk, \pi} + (1 - \ggamma_{\kk}) \var{z}{\kk, \pi}$\;
	$\var{w}{\kk, \pi} = \bbeta_{\kk} \var{z}{\kk, \pi} + (1 - \bbeta_{\kk}) \var{y}{\kk, \pi}$\;
	$\var{z}{\kk+1, \pi} = \arg\min_x \{ \frac{\Gamma_{\kk}}{2} \|x - \var{w}{\kk, \pi}\|^2 + \dotprod{\tilde{g}^{\kk, \pi}_{\iit_{\kk}}}{x_{\iit_{\kk}}}\}$\;
	$\var{x}{\kk+1, \pi} = \var{y}{\kk, \pi} + n \aalpha_{\kk} (\var{z}{\kk+1, \pi} - \var{w}{\kk, \pi})$\; 
\end{algorithm}

\begin{lemma}\label{apcg::lem::similar::proof}
Let $\zzeta_{\kk+1} = \frac{n \aalpha_{\kk} \eta_{\kk} \Gamma_{\kk} }{2(\eta_{\kk} + 1)}$, and suppose that
the parameters satisfy the following constraints:
\begin{enumerate}[i.]
\item$\eta_{\kk} > 1$;
\item$\frac{ n\aalpha_{\kk}\ggamma_{\kk}\bbeta_{\kk}}{1 - \ggamma_{\kk}} = n(1 - \aalpha_{\kk})$;
\item$(1 - \aalpha_t) \left(\frac{n \aalpha_t \Gamma_t \varphi_t(\eta_t + \frac{2}{n})}{2 (\eta_t +1) (1 - \aalpha_t)}\right) \leq (1 - \tau \aalpha_t) \left(\frac{n \aalpha_{t-1} \eta_{t-1} \Gamma_{t-1}}{2 (\eta_{t-1} + 1)}\right)$;
\item$\frac{n  \Gamma_t (1 - \varphi_t)}{2 } \leq \frac{\mu}{2}$.
\end{enumerate}
Then,
\begin{align*}
&\frac{1}{n} \sum_{k} \left[ f(\var{x}{\kk+1, \pi(k, t)}) - f(x^*) + \zeta_{t+1} \| x^* - \var{z}{\kk+1, \pi(k, t)}\|^2 \right] \\
& \smallspace \leq \frac{1}{n} (1 - \tau \aalpha_t) \sum_{k} \left[ f(\var{x}{\kk, \pi(k, t)}) - f(x^*)  + \zeta_{t} \| x^* - \var{z}{\kk, \pi(k, t)}\|^2 \right]\\
&\smallspace \smallspace - \sum_{k} \frac{\aalpha_t(\frac{4}{5}\Gamma_t - n \aalpha_t)}{2} \left(\Delta z_{\iit}^{t, \pi(k, t)}  \right)^2 \\
&\smallspace \smallspace + \frac{1}{n} \aalpha_t \sum_{k, k'}\left[\frac{\eta_t + {2}}{2 \Gamma_t}\left(\grad{\iit'}{f(\var{y}{\kk, \pi(k, t)})} - \tilde{g}^{t, \pi(k', t)}_{\iit'}\right)^2 + \frac{10}{2 \Gamma_t} \left(\grad{\iit}{f(\var{y}{\kk, \pi(k, t)})} - \grad{\iit}{f(\var{y}{\kk, \pi(k', t)})}\right)^2  \right]  \\
&\smallspace \smallspace + \frac{6}{n} \Gamma_t \aalpha_t \sum_{k, k'} \|\var{w}{\kk, \pi(k, t)}_{k'} - \var{w}{\kk, \pi(k', t)}_{k'}\|^2 \\
&\smallspace \smallspace + \frac{2}{n} \frac{\aalpha_t \Gamma_t \varphi_t}{2  (\eta_t + 1)} \sum_{k, k'} \| \var{z}{\kk, \pi(k, t)}_{k} - \var{z}{\kk, \pi(k', t)}_{k}\|^2.
\end{align*}
\end{lemma}
\begin{proof} 
We define $\Delta x^{t, \pi}_k \triangleq \var{x}{t+1, \pi} - \var{y}{t, \pi}$. 
Recall that $\Delta z_{\iit}^{t+1, \pi(k, t)} = z_{\iit}^{t, \pi(k, t)} - w_{\iit}^{t, \pi(k, t)}$.
Therefore,
$\Delta x_{\iit}^{t, \pi(k, t)} = n \aalpha_t \Delta z_{\iit}^{t, \pi(k, t)}$.

Since for any $\pi'$, $\var{y}{\kk, \pi'} = \ggamma_{\kk} \var{x}{\kk, \pi'} + (1 - \ggamma_{\kk}) \var{z}{\kk, \pi'}$,
\begin{align*}
\var{z}{\kk, \pi'} - \var{y}{\kk, \pi'} = \frac{\ggamma_{\kk}}{1 - \ggamma_{\kk}} (\var{y}{\kk, \pi'} - \var{x}{\kk, \pi'}).
\end{align*}So,
\begin{align*}
0 &= n\aalpha_{\kk} (\var{w}{\kk, \pi'} - \var{w}{\kk, \pi'}) \\
&=  n\aalpha_{\kk} \var{w}{\kk, \pi'} - n\aalpha_{\kk} \left[\bbeta_{\kk} \var{z}{\kk, \pi'} + (1 - \bbeta_{\kk}) \var{y}{\kk,\pi'}\right] ~~~~\text{(using Line 3, Algorithm \ref{alg::apcg::asyn::copy})}\\
&= n\aalpha_{\kk} \var{w}{\kk, \pi'} - n\aalpha_{\kk} \left[\var{y}{\kk, \pi'} +  \frac{\ggamma_{\kk}\bbeta_{\kk}}{1 - \ggamma_{\kk}} (\var{y}{\kk, \pi'} - \var{x}{\kk, \pi'})\right] ~~~~\text{(using Line 2, Algorithm \ref{alg::apcg::asyn::copy})}\\
&= n\aalpha_{\kk} (\var{w}{\kk, \pi'} - \var{y}{\kk, \pi'}) +  \frac{ n\aalpha_{\kk}\ggamma_{\kk}\bbeta_{\kk}}{1 - \ggamma_{\kk}} (\var{x}{\kk, \pi'} - \var{y}{\kk, \pi'}) \\
&= n \aalpha_{\kk} (\var{w}{\kk, \pi'} - \var{y}{\kk, \pi'}) + n (1 - \aalpha_{\kk}) (\var{x}{\kk, \pi'} - \var{y}{\kk, \pi'})
~~~~\text{(using Constraint (ii))}. \numberthis \label{eqn::w::y::x}
\end{align*}
Note that $\var{x}{\kk+1, \pi(k, t)}$ is the same as $\var{y}{\kk, \pi(k, t)}$ on all the coordinates other than $\iit$ \footnote{This is because $\var{z}{\kk+1}$ is the same as $\var{w}{\kk}$ on all the coordinates other than $i_{\kk}$, and $\var{x}{\kk+1} - \var{y}{\kk} = n \aalpha_t(\var{z}{\kk+1} - \var{w}{\kk})$.}, and  $f$ is a convex function with $L_{k_{\kk}, k_{\kk}} = 1$. Therefore,
\begin{align*}
f(\var{x}{\kk+1, \pi(k, t)}) &\leq f(\var{y}{\kk, \pi(k, t)}) + \dotprod{\grad{\iit}{f(\var{y}{\kk, \pi(k, t)})}}{\var{x}{\kk+1, \pi(k, t)}_{\iit} - \var{y}{\kk, \pi(k, t)}_{\iit}} + \frac{1}{2}(\var{x}{\kk+1, \pi(k, t)}_{\iit} - \var{y}{\kk, \pi(k, t)}_{\iit})^2\\
& =  f(\var{y}{\kk, \pi(k, t)}) + \frac{1}{n} \sum_{k'} \dotprod{\grad{k}{f(\var{y}{\kk, \pi(k', t)})}}{\Delta x_{k}^{t, \pi(k, t)}} + \frac{1}{2}\left(\Delta x_{\iit}^{t, \pi(k, t)}\right)^2 \\
&\largespace + \frac{1}{n} \sum_{k'} \dotprod{\grad{\iit}{f(\var{y}{\kk, \pi(k, t)})} - \grad{\iit}{f(\var{y}{\kk, \pi(k', t)})}}{\Delta x_{\iit}^{t, \pi(k, t)}}\\
& = f(\var{y}{\kk, \pi(k, t)}) + \frac{1}{n} \sum_{k'} \Big\langle\grad{\iit}{f(\var{y}{\kk, \pi(k', t)})}, \\
&\largespace n\aalpha_t (\var{w}{\kk, \pi(k', t)}_k + \Delta z_{\iit}^{t, \pi(k, t)} - \var{y}{\kk, \pi(k', t)}_k) + n (1 - \aalpha_t) (\var{x}{\kk, \pi(k', t)}_k - \var{y}{\kk, \pi(k', t)}_k) \Big\rangle \\
&\largespace + \frac{n^2 \aalpha_t^2}{2} \left(\Delta z_{\iit}^{t, \pi(k, t)} \right)^2 + \frac{1}{n} \sum_{k'} \dotprod{\grad{\iit}{f(\var{y}{\kk, \pi(k, t)})} - \grad{\iit}{f(\var{y}{\kk, \pi(k', t)})}}{\Delta x_{\iit}^{t, \pi(k, t)}}\\
&\hspace*{1in}\text{(using \eqref{eqn::w::y::x} and the fact that $\Delta x_{\iit}^{t, \pi(k, t)}  =  n \aalpha_t \Delta z_{\iit}^{t, \pi(k, t)}$)}.
\end{align*}

Summing over all $k$ gives
\begin{align*}
&\frac{1}{n} \sum_{k} f(\var{x}{\kk+1, \pi(k, t)}) \\
&\smallspace = \frac{1}{n} \sum_{k}f(\var{y}{\kk, \pi(k, t)}) + \frac{1}{n} \sum_{k}\frac{1}{n} \sum_{k'} \Big\langle\grad{\iit}{f(\var{y}{\kk, \pi(k', t)})}, \\
&\largespace n\aalpha_t (\var{w}{\kk, \pi(k', t)}_k + \Delta z_{\iit}^{t, \pi(k, t)} - \var{y}{\kk, \pi(k', t)}_k) + n (1 - \aalpha_t) (\var{x}{\kk, \pi(k', t)}_k - \var{y}{\kk, \pi(k', t)}_k) \Big\rangle \\
&\largespace + \frac{1}{n} \sum_{k}\frac{n^2 \aalpha_t^2}{2} \left(\Delta z_{\iit}^{t, \pi(k, t)} \right)^2 + \frac{1}{n} \sum_{k}\frac{1}{n} \sum_{k'} \dotprod{\grad{\iit}{f(\var{y}{\kk, \pi(k, t)})} - \grad{\iit}{f(\var{y}{\kk, \pi(k', t)})}}{\Delta x_{\iit}^{t, \pi(k, t)}} \\
&\smallspace = \frac{1}{n} \sum_{k}f(\var{y}{\kk, \pi(k, t)}) + n (1 - \aalpha_t) \frac{1}{n} \sum_{k}\frac{1}{n} \sum_{k'} \Big\langle\grad{\iit}{f(\var{y}{\kk, \pi(k', t)})}, \var{x}{\kk, \pi(k', t)}_k - \var{y}{\kk, \pi(k', t)}_k \Big\rangle\\
&\smallspace \largespace + n\aalpha_t \frac{1}{n} \sum_{k}\frac{1}{n} \sum_{k'} \Big\langle\grad{\iit}{f(\var{y}{\kk, \pi(k', t)})}, \var{w}{\kk, \pi(k', t)}_k + \Delta z_{\iit}^{t, \pi(k, t)} - \var{y}{\kk, \pi(k', t)}_k \Big\rangle \\
&\largespace + \frac{1}{n} \sum_{k}\frac{n^2 \aalpha_t^2}{2} \left(\Delta z_{\iit}^{t, \pi(k, t)} \right)^2 + \frac{1}{n} \sum_{k}\frac{1}{n} \sum_{k'} \dotprod{\grad{\iit}{f(\var{y}{\kk, \pi(k, t)})} - \grad{\iit}{f(\var{y}{\kk, \pi(k', t)})}}{\Delta x_{\iit}^{t, \pi(k, t)}} \\
&\smallspace \leq  \frac{1}{n} (1- \aalpha_t) \sum_{k}f(\var{x}{\kk, \pi(k, t)}) \\
&\smallspace\largespace+ \frac{1}{n} \aalpha_t \sum_{k} \Big[ f(\var{y}{\kk, \pi(k, t)})   + \sum_{k'} \Big\langle\grad{\iit'}{f(\var{y}{\kk, \pi(k, t)})}, \var{w}{\kk, \pi(k, t)}_{k'} + \Delta z_{\iit'}^{t, \pi(k', t)} - \var{y}{\kk, \pi(k, t)}_{k'} \Big\rangle \\
&\largespace \largespace + \sum_{k'}\frac{\Gamma_t}{2} \left(\Delta z_{\iit'}^{t, \pi(k', t)}  \right)^2 \Big] \\
&\smallspace \largespace - \sum_{k} \frac{\aalpha_t(\Gamma_t - n \aalpha_t)}{2} \left(\Delta z_{\iit}^{t, \pi(k, t)} \right)^2 + \frac{1}{n} \sum_{k}\frac{1}{n} \sum_{k'} \dotprod{\grad{\iit}{f(\var{y}{\kk, \pi(k, t)})} - \grad{\iit}{f(\var{y}{\kk, \pi(k', t)})}}{\Delta x_{\iit}^{t, \pi(k, t)}}. \numberthis \label{ineq::number::1}
\end{align*}

Since $\Delta z_{\iit}^{t, \pi(k, t)} = - \frac{1}{\Gamma_t} \tilde{g}^{t, \pi(k, t)}_{k}$,
by a simple calculation, 
\begin{align*}
&\dotprod{\grad{\iit'}{f(\var{y}{\kk, \pi(k, t)})}}{\Delta z_{\iit'}^{t, \pi(k', t)}} + \frac{\Gamma_t}{2} \left( \Delta z_{\iit'}^{t, \pi(k', t)} \right)^2 \\
&\smallspace = \dotprod{\grad{\iit'}{f(\var{y}{\kk, \pi(k, t)})}}{-\frac{1}{\Gamma_{t}} \tilde{g}^{t, \pi(k', t)}_{\iit'}} + \frac{1}{2 \Gamma_t} \left(- \tilde{g}^{t, \pi(k', t)}_{\iit'} \right)^2 \\
&\smallspace = \dotprod{\grad{\iit'}{f(\var{y}{\kk, \pi(k, t)})} - \tilde{g}^{t, \pi(k', t)}_{\iit'}}{-\frac{1}{\Gamma_{t}} \tilde{g}^{t, \pi(k', t)}_{\iit'}} +  \dotprod{\tilde{g}^{t, \pi(k', t)}_{\iit'}}{-\frac{1}{\Gamma_{t}} \tilde{g}^{t, \pi(k', t)}_{\iit'}} + \frac{1}{2\Gamma_t} \left(- \tilde{g}^{t, \pi(k', t)}_{\iit'}\right)^2 \\
& \smallspace \leq \dotprod{\grad{\iit'}{f(\var{y}{\kk, \pi(k, t)})} - \tilde{g}^{t, \pi(k', t)}_{\iit'}}{-\frac{1}{\Gamma_{t}} \tilde{g}^{t, \pi(k', t)}_{\iit'}}  + \dotprod{\tilde{g}^{t, \pi(k', t)}_{\iit'}}{-\frac{1}{\Gamma_{t}} \grad{\iit'}{f(\var{y}{\kk, \pi(k, t)})} } + \frac{1}{2\Gamma_t} \left(-\grad{\iit'}{f(\var{y}{\kk, \pi(k, t)})}  \right)^2 \\
&\largespace - \frac{1}{2 \Gamma_t}\left(\tilde{g}^{t, \pi(k', t)}_{\iit'}  - \grad{\iit'}{f(\var{y}{\kk, \pi(k, t)})}  \right)^2 \\
& \smallspace = \dotprod{\grad{\iit'}{f(\var{y}{\kk, \pi(k, t)})} - \tilde{g}^{t, \pi(k', t)}_{\iit'}}{\frac{1}{\Gamma_t} \grad{\iit'}{f(\var{y}{\kk, \pi(k, t)})}  -\frac{1}{\Gamma_{t}} \tilde{g}^{t, \pi(k', t)}_{\iit'}}  + \dotprod{\grad{\iit'}{f(\var{y}{\kk, \pi(k, t)})} }{-\frac{1}{\Gamma_{t}} \grad{\iit'}{f(\var{y}{\kk, \pi(k, t)})} } \\
&\largespace + \frac{1}{2\Gamma_t} \left(-\grad{\iit'}{f(\var{y}{\kk, \pi(k, t)})}  \right)^2 - \frac{1}{2 \Gamma_t}\left(\tilde{g}^{t, \pi(k', t)}_{\iit'}  - \grad{\iit'}{f(\var{y}{\kk, \pi(k, t)})}  \right)^2 \\
& \smallspace \leq \frac{1}{\Gamma_{t}} \|\grad{\iit'}{f(\var{y}{\kk, \pi(k, t)})} - \tilde{g}^{t, \pi(k', t)}_{\iit'}\|^2 + \dotprod{\grad{\iit'}{f(\var{y}{\kk, \pi(k, t)})} }{x^*_{k'} - \var{w}{\kk, \pi(k', t)}_{k'}} + \frac{\Gamma_t}{2} \left(x^*_{k'} - \var{w}{\kk, \pi(k', t)}_{k'}\right)^2 \\
&\largespace - \frac{1}{2 \Gamma_t}\left(\tilde{g}^{t, \pi(k', t)}_{\iit'}  - \grad{\iit'}{f(\var{y}{\kk, \pi(k, t)})}  \right)^2  - \frac{\Gamma_t}{2 }\left(x^*_{k'} - \var{w}{\kk, \pi(k', t)}_{k'} +\frac{1}{\Gamma_t} \grad{\iit'}{f(\var{y}{\kk, \pi(k, t)})} \right)^2.
\end{align*}

By Lemma~\ref{lem::para1},  for any $\eta_{\kk} > 1$ and any $\veca$ and $\vecb$, $\eta_{\kk} \|\veca - \vecb\|^2 + \|\veca\|^2 \geq \frac{\eta_{\kk}}{\eta_{\kk} + 1} \|\vecb\|^2$. Therefore, putting 
$\veca = x^*_{k'} - \var{w}{\kk, \pi(k', t)}_{k'} +\frac{1}{\Gamma_t} \grad{\iit'}{f(\var{y}{\kk, \pi(k, t)})}$ and 
$\vecb = x^*_{k'} - \var{w}{\kk, \pi(k', t)}_{k'} +\frac{1}{\Gamma_t} \tilde{g}^{t, \pi(k',t)}_{\iit'}$,
 yields
\begin{align*}
&\dotprod{\grad{\iit'}{f(\var{y}{\kk, \pi(k, t)})}}{\Delta z_{\iit'}^{t, \pi(k', t)}} + \frac{\Gamma_t}{2} \left( \Delta z_{\iit'}^{t, \pi(k', t)} \right)^2 
= -\frac{1}{\Gamma_t} \dotprod{\grad{\iit'}{f(\var{y}{\kk, \pi(k, t)})}}{ \tilde{g}^{t, \pi(k, t)}_{k}} +\frac{1}{2\Gamma_t}\left(\tilde{g}^{t, \pi(k, t)}_{k}\right)^2
\\
& \smallspace  \leq  \frac{1}{2\Gamma_{t}}\left(\tilde{g}^{t, \pi(k',t)}_{\iit'} - \grad{\iit'}{f(\var{y}{\kk, \pi(k, t)})} \right)^2
- \frac {\left( \grad{\iit'}{f(\var{y}{\kk, \pi(k, t)})} \right)^2} {2\Gamma_{t} }
\\
& \smallspace  \leq   \frac{1}{2\Gamma_{t}}\left(\tilde{g}^{t, \pi(k',t)}_{\iit'} - \grad{\iit'}{f(\var{y}{\kk, \pi(k, t)})} \right)^2
-  \frac{\left( \grad{\iit'}{f(\var{y}{\kk, \pi(k, t)})} \right)^2 } {2\Gamma_{t} }
 + \frac{\Gamma_t}{2} \|a\|^2 + \frac{ \eta_t \Gamma_{t}} {2} \| a -b \|^2 - \frac { \eta_t \Gamma_t} {2(\eta_t + 1)} \|b\|^2 
\\
& \smallspace \leq  \frac{1}{2\Gamma_{t}} \|\grad{\iit'}{f(\var{y}{\kk, \pi(k, t)})} - \tilde{g}^{t, \pi(k',t)}_{\iit'}\|^2 
-  \frac{\left( \grad{\iit'}{f(\var{y}{\kk, \pi(k, t)})} \right)^2 } {2\Gamma_{t} }\\
&\largespace + \underbrace{\dotprod{\grad{\iit'}{f(\var{y}{\kk, \pi(k, t)})} }{x^*_{k'} - \var{w}{\kk, \pi(k', t)}_{k'}} + \frac{\Gamma_t}{2} \left(x^*_{k'} - \var{w}{\kk, \pi(k', t)}_{k'}\right)^2
+  \frac{\left( \grad{\iit'}{f(\var{y}{\kk, \pi(k, t)})} \right)^2 } {2\Gamma_{t} }}_{= \frac{\Gamma_{t} } {2} \|a\|^2} \\
&\largespace + \frac{\eta_t}{2 \Gamma_t}\underbrace{\left(\tilde{g}^{t, \pi(k',t)}_{\iit'} 
 - \grad{\iit'}{f(\var{y}{\kk, \pi(k, t)})}  \right)^2}_{\|a-b\|^2}  
- \frac{\eta_t \Gamma_t}{2(\eta_t +1)}\underbrace{\left(x^*_{k'} - \var{w}{\kk, \pi(k', t)}_{k'} +\frac{1}{\Gamma_t} \tilde{g}^{t, \pi(k',t)}_{\iit'} \right)^2}_{\|b\|^2}\\
 & \smallspace =  \dotprod{\grad{\iit'}{f(\var{y}{\kk, \pi(k, t)})} }{x^*_{k'} - \var{w}{\kk, \pi(k', t)}_{k'}} + \frac{\Gamma_t}{2} \left(x^*_{k'} - \var{w}{\kk, \pi(k', t)}_{k'}\right)^2 \\
&\largespace + \frac{\eta_t + 1}{2 \Gamma_t}\left(\tilde{g}^{t, \pi(k', t)}_{\iit'}  - \grad{\iit'}{f(\var{y}{\kk, \pi(k, t)})}  \right)^2  - \frac{\eta_t \Gamma_t}{2(\eta_t +1)}\left(x^*_{k'} - \var{w}{\kk, \pi(k', t)}_{k'} +\frac{1}{\Gamma_t} \tilde{g}^{t, \pi(k',t)}_{\iit'} \right)^2.
\end{align*}

Plugging into \eqref{ineq::number::1} gives
\begin{align*}
&\frac{1}{n} \sum_{k} f(\var{x}{\kk+1, \pi(k, t)}) \\
&\smallspace \leq  \frac{1}{n} (1- \aalpha_t) \sum_{k}f(\var{x}{\kk, \pi(k, t)}) \\
&\smallspace\largespace+ \frac{1}{n} \aalpha_t \sum_{k} \Big[ f(\var{y}{\kk, \pi(k, t)})   + \sum_{k'} \Big\langle\grad{\iit'}{f(\var{y}{\kk, \pi(k, t)})}, \var{w}{\kk, \pi(k, t)}_{k'} + x^*_{k'} - \var{w}{\kk, \pi(k', t)}_{k'} - \var{y}{\kk, \pi(k, t)}_{k'} \Big\rangle \\
&\largespace \largespace + \sum_{k'}\frac{\Gamma_t}{2} \left(x^*_{k'} - \var{w}{\kk, \pi(k', t)}_{k'} \right)^2 \Big] \\
&\smallspace \largespace - \sum_{k} \frac{\aalpha_t(\Gamma_t - n \aalpha_t)}{2} \left(\Delta z_{\iit}^{t, \pi(k, t)} \right)^2 + \frac{1}{n} \sum_{k}\frac{1}{n} \sum_{k'} \dotprod{\grad{\iit}{f(\var{y}{\kk, \pi(k, t)})} - \grad{\iit}{f(\var{y}{\kk, \pi(k', t)})}}{\Delta x_{\iit}^{t, \pi(k, t)}}\\
&\smallspace \smallspace + \frac{1}{n} \aalpha_t \sum_{k, k'}\left[\frac{\eta_t + 1}{2 \Gamma_t}\left(\tilde{g}^{t, \pi(k', t)}_{\iit'}  - \grad{\iit'}{f(\var{y}{\kk, \pi(k, t)})}   \right)^2  - \frac{\eta_t \Gamma_t}{2(\eta_t +1)}\left(x^*_{k'} - \var{w}{\kk, \pi(k', t)}_{k'} +\frac{1}{\Gamma_t} \tilde{g}^{t, \pi(k',t)}_{\iit'} \right)^2\right]\\
&\smallspace \leq  \frac{1}{n} (1- \aalpha_t) \sum_{k}f(\var{x}{\kk, \pi(k, t)}) \\
&\smallspace\largespace+ \frac{1}{n} \aalpha_t \sum_{k} \Big[ f(\var{y}{\kk, \pi(k, t)})   + \sum_{k'} \Big\langle\grad{\iit'}{f(\var{y}{\kk, \pi(k, t)})},  x^*_{k'}  - \var{y}{\kk, \pi(k, t)}_{k'} \Big\rangle \\
&\largespace \largespace + \sum_{k'}\frac{\Gamma_t}{2} \left(x^*_{k'} - \var{w}{\kk, \pi(k', t)}_{k'} \right)^2 \Big] \\
&\smallspace \largespace - \sum_{k} \frac{\aalpha_t(\Gamma_t - n \aalpha_t)}{2} \left(\Delta z_{\iit}^{t, \pi(k, t)} \right)^2 + \frac{1}{n} \sum_{k}\frac{1}{n} \sum_{k'} \dotprod{\grad{\iit}{f(\var{y}{\kk, \pi(k, t)})} - \grad{\iit}{f(\var{y}{\kk, \pi(k', t)})}}{\Delta x_{\iit}^{t, \pi(k, t)}}\\
&\smallspace \smallspace + \frac{1}{n} \aalpha_t \sum_{k, k'}\left[\frac{\eta_t + 1}{2 \Gamma_t}\left(\tilde{g}^{t, \pi(k', t)}_{\iit'}  - \grad{\iit'}{f(\var{y}{\kk, \pi(k, t)})}   \right)^2  - \frac{\eta_t \Gamma_t}{2(\eta_t +1)}\left(x^*_{k'} - \var{w}{\kk, \pi(k', t)}_{k'} +\frac{1}{\Gamma_t} \tilde{g}^{t, \pi(k',t)}_{\iit'} \right)^2\right]\\
&\smallspace \smallspace + \frac{1}{n} \aalpha_t \sum_{k, k'}  \dotprod{\grad{\iit'}{f(\var{y}{\kk, \pi(k, t)})}}{\var{w}{\kk, \pi(k, t)}_{k'} - \var{w}{\kk, \pi(k', t)}_{k'}}\\
&\smallspace \leq  \frac{1}{n} (1- \aalpha_t) \sum_{k}f(\var{x}{\kk, \pi(k, t)}) \\
&\smallspace\largespace+ \frac{1}{n} \aalpha_t \sum_{k} \Big[ f(x^*)   - \frac{\mu}{2} \underbrace{\|x^* - \var{y}{\kk, \pi(k, t)}\|^2}_{A}  + \sum_{k'}\frac{\Gamma_t}{2} \underbrace{\left(x^*_{k'} - \var{w}{\kk, \pi(k', t)}_{k'} \right)^2}_{B} \Big] \\
&\smallspace \largespace - \underbrace{\sum_{k} \frac{\aalpha_t(\Gamma_t - n \aalpha_t)}{2} \left(\Delta z_{\iit}^{t, \pi(k, t)} \right)^2}_{C} + \underbrace{\frac{1}{n} \sum_{k}\frac{1}{n} \sum_{k'} \dotprod{\grad{\iit}{f(\var{y}{\kk, \pi(k, t)})} - \grad{\iit}{f(\var{y}{\kk, \pi(k', t)})}}{\Delta x_{\iit}^{t, \pi(k, t)}}}_{D}\\
&\smallspace \smallspace + \frac{1}{n} \aalpha_t \sum_{k, k'}\left[\underbrace{\frac{\eta_t + 1}{2 \Gamma_t}\left(\tilde{g}^{t, \pi(k', t)}_{\iit'}  - \grad{\iit'}{f(\var{y}{\kk, \pi(k, t)})}   \right)^2}_{E}   - \underbrace{\frac{\eta_t \Gamma_t}{2(\eta_t +1)}\left(x^*_{k'} - \var{z}{\kk+1, \pi(k', t)}_{k'}  \right)^2}_{F}\right]\\
&\smallspace \smallspace + \underbrace{\frac{1}{n} \aalpha_t \sum_{k, k'}  \dotprod{\grad{\iit'}{f(\var{y}{\kk, \pi(k, t)})}}{\var{w}{\kk, \pi(k, t)}_{k'} - \var{w}{\kk, \pi(k', t)}_{k'}}}_{G}.
\end{align*}

The labeling of terms is to facilitate matching terms in the
next series of inequalities.

We note the following inequalities: 
\begin{enumerate}
\item If $k \neq k'$, by line 3 of Algorithm~\ref{alg::apcg::asyn::copy},
\begin{align}
\label{eqn::ineq-on-x-diff}
( x^*_{k'} - \var{z}{\kk + 1, \pi(k, t)}_{k'}  )^2 = \left(x^*_{k'} - \var{w}{\kk, \pi(k, t)}_{k'} \right)^2 \leq \varphi_t \left(x^*_{k'} - \var{z}{\kk, \pi(k, t)}_{k'} \right)^2+ (1 - \varphi_t) \left(x^*_{k'} - \var{y}{\kk, \pi(k, t)}_{k'} \right)^2.
\end{align}
\item Otherwise, again by line 3 of Algorithm~\ref{alg::apcg::asyn::copy},
\begin{align*}
\underbrace{\left(x^*_{k'} - \var{w}{\kk, \pi(k', t)}_{k'} \right)^2}_{B} \leq \underbrace{\varphi_t \left(x^*_{k'} - \var{z}{\kk, \pi(k', t)}_{k'} \right)^2}_{H}+ \underbrace{(1 - \varphi_t) \left(x^*_{k'} - \var{y}{\kk, \pi(k', t)}_{k'} \right)^2}_{I}.
\end{align*}
\end{enumerate}

We move term $F$ from the RHS to the LHS and add the following
term to both sides.
\begin{align*}
\frac{ \aalpha_t \eta_t \Gamma_t}{2 (\eta_t + 1)} \sum_{k, k' \neq k} ( x^*_{k'} - \var{z}{\kk + 1, \pi(k, t)}_{k'}  )^2 
& \le
\sum_{k' \neq k} \underbrace{\frac{n \aalpha_t \eta_t \Gamma_t \varphi_t}{2 (\eta_t + 1)} (x^*_{k'} - \var{z}{\kk, \pi(k, t)}_{k'})^2}_{J}\\
& ~~~~~~+ \underbrace{ \aalpha_t \sum_{k' \neq k} \frac{n  \eta_t \Gamma_t (1 - \varphi_t)}{2 (\eta_t + 1)} (x^*_{k'} - \var{y}{\kk, \pi(k, t)}_{k'})^2}_{K} 
~~~\text{(by~\ref{eqn::ineq-on-x-diff})}.
\end{align*}
This yields: 
\begin{align*}
&\frac{1}{n} \sum_{k} \left[ f(\var{x}{\kk+1, \pi(k, t)}) - f(x^*) + \frac{n \aalpha_t \eta_t \Gamma_t}{2 (\eta_t + 1)} \| x^* - \var{z}{\kk+1, \pi(k, t)}\|^2 \right] \\
& \smallspace \leq \frac{1}{n} (1 - \aalpha_t) \sum_{k} \left[ f(\var{x}{\kk, \pi(k, t)}) - f(x^*) + \underbrace{\frac{n \aalpha_t \Gamma_t \varphi_t}{2 (1 - \aalpha_t)} ( x^*_{k} - \var{z}{\kk, \pi(k, t)}_{k} )^2}_{H} + \sum_{k' \neq k} \underbrace{\frac{n \aalpha_t \eta_t \Gamma_t \varphi_t}{2 (\eta_t + 1) (1 - \aalpha_t)} (x^*_{k'} - \var{z}{\kk, \pi(k, t)}_{k'})^2}_{J}\right]\\
& \smallspace \smallspace - \underbrace{\frac{\aalpha_t}{n} \sum_{k} \left[  \underbrace{\frac{\mu}{2} \|x^* - \var{y}{\kk, \pi(k, t)}\|^2}_{A} -   \underbrace{\frac{n \Gamma_t (1 - \varphi_t)}{2} ( x^*_k - \var{y}{\kk, \pi(k, t)}_{k})^2}_{I} - \underbrace{\sum_{k' \neq k} \frac{n  \eta_t \Gamma_t (1 - \varphi_t)}{2 (\eta_t + 1)} (x^*_{k'} - \var{y}{\kk, \pi(k, t)}_{k'})^2}_{K} \right]}_{L}  \\
&\smallspace \smallspace - \underbrace{\sum_{k} \frac{\aalpha_t(\Gamma_t - n \aalpha_t)}{2} \left(\Delta z_{\iit}^{t, \pi(k, t)} \right)^2}_{C} + \underbrace{\frac{1}{n} \sum_{k}\frac{1}{n} \sum_{k'} \dotprod{\grad{\iit}{f(\var{y}{\kk, \pi(k, t)})} - \grad{\iit}{f(\var{y}{\kk, \pi(k', t)})}}{\Delta x_{\iit}^{t, \pi(k, t)}}}_{D}\\
&\smallspace \smallspace + \underbrace{\frac{1}{n} \aalpha_t \sum_{k, k'}\left[\frac{\eta_t + 1}{2 \Gamma_t}\left(\tilde{g}^{t, \pi(k', t)}_{\iit'}  - \grad{\iit'}{f(\var{y}{\kk, \pi(k, t)})}   \right)^2  \right]}_{E}\\
&\smallspace \smallspace + \underbrace{\frac{1}{n} \aalpha_t \sum_{k, k'}  \dotprod{\grad{\iit'}{f(\var{y}{\kk, \pi(k, t)})}}{\var{w}{\kk, \pi(k, t)}_{k'} - \var{w}{\kk, \pi(k', t)}_{k'}}}_{G}.
\end{align*}

Note that the coefficient of $H$ is  bigger than the coefficient of $J$. We intend to move some of $H$ to term $J$. As
\begin{align*}
&\frac{1}{n} \frac{\aalpha_t \Gamma_t \varphi_t}{2  (\eta_t + 1) (1 - \aalpha_t)} ( x^*_{k} - \var{z}{\kk, \pi(k, t)}_{k})^2 \\
&\largespace\leq \frac{2}{n} \frac{\aalpha_t \Gamma_t \varphi_t}{2  (\eta_t + 1) (1 - \aalpha_t)} ( x^*_{k} - \var{z}{\kk, \pi(k', t)}_{k})^2 + \frac{2}{n} \frac{\aalpha_t \Gamma_t \varphi_t}{2  (\eta_t + 1) (1 - \aalpha_t)} ( \var{z}{\kk, \pi(k, t)}_{k} - \var{z}{\kk, \pi(k', t)}_{k})^2,
\end{align*}
we get 
\begin{align*}
&\frac{1}{n} (1 - \aalpha_t) \sum_{k} \left[ \frac{n \aalpha_t \Gamma_t \varphi_t}{2 (1 - \aalpha_t)} ( x^*_{k} - \var{z}{\kk, \pi(k, t)}_{k} )^2 + \sum_{k' \neq k} \frac{n \aalpha_t \eta_t \Gamma_t \varphi_t}{2 (\eta_t + 1) (1 - \aalpha_t)} (x^*_{k'} - \var{z}{\kk, \pi(k, t)}_{k'})^2 \right]\\
&\largespace \leq \underbrace{\frac{1}{n} \sum_{k} (1 - \aalpha_t) \left(\frac{n \aalpha_t \Gamma_t \varphi_t(\eta_t + \frac{2}{n})}{2 (\eta_t +1) (1 - \aalpha_t)}\right) \|x^* - \var{z}{\kk, \pi(k, t)}\|^2}_{M} + \underbrace{\frac{2}{n} \frac{\aalpha_t \Gamma_t \varphi_t}{2  (\eta_t + 1)} \sum_{k, k'} ( \var{z}{\kk, \pi(k, t)}_{k} - \var{z}{\kk, \pi(k', t)}_{k})^2}_{N}.
\end{align*}

By assumption $\frac{n  \Gamma_t (1 - \varphi_t)}{2 } \leq \frac{\mu}{2}$, and also $\frac{n  \eta_t \Gamma_t (1 - \varphi_t)}{2 (\eta_t + 1)} \leq \frac{n  \Gamma_t (1 - \varphi_t)}{2 }$, so term $L$ is non-positive and hence
can be dropped.

Term $D$ is bounded as follows.
\begin{align*}
&\frac{1}{n} \sum_{k}\frac{1}{n} \sum_{k'} \dotprod{\grad{\iit}{f(\var{y}{\kk, \pi(k, t)})} - \grad{\iit}{f(\var{y}{\kk, \pi(k', t)})}}{\Delta x_{\iit}^{t, \pi(k, t)}}  \\
&\largespace = \frac{1}{n^2} \sum_{k, k'} n \aalpha_t \dotprod{\grad{\iit}{f(\var{y}{\kk, \pi(k, t)})} - \grad{\iit}{f(\var{y}{\kk, \pi(k', t)})}}{\Delta z_{\iit}^{t, \pi(k, t)}}\\
&\largespace \leq \frac{1}{2 n} \sum_{k, k'} 
\aalpha_t \left[
\underbrace{\frac{10}{\Gamma_t}(\grad{\iit}{f(\var{y}{\kk, \pi(k, t)})} - \grad{\iit}{f(\var{y}{\kk, \pi(k', t)})})^2}_{O}  
+ \underbrace{\frac{\Gamma_t}{10} \left(\Delta z_{\iit}^{t, \pi(k, t)}\right)^2}_{P} 
\right].
\end{align*}

Term $G$ is bounded as follows.
\begin{align*}
&\frac{1}{n} \aalpha_t \sum_{k, k'}  \dotprod{\grad{\iit'}{f(\var{y}{\kk, \pi(k, t)})}}{\var{w}{\kk, \pi(k, t)}_{k'} - \var{w}{\kk, \pi(k', t)}_{k'}} \\
&\largespace = \frac{1}{n} \aalpha_t \sum_{k, k'}  \dotprod{ \grad{\iit'}{f(\var{y}{\kk, \pi(k, t)})} - \tilde{g}^{t, \pi(k', t)}_{\iit'}}{\var{w}{\kk, \pi(k, t)}_{k'} - \var{w}{\kk, \pi(k', t)}_{k'}} \\
&\largespace \smallspace  + \frac{1}{n} \aalpha_t \sum_{k, k'}  \dotprod{- \Gamma_t \Delta z_{\iit'}^{t, \pi(k', t)} }{\var{w}{\kk, \pi(k, t)}_{k'} - \var{w}{\kk, \pi(k', t)}_{k'}} \\
&\largespace \leq \frac{1}{2n} \aalpha_t \sum_{k, k'}  \left[\underbrace{\frac{1}{\Gamma_t} \|\grad{\iit'}{f(\var{y}{\kk, \pi(k, t)})} - \tilde{g}^{t, \pi(k', t)}_{\iit'}\|^2}_{Q} + \underbrace{\frac{\Gamma_t}{1} \|\var{w}{\kk, \pi(k, t)}_{k'} - \var{w}{\kk, \pi(k', t)}_{k'}\|^2}_{R}\right] \\
&\largespace \smallspace+ \frac{1}{2n} \aalpha_t \sum_{k, k'} \left[\underbrace{\frac{\Gamma_t}{10} \|\Delta z_{\iit'}^{t, \pi(k', t)}\|^2}_{S} + \underbrace{10 \Gamma_t \|\var{w}{\kk, \pi(k, t)}_{k'} - \var{w}{\kk, \pi(k', t)}_{k'}\|^2}_{T}\right].
\end{align*}

The coefficient of term $M$ is bounded as follows.\\
Recall that by assumption its coefficient $(1 - \aalpha_t) \left(\frac{n \aalpha_t \Gamma_t \varphi_t(\eta_t + \frac{2}{n})}{2 (\eta_t +1) (1 - \aalpha_t)}\right) \leq (1 - \tau \aalpha_t) \left(\frac{n \aalpha_{t-1} \eta_{t-1} \Gamma_{t-1}}{2 (\eta_{t-1} + 1)}\right)= (1 - \tau \aalpha_t)\zeta_k$, by the definition of $\zeta_t$.

Then, with the underbraces indicating matching terms, we obtain
\begin{align*}
&\frac{1}{n} \sum_{k} \left[ f(\var{x}{\kk+1, \pi(k, t)}) - f(x^*) + \zeta_{t+1} \| x^* - \var{z}{\kk+1, \pi(k, t)}\|^2 \right]\\
& \smallspace \leq \frac{1}{n} (1 - \tau \aalpha_t) \sum_{k} \left[ f(\var{x}{\kk, \pi(k, t)}) - f(x^*)  + \underbrace{\zeta_{t} \| x^* - \var{z}{\kk, \pi(k, t)}\|^2}_{M} \right]\\
&\smallspace \smallspace - \underbrace{\sum_{k} \frac{\aalpha_t(\frac{4}{5}\Gamma_t - n \aalpha_t)}{2} \left(\Delta z_{\iit}^{t, \pi(k, t)}  \right)^2}_{-C+P+S} \\
&\smallspace \smallspace + \frac{1}{n} \aalpha_t \sum_{k, k'}\left[
\underbrace{\frac{\eta_t + 2}{2 \Gamma_t}\left(\grad{\iit'}{f(\var{y}{\kk, \pi(k, t)})} - \tilde{g}^{t, \pi(k', t)}_{\iit'}\right)^2}_{E+Q}
+ \underbrace{\frac{10}{2 \Gamma_t} \left(\grad{\iit}{f(\var{y}{\kk, \pi(k, t)})} - \grad{\iit}{f(\var{y}{\kk, \pi(k', t)})}\right)^2}_{O}  \right]  \\
&\smallspace \smallspace + \underbrace{\frac{6}{2n} \Gamma_t \aalpha_t \sum_{k, k'} \|\var{w}{\kk, \pi(k, t)}_{k'} - \var{w}{\kk, \pi(k', t)}_{k'}\|^2}_{R+T} \\
&\smallspace \smallspace + \underbrace{\frac{2}{n} \frac{\aalpha_t \Gamma_t \varphi_t}{2  (\eta_t + 1)} \sum_{k, k'} \| \var{z}{\kk, \pi(k, t)}_{k} - \var{z}{\kk, \pi(k', t)}_{k}\|^2}_{N}.
\end{align*}
\end{proof}

\pfof{Lemma~\ref{lem::apcg::error::non::str}}
We will apply Lemma ~\ref{apcg::lem::similar::proof}.

\textbf{Strongly convex case:}

In this case, we set $\aalpha_{\kk} = \phi$, $\bbeta_{\kk} = (1 - \phi)$, $\ggamma_{\kk} = \frac{1}{1 + \phi}$, $\eta_{\kk} + 1 = \frac{3}{n \phi (1 - \parb)}$, and let $\Gamma_t$  be a constant series.
Let's check the constraints of Lemma~\ref{apcg::lem::similar::proof} one by one.

\smallskip
\noindent
\emph{Constraint} 1. $\eta_{\kk} > 1$.\\
This constraint holds since $\phi \leq \frac{1}{n}$ by assumption.
\begin{align*}
\text{\emph{Constraint} 2}.~~~\frac{ n\aalpha_{\kk}\ggamma_{\kk}\bbeta_{\kk}}{1 - \ggamma_{\kk}} = n(1 - \aalpha_{\kk}).\hspace*{5in}
\end{align*}
Substituting yields
\begin{align*}
\frac{\phi \frac{1}{1 + \phi} (1 - \phi)}{1 - \frac{1}{1 + \phi}} = 1 - \phi.
\end{align*}
\begin{align*}
\text{\emph{Constraint} 3}.~~~~~(1 - \aalpha_t) \left(\frac{n \aalpha_t \Gamma_t \varphi_t(\eta_t + \frac{2}{n})}{2 (\eta_t +1) (1 - \aalpha_t)}\right) \leq (1 - \tau \aalpha_t) \left(\frac{n \aalpha_{t-1} \eta_{t-1} \Gamma_{t-1}}{2 (\eta_{t-1} + 1)}\right).
\hspace*{2in}
\end{align*}
Substituting yields
\begin{align*}
(1 - \phi) \left(1 - \frac{n-2}{3}\phi(1 - \tau)\right) \leq (1 - \tau \phi)  \left(1 - \frac{n}{3}\phi(1 - \tau)\right).
\end{align*}
This is true since $\phi \leq \frac{1}{n}$.

\begin{align*}\text{Constraint 4}.~~~~\frac{n  \Gamma_t (1 - \varphi_t)}{2 } \leq \frac{\mu}{2}.\hspace*{5in}
\end{align*}
Substituting gives
\begin{align*}
n \phi \Gamma \leq \mu.
\end{align*}
which is constraint (ii) of Lemma~\ref{app::progress::lemma}.

\smallskip
\noindent
\textbf{Non-strongly convex case:}

In this case, we set $\phi_t = \frac{2}{t_0 + t}$ for $t_0 \geq 2n$, $\varphi_t = 1$, $\psi_t = \frac{t_0 + t - 2}{t_0 + t}$, $\eta_{\kk} + 1 = \frac{3}{n \aalpha_{\kk} (1 - \tilde{\parb} - \frac{1}{4 t_0})}$, and let $\Gamma_t$ be a series such that $\left(1 - \tilde{\parb}\aalpha_{\kk}\right) \frac{\Gamma_\kk}{n \aalpha_\kk} \leq \left(1 - \parb \aalpha_{\kk}\right) \frac{\Gamma_{\kk - 1}}{n \aalpha_{\kk - 1}}$. 

\smallskip
\noindent
\emph{Constraint} 1. $\eta_{\kk} > 1$.\\
This is easy to verify since $\eta_{\kk} + 1 = \frac{3}{n \aalpha_{\kk} (1 - \tilde{\parb} - \frac{1}{4 t_0})}$.
\begin{align*}
\text{\emph{Constraint} 2}.~~~\frac{ n\aalpha_{\kk}\ggamma_{\kk}\bbeta_{\kk}}{1 - \ggamma_{\kk}} = n(1 - \aalpha_{\kk}).\hspace*{5in}
\end{align*}
A simple calculation shows that this holds.
\begin{align*}
\text{\emph{Constraint} 3}.~~~~~(1 - \aalpha_t) \left(\frac{n \aalpha_t \Gamma_t \varphi_t(\eta_t + \frac{2}{n})}{2 (\eta_t +1) (1 - \aalpha_t)}\right) \leq (1 - \tau \aalpha_t) \left(\frac{n \aalpha_{t-1} \eta_{t-1} \Gamma_{t-1}}{2 (\eta_{t-1} + 1)}\right).
\hspace*{2in}
\end{align*}

The inequality is equivalent to 
\begin{align*} \label{apcg::para::choice::obj}
\frac{\phi_t^2}{\phi_{t-1}^2} \frac{\Gamma_{\kk}}{n \phi_{\kk}}\left(1 - \frac{n-2}{n (\eta_t + 1)}\right) \leq (1 - \parb \aalpha_{\kk})\left(1 - \frac{1}{\eta_{\kk - 1} + 1}\right)\frac{\Gamma_{\kk - 1}}{n \aalpha_{\kk - 1}}. \numberthis
\end{align*}
Since $\phi_t = \frac{2}{t_0 + t}$, $\frac{\phi_t^2}{\phi_{t-1}^2} = \left(1 - \phi_t + \frac{\phi_t^2}{4}\right) $
With some further calculation, given in the footnote \footnote{
In order to prove this, we only need $1 - \frac{n-2}{n(\eta_{\kk} + 1)} - (\aalpha_{\kk} - \frac{\aalpha_{\kk}^2}{4}) + \frac{n-2}{n} \cdot \frac{(\aalpha_{\kk} - \frac{\aalpha_{\kk}^2}{4})}{(\eta_{\kk} + 1)} \leq 1 - \tilde{\parb} \aalpha_{\kk}  - \frac{1}{\eta_{\kk - 1} + 1} + \frac{\tilde{\parb}\aalpha_{\kk}}{\eta_{\kk - 1} + 1}$ and this is equivalent to $ \frac{1}{\eta_{\kk - 1} + 1} - \frac{n-2}{n} \cdot \frac{1}{\eta_{\kk} + 1} + \frac{n - 2}{n}\frac{\aalpha_{\kk}(1 - \frac{\aalpha_{\kk}}{4}) }{\eta_{\kk} + 1} - \frac{\tilde{\parb} \aalpha_{\kk}}{\eta_{\kk  - 1}+1} \leq (1 - \tilde{\parb} - \frac{\aalpha_t}{4}) \aalpha_{\kk}$. Using the equality $\eta_{\kk} + 1 = \frac{3}{n \aalpha_{\kk} (1 - \tilde{\parb} - \frac{1}{4 t_0})}$, if suffices to have
\begin{align*}
\frac{n \aalpha_{\kk - 1} (1 - \tilde{\parb}- \frac{1}{4 t_0})}{3} - \frac{n-2}{n} \frac{n \aalpha_{\kk } (1 - \tilde{\parb} - \frac{1}{4 t_0})}{3} +\frac{n-2}{n} \frac{n \aalpha_{\kk}^2 (1 - \tilde{\parb} - \frac{1}{4 t_0})}{3} - \tilde{\parb} \aalpha_{\kk} \frac{n \aalpha_{\kk - 1} (1 - \tilde{\parb}- \frac{1}{4 t_0})}{3} \leq (1 - \tilde{\parb} - \frac{1}{4 t_0})\aalpha_{\kk}.
\end{align*}
Or equivalently,
\begin{align*}
n \aalpha_{t-1} - (n -2) \aalpha_t + (n-2) \aalpha_t^2 - n \tilde{\tau} \aalpha_t \aalpha_{t-1} \leq 3 \aalpha_t.
\end{align*}
Rearranging terms yields
\begin{align*}
n \aalpha_{t-1} \leq (n + 1) \aalpha_t +  n \tilde{\tau} \aalpha_t \aalpha_{t-1} -  (n-2) \aalpha_t^2.
\end{align*}
Since $\frac{\aalpha_{\kk}}{\aalpha_{\kk - 1}}  \geq 1 - \frac{1}{2n}$ if $t_0 \geq 2n$, 
\begin{align*}
n \aalpha_{t-1} \leq \left(n + \frac{n}{2n - 1}\right) \aalpha_t.
\end{align*}
So the last thing to do is to prove $\left(n + \frac{n}{2n - 1}\right) \aalpha_t \leq (n + 1) \aalpha_t +  n \tilde{\tau} \aalpha_t \aalpha_{t-1} -  (n-2) \aalpha_t^2$, and this is equivalent to 
\begin{align*}
n + \frac{n}{2n - 1} \leq n + 1 - (n - 2) \aalpha_t + n \tilde{\tau} \aalpha_{t-1}.
\end{align*}
This is equivalent to 
\begin{align*}
(n - 2) \aalpha_t - n \tilde{\tau} \aalpha_{t - 1} \leq \frac{n-1}{2n - 1}.
\end{align*}
This inequality is easy to verify as $\frac{1}{n} \geq \aalpha_{t-1} \geq \aalpha_t$ and $\tilde{\tau} \geq \frac{1}{2}$.
}
we obtain
\begin{align*} \label{apcg::para::choice::1}
\left(1 - \frac{n-2}{n(\eta_{\kk}+1)}\right) (1 - \phi_{\kk} + \frac{\phi_t^2}{4}) \leq \left(1 - \tilde{\parb}\aalpha_{\kk}\right)\left(1 - \frac{1}{ \eta_{\kk - 1} + 1}\right). \numberthis
\end{align*}

By constraint (v) of Lemma~\ref{lem::apcg::error::non::str},  $\left(1 - \tilde{\parb}\aalpha_{\kk}\right) \frac{\Gamma_\kk}{n \aalpha_\kk} \leq \left(1 - \parb \aalpha_{\kk}\right) \frac{\Gamma_{\kk - 1}}{n \aalpha_{\kk - 1}}$; multiplying  \eqref{apcg::para::choice::1}
by $\frac {\Gamma_t}{n\aalpha_t}$, implies \eqref{apcg::para::choice::obj}.

\begin{align*}\text{Constraint 4}.~~~~\frac{n  \Gamma_t (1 - \varphi_t)}{2 } \leq \frac{\mu}{2}.\hspace*{5in}
\end{align*}
This holds as $\varphi_t = 1$ in this case.
\end{proof}

%% file: app-FE-Delta-bounds.tex
\section{Proof of Lemma~\ref{lem:FEandDeltaBounds},
the Amortization Bounds}
\label{Appendix::C}

\begin{lemma}
\label{lem:FEandDeltaBounds}
Let $I=[0,T-1]$. Suppose $r  = \max_t \left\{ \frac{36 (3q)^2 L_{\overline{res}}^2 n^2 \xi \phi_{t}^2}{ \Gamma_{t}^2 n}\right\} {< 1}$, $\xi = \max_t \max_{s \in [ t - 3q, t + q]} \frac{\phi_{s}^2 \Gamma_{t}^2}{\phi_{t}^2 \Gamma_{s}^2}$, $\frac{{224}q^2 L_{\overline{res}}^2 }{n} \leq 1$, and  $\var{B}{t}$ are good,   then:
\begin{align*}
(\Delta_{\kk}^{\tFE})^2 &\leq \frac{1}{\Gamma_t^2} \E{\pi}{(g_{\max, \iit_{\kk}}^{\kk, \pi} - g_{\min, \iit_{\kk}}^{\kk, \pi})^2}; \numberthis \label{ineq::tran::1} \\
\E{\pi}{(g_{\max, \iit_{\kk}}^{\kk, \pi} - g_{\min, \iit_{\kk}}^{\kk, \pi})^2} &\leq \frac{54  q \Lresbar^2 n^2 \phi_{\kk}^2}{n} \sum_{s \in I \cap [\kk - 2q, \kk + 2q] \setminus \{\kk\}} \Big[ (\Delta_{s}^{\tFE})^2+ E_s^{\Delta}\Big]; \numberthis \label{ineq::tran::2}\\
\E{\pi}{\sum_{k'} \left(\var{w}{\kk, \pi}_{k'} - \var{w}{\kk, \pi(k', t)}_{k'}\right)^2} &\leq  16q \sum_{s \in I \cap [\kk - q - 1, \kk - 1]} (\Delta_{s}^{\tFE})^2; \numberthis \label{ineq::tran::3}\\
\E{\pi}{\sum_{k'} \left(\var{z}{\kk, \pi}_{k'} - \var{z}{\kk, \pi(k', t)}_{k'}\right)^2} &\leq  16q \sum_{s \in I \cap [\kk - q - 1, \kk - 1]} (\Delta_{s}^{\tFE})^2; \numberthis \label{ineq::tran::4}\\
\E{\pi}{\sum_{k'} \left(\grad{\iit'}{f(\var{y}{\kk, \pi})} - \tilde{g}^{t, \pi(k', t)}_{\iit'}\right)^2} & \leq 9 n^2 \phi_t^2 q L_{\overline{res}}^2 \sum_{s \in I \cap [t - q, t - 1]} (\Delta_{s}^{\tFE})^2 \\
&\largespace + {\frac{24r}{1-r}} n  n^2 \phi_t^2  r (\Delta_{t}^{\tFE})^2 + {\frac{16r}{1-r}} n n^2 \phi_t^2 r  E^{\Delta}_t \\
&\largespace + {\frac{r}{27(1-r)}} \frac{n n^2 \phi_t^2}{q}  \sum_{s \in  I \cap [t - 3q , t +q] \setminus \{t\}} \left((\Delta_{s}^{\tFE})^2 + E^{\Delta}_s \right)\\
&\largespace+ 4 n  \E{\pi}{(g_{\max, \iit_{\kk}}^{\pi, \kk} - g_{\min, \iit_{\kk}}^{\pi, \kk})^2}; \numberthis \label{ineq::tran::5} \\
\E{\pi}{\sum_{k'} \left(\grad{k'}{f(\var{y}{\kk, \pi})} - \grad{k'}{f(\var{y}{\kk, \pi(k', t)})}\right)^2} &\leq \frac{9}{2}q \Lresbar^2  n^2 \phi_t^2  \sum_{s \in I \cap [t - q - 1, t -1]}(\Delta_{s}^{\tFE})^2  \\
&\largespace + {\frac{6r}{1-r}} n  n^2 \phi_t^2   (\Delta_{t}^{\tFE})^2 + {\frac{4r}{1-r}} n n^2 \phi_t^2  E^{\Delta}_t \\
&\largespace + {\frac{r}{108(1-r)}} \frac{ n n^2 \phi_t^2}{q}   \sum_{s \in I \cap [t - 3q , t +q] \setminus \{t\}} \left((\Delta_{s}^{\tFE})^2 + E^{\Delta}_s \right). \numberthis\label{ineq::tran::6}
\end{align*}
\end{lemma}

We introduce the following notation.

$\Delta_{\max}^{u, R} z_{\iit_t}^{t, \pi}$ will denote the maximum value that $\Delta z_{\iit_t}^{t, \pi}$ can attain  when the first $u-q-1$ updates on path $\pi$ have been fixed, assuming the update happens at coordinate $\iit_{\kk}$, and it does not read any of the updates at times in $R$, nor any of the variables updated at time $v > u+q$. Here, $R$ is either $\emptyset$ or $\{\kk\}$. Let $\Delta_{\min}^{u, R} z_{\iit_t}^{t, \pi} $ denote the analogous minimum value.

\begin{align*}
\text{Let}~~~~~~~~~~~\overline{\Delta}_{\max} z_{\iit_t}^{t, \pi}  =  \max_{u\in[\kk-q,\kk]} \Delta_{\max}^{u, \emptyset} z_{\iit_t}^{t, \pi}
~~~~~~~~~~~\text{and}~~~~~~~~~~~
\overline{\Delta}_{\min} z_{\iit_t}^{t, \pi}  =  \max_{u\in[\kk-q,\kk]} \Delta_{\min}^{u, \emptyset} z_{\iit_t}^{t, \pi}.
\end{align*}

Let $g_{\max, \iit_{\kk}}^{\kk, \pi}$ (and $g_{\min, \iit_{\kk}}^{\kk, \pi}$) denote the maximum (and minimum) gradient with the same constraints as $\overline{\Delta}_{\max} z_{\iit_t}^{t, \pi}$ (and $\overline{\Delta}_{\min} z_{\iit_t}^{t, \pi}$).

Note that $\Delta z_{\iit_t}^{t, \pi}  = \arg \min_{h} \{ \frac{\Gamma_k }{2} \| h \|^2 + \langle \tilde{g}^{\kk, \pi}_{\iit_{\kk}} , h \rangle  \}$ (see Step 7 of the asynchronous version of Algorithm~\ref{alg::apcg}). So,
\begin{align*}
(\overline{\Delta}_{\max} z_{\iit_t}^{\kk, \pi}  - \overline{\Delta}_{\min} z_{\iit_t}^{\kk, \pi} )^2 \leq  \frac{1}{\Gamma_t^2} (g_{\max, \iit_{\kk}}^{\kk, \pi} - g_{\min, \iit_{\kk}}^{\kk, \pi})^2. \numberthis \label{apcg::ineq::delta::to::grad}
\end{align*}

Let $(\Delta_{\kk}^{\tFE})^2$ denote the resulting expectation at time $\kk$:
\begin{align*}
(\Delta_{\kk}^{\tFE})^2 \triangleq \E{\pi}{\left(\overline{\Delta}_{\max} \var{z}{t, \pi}_{k_t}- \overline{\Delta}_{\min} \var{z}{t, \pi}_{k_t}\right)^2}.
\end{align*}
Also, let $(E_t^{\Delta}) \triangleq \E{}{(\Delta \var{z}{t, \pi}_{k_t})^2}$. 

\subsection{Proof of Lemma~\ref{lem:FEandDeltaBounds}, Equation~\eqref{ineq::tran::2}}
\label{Appendix::C1}
In order to bound the difference of the gradient on the RHS of \eqref{apcg::ineq::delta::to::grad}, we first bound the possible difference on $\var{\tilde{y}}{t}$, on which the algorithm calculates the gradient. Suppose that $\kk - q \leq \kk_1 \leq \kk$ (and $\kk - q \leq \kk_2 \leq \kk$), then we define $\left[\var{\tilde{y}}{\kk}\right]^{\pi,R,\kk_1,\iit_{\kk}}$ (and $\left[\var{\tilde{y}}{\kk}\right]^{\pi,R,\kk_2,\iit_{\kk}}$) to be some $\var{\tilde{y}}{\kk}$ when the first $\kk_1-q -1$ (resp. $\kk_2 - q - 1$) updates on path $\pi$ have been fixed, assuming the update happens at coordinate $\iit_{\kk}$, and it does not read any of the updates at times in $R$, nor any of the variables updated at time $v > \kk_1+q$ (resp. $ v > \kk_2 + q$).

\begin{lemma}\label{lem::diff::y::cord}
If the $\var{B}{t}$ are good then
\begin{align*}
&\Big|\Big[\left[\var{\tilde{y}}{\kk}\right]^{\pi,R,\kk_1,\iit_{\kk}} - \left[\var{\tilde{y}}{\kk}\right]^{\pi,R,\kk_2,\iit_{\kk}}\Big]_{\iit}\Big| \\
&~~~~~~\leq 3 n \phi_{\kk} \sum_{ \iit_s = \iit\text{ and }s \in \{[\kk - 2q, \kk + q]\setminus (R \cup \{\kk\})\}}   \max\Bigg\{\Big|\max_{l \in [ \max \{s-q, \kk-q\}, \min \{s, \kk\}] \cup \{s\}}\{\Delta_{\max}^{l, R \cup \{\kk\}} z_{\iit_s}^{s, \pi} \} \\
&~~~~~~~~~~~~~~~~~~~~~~~~~~~~~~~~~~~~~~~~~~~~~~~~~~~~~~~~~~~~~~~~~~~~-\min_{l \in [ \max \{s-q, \kk-q\}, \min \{s, \kk\}] \cup \{s\}} \{\Delta_{\min}^{l, R \cup \{\kk\}} z_{\iit_s}^{s, \pi}\}\Big|, \\
&~~~~~~~~~~~~~~~~~~~~~~~~~~~~~~~~~~ \Big|\max_{l \in [ \max \{s-q, \kk-q\}, \min \{s, \kk\}] \cup \{s\}}\{\Delta_{\max}^{l, R \cup \{\kk\}} z_{\iit_s}^{s, \pi}\}\Big|, \\
&~~~~~~~~~~~~~~~~~~~~~~~~~~~~~~~~~~ \Big|\min_{l \in [ \max \{s-q, \kk-q\}, \min \{s, \kk\} ]\cup \{s\}} \{\Delta_{\min}^{l, R \cup \{\kk\}} z_{\iit_s}^{s, \pi}\}\Big|\Bigg\}. \numberthis \label{lem::ineq::diff::y::cord}
\end{align*}
\end{lemma}
\vspace{5mm}
Recall the definition of matrix $L$:
\begin{align*}
(\nabla_{\iit_{\kk}} f(x) - \nabla_{\iit_{\kk}} f(x'))^2 \leq \Big(\sum_k L_{\iit, \iit_{\kk}} |x_{\iit} - x'_{\iit}|\Big)^2.
\end{align*}

Using Lemma~\ref{lem::diff::y::cord} and the Cauchy-Schwarz inequality yields
\begin{align*}
&(g_{\max, \iit_{\kk}}^{\pi, \kk} - g_{\min, \iit_{\kk}}^{\pi,  \kk})^2  \\
& ~~~~~~ \leq 3q \cdot 9 n^2 \phi_{\kk}^2 \sum_{s \in [\kk - 2q, \kk + q]\setminus \{k\}} L^2_{\iit_s, \iit_{\kk}} \max\Bigg\{\Big(\max_{l \in [ \max \{s-q, \kk-q\}, \min \{s, \kk\}] \cup \{s\}}\{\Delta_{\max}^{l,  \{\kk\}} z_{\iit_s}^{s, \pi} \} \\
&~~~~~~~~~~~~~~~~~~~~~~~~~~~~~~~~~~~~~~~~~~~~~~~~~~~~~~~~~~~~~~~~~~~~-\min_{l \in [ \max \{s-q, \kk-q\}, \min \{s, \kk\}] \cup \{s\}} \{\Delta_{\min}^{l, \{\kk\}} z_{\iit_s}^{s, \pi}\}\Big)^2, \\
&~~~~~~~~~~~~~~~~~~~~~~~~~~~~~~~~~~~~~~~~~~~~~~~~~~~~~~~~~~~~~~~~~~~~\Big(\max_{l \in [ \max \{s-q, \kk-q\}, \min \{s, \kk\}] \cup \{s\}}\{\Delta_{\max}^{l,  \{\kk\}} z_{\iit_s}^{s, \pi}\}\Big)^2, \\
&~~~~~~~~~~~~~~~~~~~~~~~~~~~~~~~~~~~~~~~~~~~~~~~~~~~~~~~~~~~~~~~~~~~~ \Big(\min_{l \in [ \max \{s-q, \kk-q\}, \min \{s, \kk\}] \cup \{s\}} \{\Delta_{\min}^{l,  \{\kk\}} z_{\iit_s}^{s, \pi}\}\Big)^2\Bigg\}\\
&~~~~~~~~~~~~~~~~~~~~~~~~~~~\mbox{as there are at most $3q$ terms on the RHS of \eqref{lem::ineq::diff::y::cord})}. 
\end{align*}

Taking the average over $\pi(k, t)$ yields
\begin{align*}
&\mathbb{E}_{k} \left[(g_{\max, \iit}^{\pi(k, t), \kk} - g_{\min, \iit}^{\pi(k, t),  \kk})^2 \right] \\
& ~~~~~~ \leq 3q \cdot 9 n^2 \phi_{\kk}^2 \sum_{s \in [\kk - 2q, \kk + q]\setminus \{k\}} \frac{L^2_{res}}{n} \max\Bigg\{\Big(\max_{l \in [ \max \{s-q, \kk-q\}, \min \{s, \kk\}] \cup \{s\}}\{\Delta_{\max}^{l, \{\kk\}} z_{\iit_s}^{s, \pi} \} \\
&~~~~~~~~~~~~~~~~~~~~~~~~~~~~~~~~~~~~~~~~~~~~~~~~~~~~~~~~~~~~~~~~~~~~-\min_{l \in [ \max \{s-q, \kk-q\}, \min \{s, \kk\}] \cup \{s\}} \{\Delta_{\min}^{l, \{\kk\}} z_{\iit_s}^{s, \pi}\}\Big)^2, \\
&~~~~~~~~~~~~~~~~~~~~~~~~~~~~~~~~~~~~~~~~~~~~~~~~~~~~~~~~~~~~~~~~~~~~\Big(\max_{l \in [ \max \{s-q, \kk-q\}, \min \{s, \kk\}] \cup \{s\}}\{\Delta_{\max}^{l, \{\kk\}} z_{\iit_s}^{s, \pi}\}\Big)^2, \\
&~~~~~~~~~~~~~~~~~~~~~~~~~~~~~~~~~~ ~~~~~~~~~~~~~~~~~~~~~~~~~~~~~~~~~~\Big(\min_{l \in [ \max \{s-q, \kk-q\}, \min \{s, \kk\}] \cup \{s\}} \{\Delta_{\min}^{l,  \{\kk\}} z_{\iit_s}^{s, \pi}\}\Big)^2\Bigg\}. 
\end{align*}
This is legitimate because we exclude the update $t$ on the right hand side, and also $\pi$ will be equal to $\pi(k, t)$ for times other than $t$. Therefore,
\begin{align*}
\mathbb{E}_{k} \left[(g_{\max, \iit}^{\pi(k, t), \kk} - g_{\min, \iit}^{\pi(k, t),  \kk})^2 \right]  \leq 3q \cdot 9 n^2 \phi_{\kk}^2\sum_{s \in [\kk - 2q, \kk + q]\setminus \{k\}} \frac{L^2_{res}}{n} \left[2 ( \Delta z_{\iit_s}^{s, \pi})^2 +2  (\overline{\Delta}_{\max} z_{\iit_s}^{s, \pi}  - \overline{\Delta}_{\min} z_{\iit_s}^{s, \pi} )^2 \right], \numberthis \label{ineq::grad::to::delta}
\end{align*}
as $\Delta z_{\iit_s}^{s, \pi} \in \left[\overline{\Delta}_{\min} z_{\iit_s}^{s, \pi}, \overline{\Delta}_{\max} z_{\iit_s}^{s, \pi}\right]$. 

By the definition of $(\Delta_{\kk}^{\tFE})^2$ and $E_t^{\Delta}$, the result follows.

\subsection{Proof of Lemma~\ref{lem:FEandDeltaBounds}, Equations~\eqref{ineq::tran::3} and \eqref{ineq::tran::4}}

Next, we show that $\left(\var{w}{\kk, \pi}_{k'} - \var{w}{\kk, \pi(k', t)}_{k'}\right)^2 $  and  $\left(\var{z}{\kk, \pi}_{k'} - \var{z}{\kk, \pi(k', t)}_{k'}\right)^2 $ can be upper bounded by  terms of the form $(\overline{\Delta}_{\min} z_{\iit_s}^{s, \pi} - \overline{\Delta}_{\max} z_{\iit_s}^{s, \pi})^2$. 
\begin{lemma} \label{lem::diff::w::cord}
If the $\var{B}{t}$ are good then
\begin{align*}
&\sum_{k'} \left(\var{w}{\kk, \pi}_{k'} - \var{w}{\kk, \pi(k', t)}_{k'}\right)^2, \sum_{k'} \left(\var{z}{\kk, \pi}_{k'} - \var{z
}{\kk, \pi(k', t)}_{k'}\right)^2 \\
&\largespace \leq 8q \sum_{l \in [\kk - q , t - 1] } \left[\left(\overline{\Delta}_{\min} \var{z}{l, \pi}_{k_l}- \overline{\Delta}_{\max} \var{z}{l, \pi}_{k_l}\right)^2 + \left(\overline{\Delta}_{\min} \var{z}{l, \pi(k', t)}_{k_l} - \overline{\Delta}_{\max} \var{z}{l, \pi(k', t)}_{k_l}\right)^2 \right]. \numberthis \label{ineq::w::to::delta}
\end{align*}
\end{lemma}

\subsection{Proof of Lemma~\ref{lem:FEandDeltaBounds}, Equations~\eqref{ineq::tran::5} and \eqref{ineq::tran::6}}

Finally, we want to bound $\left(\grad{\iit'}{f(\var{y}{\kk, \pi})} - \tilde{g}^{t, \pi(k', t)}_{\iit'}\right)^2$.
We observe: 
\begin{obs}\label{obs::diff::grad::y}
\begin{align*}
&\sum_{k'} \left(\grad{\iit'}{f(\var{y}{\kk, \pi})} - \tilde{g}^{t, \pi(k', t)}_{\iit'}\right)^2 \\
& \largespace \leq 2 \sum_{k'}\left(\grad{\iit'}{f(\var{y}{\kk, \pi})} - \grad{\iit'}{f(\var{y}{\kk, \pi(k', t)})}\right)^2 + 2  \sum_{k'}\left(\grad{\iit'}{f(\var{y}{\kk, \pi(k', t)})} - \tilde{g}^{t, \pi(k', t)}_{\iit'}\right)^2.
\end{align*}
\end{obs}

\hide{The second term is easily bounded:
\begin{align*}
\left(\grad{\iit'}{f(\var{y}{\kk, \pi(k', t)})} - \tilde{g}^{t, \pi(k', t)}_{\iit'}\right)^2 \leq \left(g^{t, \pi(k', t)}_{\max, k'} - g^{t, \pi(k', t)}_{\min, k'}\right)^2.
\end{align*}
}

The following lemma gives a bound on the first term.
\begin{lemma} \label{lem::diff::true::y::cord}
Suppose $r  = \max_t \left\{ \frac{{36} {(3q)^2} L_{\overline{res}}^2  n^2 \xi \phi_{t}^2}{ \Gamma_{t}^2 n}\right\} {< 1}$, $\xi = \max_t \max_{s \in [ t - 3q, t + q]} \frac{\phi_{s}^2 \Gamma_{t}^2}{\phi_{t}^2 \Gamma_{s}^2}$, $\frac{36 (3q)^2 L_{\overline{res}}^2 }{n} \leq 1$, and the $\var{B}{t}$ are good. Then
\begin{align*}
&\mathbb{E}\left[\sum_{k'} \left(\grad{k'}{f(\var{y}{\kk, \pi})} - \grad{k'}{f(\var{y}{\kk, \pi(k', t)})}\right)^2\right] \\
 & \hspace*{0.5in} \leq \frac{9}{2} n^2 \phi_t^2 q \sum_{l \in [\kk - q, t - 1]}  \Lresbar^2 \E{}{\left(\overline{\Delta}_{\min} \var{z}{l, \pi}_{k_l}- \overline{\Delta}_{\max} \var{z}{l, \pi}_{k_l}\right)^2}\\
 &\largespace + \frac{9}{2} n^2 \phi_t^2 n \mathbb{E}\Bigg[ {\frac{4r}{3(1-r)}} \Big(\overline{\Delta}_{\min}  z_{\iit_t}^{t, \pi} -  \overline{\Delta}_{\min} z_{\iit_t}^{t, \pi}\Big)^2  +   {\frac{8r}{9(1-r)}} (\Delta z^{t, \pi}_{k_t})^2 \\
 &\largespace \largespace + \sum_{s \in [t - 3q, t+ q] \setminus \{t\}} {\frac{r}{486q(1-r)}}   \Big(\overline{\Delta}_{\max}  z_{\iit_s}^{s, \pi} -  \overline{\Delta}_{\min} z_{\iit_s}^{s, \pi}\Big)^2 \\
&\largespace \largespace+ \sum_{s \in [t - 3q, t+ q] \setminus \{t\}} {\frac{r}{486q(1-r)}} (\Delta z^{s, \pi}_{k_s} )^2\Bigg)\Bigg].
\end{align*}
\end{lemma}

For the second term, note that $\grad{\iit'}{f(\var{y}{\kk, \pi(k', t)})}$ may not be in $[ g^{t, \pi(k', t)}_{\min, k'}, g^{t, \pi(k', t)}_{\max, k'} ]$
{(because $y^{t,\pi(k',t))}$ are the actual values of the coordinates immediately prior to the time $t$ update,
and some of the earlier updates that produced $y^{t,\pi(k',t))}$ may read the updated value at time $t$,
while the terms $g^{t, \pi(k', t)}_{\min, k'}$ and $g^{t, \pi(k', t)}_{\max, k'}$ depend only on the values of earlier updates that do
not read the time $t$ update).}
{To obtain a bound, we consider the gradient value $g^{\mathbf{S}, t, \pi(k', t)}_{\iit'}$ that would occur if there were synchronous updates from time $t - q$ to $t$.} We have
\begin{align}
\label{eqn::grad-bound-using-sync-updates}
\left(\grad{\iit'}{f(\var{y}{\kk, \pi(k', t)})} - \tilde{g}^{t, \pi(k', t)}_{\iit'}\right)^2 \leq 2 \left(g^{\mathbf{S}, t, \pi(k', t)}_{\iit'} - \tilde{g}^{t, \pi(k', t)}_{\iit'}\right)^2 + 2 \left(\grad{\iit'}{f(\var{y}{\kk, \pi(k', t)})} - g^{\mathbf{S}, t, \pi(k', t)}_{\iit'}\right)^2.
\end{align}
 Note that $g^{\mathbf{S}, t, \pi(k', t)}_{\iit'} \in [ g^{t, \pi(k', t)}_{\min, k'}, g^{t, \pi(k', t)}_{\max, k'} ]$. Therefore,
\begin{align*}
\left(g^{\mathbf{S}, t, \pi(k', t)}_{\iit'} - \tilde{g}^{t, \pi(k', t)}_{\iit'}\right)^2 \leq \left( g^{t, \pi(k', t)}_{\min, k'} - g^{t, \pi(k', t)}_{\max, k'}\right)^2.
\end{align*}
Similarly to the proof in Appendix~\ref{Appendix::C1}, we obtain the bound
\begin{align*}
\left(\grad{\iit'}{f(\var{y}{\kk, \pi(k', t)})} - g^{\mathbf{S}, t, \pi(k', t)}_{\iit'}\right)^2 \leq \frac{9}{4} n^2 \phi_t^2 q \sum_{l \in [t - q, t-1] } L_{k_l, k_t}^2 \left(\overline{\Delta}_{\max} z^{l, \pi}_{k_l} -  \overline{\Delta}_{\min} z^{l, \pi}_{k_l}\right)^2.
\end{align*}
By Lemma~\ref{lem::new::amor} in Appendix~\ref{Appendix::new::amor}, in expectation, we can bound {the second term in~\eqref{eqn::grad-bound-using-sync-updates}} by 
\begin{align*}
&\frac{9}{4} n^2 \phi_t^2 \mathbb{E}\Bigg[ {\frac {4r}{3(1-r)}} \Big(\overline{\Delta}_{\min}  z_{\iit_t}^{t, \pi} -  \overline{\Delta}_{\min} z_{\iit_t}^{t, \pi}\Big)^2  +  {\frac {8r}{9(1-r)}}  (\Delta z^{t, \pi}_{k_t})^2 \\
&\largespace \largespace + \sum_{s \in [t - 3q, t+ q] \setminus \{t\}} {\frac{r}{486q(1-r)}}  \Big(\overline{\Delta}_{\max}  z_{\iit_s}^{s, \pi} -  \overline{\Delta}_{\min} z_{\iit_s}^{s, \pi}\Big)^2 \\
 &\largespace \largespace+ \sum_{s \in [t - 3q, t+ q] \setminus \{t\}} {\frac{r}{486 q(1-r)}} (\Delta z^{s, \pi}_{k_s} )^2\Bigg].
\end{align*}

\hide{
$\left(\grad{\iit'}{f(\var{y}{\kk, \pi(k', t)})} - \tilde{g}^{t, \pi(k', t)}_{\iit'}\right)^2 \leq \left(g^{t, \pi(k', t)}_{\max, k'} - g^{t, \pi(k', t)}_{\min, k'}\right)^2$}

\subsection{Proofs of the Subsidiary Lemmas}

%% file: OtherProofsFEDeltaBounds.tex
\pfof{Lemma~\ref{lem::diff::y::cord}}
WLOG, we assume $\kk - q \leq \kk_1 \leq \kk_2 \leq \kk$. Let $\left[\var{\tilde{u}}{\kk}\right]^{\pi,R,\kk_1,\iit_{\kk}}$ (resp. $\left[\var{\tilde{v}}{\kk}\right]^{\pi,R,\kk_1,\iit_{\kk}}$) denote the $\tildevar{u}{\kk}$ (resp. $\tildevar{v}{\kk}$) used to evaluate $\left[\var{\tilde{y}}{\kk}\right]^{\pi,R,\kk_1,\iit_{\kk}}$, and let  $\left[\var{\tilde{u}}{\kk}\right]^{\pi,R,\kk_2,\iit_{\kk}}$ (resp. $\left[\var{\tilde{v}}{\kk}\right]^{\pi,R,\kk_2,\iit_{\kk}}$) denote the $\tildevar{u}{\kk}$ (resp. $\tildevar{v}{\kk}$) used to evaluate $\left[\var{\tilde{y}}{\kk}\right]^{\pi,R,\kk_2,\iit_{\kk}}$. Then,
\begin{align*}
&\Big[\left[\var{\tilde{y}}{\kk}\right]^{\pi,R,\kk_1,\iit_{\kk}} - \left[\var{\tilde{y}}{\kk}\right]^{\pi,R,\kk_2,\iit_{\kk}}\Big]_{\iit} \\
 &~~~~~~= \Bigg[\var{B}{\kk} \Bigg (\begin{matrix} 
{\left[\var{\tilde{u}}{\kk}\right]^{\pi,R,\kk_1,\iit_{\kk}}}\tran \\
 {\left[\var{\tilde{v}}{\kk}\right]^{\pi,R,\kk_1,\iit_{\kk}}} \tran
\end{matrix} \Bigg) - \var{B}{\kk}\Bigg (\begin{matrix} 
{\left[\var{\tilde{u}}{\kk}\right]^{\pi,R,\kk_2,\iit_{\kk}}}\tran \\
 {\left[\var{\tilde{v}}{\kk}\right]^{\pi,R,\kk_2,\iit_{\kk}}} \tran\end{matrix} \Bigg) \Bigg]_{(1,{\iit})}
\end{align*}
The difference between $\Bigg (\begin{matrix} 
{\left[\var{\tilde{u}}{\kk}\right]^{\pi,R,\kk_1,\iit_{\kk}}}\tran \\
 {\left[\var{\tilde{v}}{\kk}\right]^{\pi,R,\kk_1,\iit_{\kk}}} \tran
\end{matrix} \Bigg)$ and $\Bigg (\begin{matrix} 
{\left[\var{\tilde{u}}{\kk}\right]^{\pi,R,\kk_2,\iit_{\kk}}}\tran \\
 {\left[\var{\tilde{v}}{\kk}\right]^{\pi,R,\kk_2,\iit_{\kk}}} \tran\end{matrix} \Bigg)$ is that  some updates may be included in $\Bigg (\begin{matrix} 
{\left[\var{\tilde{u}}{\kk}\right]^{\pi,R,\kk_1,\iit_{\kk}}}\tran \\
 {\left[\var{\tilde{v}}{\kk}\right]^{\pi,R,\kk_1,\iit_{\kk}}} \tran
\end{matrix} \Bigg)$ and not in $\Bigg (\begin{matrix} 
{\left[\var{\tilde{u}}{\kk}\right]^{\pi,R,\kk_2,\iit_{\kk}}}\tran \\
 {\left[\var{\tilde{v}}{\kk}\right]^{\pi,R,\kk_2,\iit_{\kk}}} \tran\end{matrix} \Bigg)$, and conversely. So, 
 \begin{align*}
 &\Bigg (\begin{matrix} 
{\left[\var{\tilde{u}}{\kk}\right]^{\pi,R,\kk_1,\iit_{\kk}}}\tran \\
 {\left[\var{\tilde{v}}{\kk}\right]^{\pi,R,\kk_1,\iit_{\kk}}} \tran
\end{matrix} \Bigg) - \Bigg (\begin{matrix} 
{\left[\var{\tilde{u}}{\kk}\right]^{\pi,R,\kk_2,\iit_{\kk}}}\tran \\
 {\left[\var{\tilde{v}}{\kk}\right]^{\pi,R,\kk_2,\iit_{\kk}}} \tran\end{matrix} \Bigg) \\
&~~~~~~ = \sum_{s \in [\kk - 2q, \kk_1 + q] \setminus (R\cup \{\kk\})} \cdot {\var{B}{s+1}}^{-1} \cdot \text{Update}^{s, t_1}\\
&~~~~~~~~~~~~ -  \sum_{s \in [\kk - 2q, \kk_2+q] \setminus (R\cup \{\kk\})} \cdot {\var{B}{s+1}}^{-1} \cdot \text{Update}^{s, t_2},
 \end{align*}

where $\text{Update}^{s, t_1}$ can be one of following, where $\ivec_{\iit_{s}}$ denotes a vector which is $1$ on coordinate $\iit_{s}$ and $0$ on others:
\begin{table}[h]
\centering
\label{my-label}
\begin{tabular}{ll}
 $\Bigg  (\begin{matrix} 
[n \psi_{s+1} \phi_s + (1 - \psi_{s+1})] \Delta^{\kk_1, R \cup \{\kk\}} z^{s, \pi}_{k_s}\ivec_{\iit_{s}} \\
\Delta^{\kk_1, R \cup \{\kk\}} z^{s, \pi}_{k_s}\ivec_{\iit_{s}}
\end{matrix} \Bigg)$; & $\Bigg(\begin{matrix} 
0 \\
\Delta^{\kk_1, R \cup \{\kk\}} z^{s, \pi}_{k_s}\ivec_{\iit_{s}}
\end{matrix} \Bigg)$; \\
$\Bigg(\begin{matrix} 
[n \psi_{s+1} \phi_s + (1 - \psi_{s+1})] \Delta^{\kk_1, R \cup \{\kk\}} z^{s, \pi}_{k_s}\ivec_{\iit_{s}}\\
0
\end{matrix} \Bigg)$; & $\Bigg(\begin{matrix} 
0 \\
0
\end{matrix} \Bigg)$,
\end{tabular}
\end{table}

\begin{table}[h]
 and $\text{Update}^{s, t_2}$ can be one of following:
\\
 
\centering
\label{my-label}
\begin{tabular}[h]{ll}
 $\Bigg  (\begin{matrix} 
[n \psi_{s+1} \phi_s + (1 - \psi_{s+1})] \Delta^{\kk_2, R \cup \{\kk\}} z^{s, \pi}_{k_s}\ivec_{\iit_{s}} \\
\Delta^{\kk_2, R \cup \{\kk\}} z^{s, \pi}_{k_s}\ivec_{\iit_{s}}
\end{matrix} \Bigg)$; &$\Bigg(\begin{matrix} 
0 \\
\Delta^{\kk_2, R \cup \{\kk\}} z^{s, \pi}_{k_s}\ivec_{\iit_{s}}
\end{matrix} \Bigg)$;\\
$\Bigg(\begin{matrix} 
[n \psi_{s+1} \phi_s + (1 - \psi_{s+1})] \Delta^{\kk_2, R \cup \{\kk\}} z^{s, \pi}_{k_s}\ivec_{\iit_{s}}  \\
0
\end{matrix} \Bigg)$; & $\Bigg(\begin{matrix} 
0 \\
0
\end{matrix} \Bigg)$.
\end{tabular}
\end{table}
 
Therefore,
\begin{align*}
&\Bigg[\var{B}{\kk} \Bigg (\begin{matrix} 
{\left[\var{\tilde{u}}{\kk}\right]^{\pi,R,\kk_1,\iit_{\kk}}}\tran \\
 {\left[\var{\tilde{v}}{\kk}\right]^{\pi,R,\kk_1,\iit_{\kk}}} \tran
\end{matrix} \Bigg) - \var{B}{\kk}\Bigg (\begin{matrix} 
{\left[\var{\tilde{u}}{\kk}\right]^{\pi,R,\kk_2,\iit_{\kk}}}\tran \\
 {\left[\var{\tilde{v}}{\kk}\right]^{\pi,R,\kk_2,\iit_{\kk}}} \tran\end{matrix} \Bigg) \Bigg]\\
&\largespace= ~\sum_{s \in [\kk - 2q, \kk_1 + q] \setminus (R\cup \{\kk\})}  \var{B}{\kk} \cdot {\var{B}{s+1}}^{-1}\cdot \text{Update}^{s, t_1}\\
&\largespace~~~~~~~~~~~~ -  \sum_{s \in [\kk - 2q, \kk_2+q] \setminus (R\cup \{\kk\})}  \var{B}{\kk} \cdot {\var{B}{s+1}}^{-1}\cdot \text{Update}^{s, t_2}.
 \end{align*}

We know that $\var{B}{\kk} \cdot {\var{B}{s+1}}^{-1}$ is a $2 \times 2$ matrix. Now let $\delta^{\kk_1}_{\kk,s}$ (resp. $\delta^{\kk_2}_{\kk,s}$) denote the first entry of the vector \begin{align*} \Bigg  (\begin{matrix} 
[n \psi_{s+1} \phi_s + (1 - \psi_{s+1})] \\
1
\end{matrix} \Bigg) &~~~~~~\mbox{ or } &\var{B}{\kk} \cdot {\var{B}{s +1}}^{-1} \cdot \Bigg(\begin{matrix} 
[n \psi_{s+1} \phi_s + (1 - \psi_{s+1})]  \\
0
\end{matrix} \Bigg) \\
\mbox{ or }\var{B}{\kk} \cdot {\var{B}{s + 1}}^{-1} \cdot \Bigg(\begin{matrix} 
0  \\
1
\end{matrix} \Bigg)  &~~~~~~\mbox{ or } &\var{B}{\kk} \cdot {\var{B}{s + 1}}^{-1} \cdot \Bigg(\begin{matrix} 
0\\
0
\end{matrix} \Bigg),
\end{align*}
corresponding to the choice of $\text{Update}^{s, t_1}$ (resp.\ $\text{Update}^{s, t_2}$).
Since $\kk_1, \kk_2 \in [\kk - q, \kk]$, $s \in [\kk - 2q, \kk + 2q]$ and $\var{B}{t}$ are good, 
$| \delta^{t_1}_{\kk,s} |,  | \delta^{t_2}_{\kk,s} | \leq \frac{3}{2} n \phi_t$. 

Since $\ivec_{\iit_{s}}$ is $1$ on coordinate $\iit_{s}$ and $0$ on all other coordinates, 
\begin{align*}
\Big[\left[\var{\tilde{y}}{\kk}\right]^{\pi,R,\kk_1,\iit_{\kk}} - \left[\var{\tilde{y}}{\kk}\right]^{\pi,R,\kk_2,\iit_{\kk}}\Big]_{\iit} 
 = \sum_{s \in [\kk - 2q, \kk + q] \setminus (R\cup \{k\}) \textit{ and } {k_s = k}} \Big( \delta^{t_1}_{\kk,s} \Delta^{\kk_1, R \cup \{\kk\}} z^{s, \pi}_{k_s} -   \delta^{t_2}_{\kk,s} \Delta^{\kk_2, R \cup \{\kk\}} z^{s, \pi}_{k_s}\Big).
\end{align*}
 Then,  by Lemma~\ref{lem::prop::B}, 
\begin{align*}
&\Big|\Big[\left[\var{\tilde{y}}{\kk}\right]^{\pi,R,\kk_1,\iit_{\kk}} - \left[\var{\tilde{y}}{\kk}\right]^{\pi,R,\kk_2,\iit_{\kk}}\Big]_{\iit} \Big| \\
&~~~~~~\leq \sum_{s \in [\kk - 2q, \kk_1 + q] \setminus (R\cup \{\kk\}) \textit{ and } {k_s = k}} 3 n \phi_{\kk}  \max\{ |\Delta^{\kk_1, R \cup \{\kk\}} z^{s, \pi}_{k_s} - \Delta^{\kk_2, R \cup \{\kk\}} z^{s, \pi}_{k_s}| ,  \\
&~~~~~~~~~~~~~~~~~~~~~~~~~~~~~~~~~~~~~~~~~~~~~~~~~~ |\Delta^{\kk_1, R \cup \{\kk\}} z^{s, \pi}_{k_s}|, |\Delta^{\kk_2, R \cup \{\kk\}} z^{s, \pi}_{k_s} |\} \\
&~~~~~~~~~+ \sum_{s \in [\kk_1 + q + 1, \kk_2 + q] \setminus (R\cup \{\kk\}) \textit{ and } {k_s = k}}  3 n \phi_{\kk}  |\Delta^{\kk_2, R \cup \{\kk\}} z^{s, \pi}_{k_s} | \\
&~~~~~~\leq \sum_{s \in [\kk - 2q, \kk_1 + q] \setminus (R\cup \{\kk\}) \textit{ and } {k_s = k}} 3 n \phi_{\kk}  \max\{ |\Delta^{\kk_1, R \cup \{\kk\}}_{\max} z^{s, \pi}_{k_s} - \Delta^{\kk_2, R \cup \{\kk\}}_{\min} z^{s, \pi}_{k_s}| ,  \\
&~~~~~~~~~~~~~~~~~~~~~~~~~~~~~~~~~~~~~~~~~~~~~~~~~~|\Delta^{\kk_2, R \cup \{\kk\}}_{\max} z^{s, \pi}_{k_s} - \Delta^{\kk_1, R \cup \{\kk\}}_{\min} z^{s, \pi}_{k_s}|,  \\
&~~~~~~~~~~~~~~~~~~~~~~~~~~~~~~~~~~~~~~~~~~~~~~~~~~ |\Delta^{\kk_1, R \cup \{\kk\}}_{\max} z^{s, \pi}_{k_s}|, |\Delta^{\kk_1, R \cup \{\kk\}}_{\min} z^{s, \pi}_{k_s}|\} \\
&~~~~~~~~~~~~~~~~~~~~~~~~~~~~~~~~~~~~~~~~~~~~~~~~~~ |\Delta^{\kk_2, R \cup \{\kk\}}_{\max} z^{s, \pi}_{k_s}|, |\Delta^{\kk_2, R \cup \{\kk\}}_{\min} z^{s, \pi}_{k_s}|\} \\
&~~~~~~~~~+ \sum_{s \in \{[\kk_1 + q + 1, \kk_2 + q] \setminus (R\cup \{\kk\})\} \textit{ and } {k_s = k}}  3 n \phi_{\kk}  \max\{ |\Delta^{\kk_2, R \cup \{\kk\}}_{\max} z^{s, \pi}_{k_s}|, \Delta^{\kk_1, R \cup \{\kk\}}_{\min} z^{s, \pi}_{k_s}|\}. 
\end{align*}

Next, we make the following assertions:

\begin{itemize}
\item If $s \in [\kk - 2q, \kk_1+q]$ and $\kk_1 \in [\kk - q, \kk]$, then
\begin{align*}
\Delta^{\kk_1, R \cup \{\kk\}}_{\max} z^{s, \pi}_{k_s} \leq \max_{l \in [ \max \{s-q, \kk-q\}, \min \{s, \kk\}] \cup \{s\}}\{\Delta^{l, R \cup \{\kk\}}_{\max} z^{s, \pi}_{k_s}\}.
\end{align*}
\item If $s \in [\kk - 2q, \kk_1+q]$ and $\kk_1 \in [\kk - q, \kk]$, then
\begin{align*}
\Delta^{\kk_1, R \cup \{\kk\}}_{\min} z^{s, \pi}_{k_s} \geq \min_{l \in [ \max \{s-q, \kk-q\}, \min \{s, \kk\}] \cup \{s\}}\{\Delta^{l, R \cup \{\kk\}}_{\max} z^{s, \pi}_{k_s}\}.
\end{align*}
\item If $s \in [\kk - 2q, \kk_2+q]$ and $\kk_2 \in [\kk - q, \kk]$, then
\begin{align*}
\Delta^{\kk_2, R \cup \{\kk\}}_{\max} z^{s, \pi}_{k_s}\leq \max_{l \in [ \max \{s-q, \kk-q\}, \min \{s, \kk\}] \cup \{s\}}\{\Delta^{l, R \cup \{\kk\}}_{\max} z^{s, \pi}_{k_s}\}.
\end{align*}
 \item If $s \in [\kk - 2q, \kk_2+q]$ and $\kk_2 \in [\kk - q, \kk]$, then
\begin{align*}
\Delta^{\kk_2, R \cup \{\kk\}}_{\min} z^{s, \pi}_{k_s}\geq \min_{l \in [ \max \{s-q, \kk-q\}, \min \{s, \kk\}] \cup \{s\}}\{\Delta^{l, R \cup \{\kk\}}_{\min} z^{s, \pi}_{k_s}\}.
\end{align*}
\end{itemize}
We justify the first assertion. The arguments for the others are very similar.

We consider two cases.

\noindent {\bf Case 1}.  $s \in [\kk_1, \kk_1 + q]$.\\
Then, $\kk_1 \in [ \max \{s-q, \kk-q\}, \min \{s, \kk\}]$. So the assertion is true. 

\noindent
{\bf Case 2}. $s \in [\kk - 2q, \kk_1 - 1]$.\\
We use the fact that $\Delta^{l, R \cup \{\kk\}}_{\max} z^{s, \pi}_{k_s}  \leq \Delta^{s, R \cup \{\kk\}}_{\max} z^{s, \pi}_{k_s}$ if $l > s$ from \cite{2016arXiv161209171K} and \cite{CCT2018}.
\begin{align*}
\Delta^{\kk_1, R \cup \{\kk\}}_{\max} z^{s, \pi}_{k_s}&\leq \Delta^{s, R \cup \{\kk\}}_{\max} z^{s, \pi}_{k_s} ~~~~~~~~~\mbox{(as $t_1 > s$)}\\
&\leq \max_{l \in [ \max \{s-q, \kk-q\}, \min \{s, \kk\}]\cup \{s\}}\{\Delta^{l, R \cup \{\kk\}}_{\max} z^{s, \pi}_{k_s}\}.
\end{align*}

Now we can conclude that
\begin{align*}
&\Big|\Big[\left[\var{\tilde{y}}{\kk}\right]^{\pi,R,\kk_1,\iit_{\kk}} - \left[\var{\tilde{y}}{\kk}\right]^{\pi,R,\kk_2,\iit_{\kk}}\Big]_{\iit}\Big| \\
&~~~~~~\leq 3 n \phi_{\kk} \sum_{ \iit_s = \iit\text{ and }s \in \{[\kk - 2q, \kk + q]\setminus R \cup \{\kk\}\}}   \max\Bigg\{\Big|\max_{l \in [ \max \{s-q, \kk-q\}, \min \{s, \kk\}] \cup \{s\} }\{\Delta^{l, R \cup \{\kk\}}_{\max} z^{s, \pi}_{k_s}\} \\
&~~~~~~~~~~~~~~~~~~~~~~~~~~~~~~~~~~~~~~~~~~~~~~~~~~~~~~~~~~~~~~~~~~~~-\min_{l \in [ \max \{s-q, \kk-q\}, \min \{s, \kk\}] \cup \{s\}} \{\Delta^{l, R \cup \{\kk\}}_{\min} z^{s, \pi}_{k_s}\}\Big|, \\
&~~~~~~~~~~~~~~~~~~~~~~~~~~~~~~~~~~\Big|\max_{l \in [ \max \{s-q, \kk-q\}, \min \{s, \kk\}] \cup \{s\}}\{\Delta^{l, R \cup \{\kk\}}_{\max} z^{s, \pi}_{k_s}\}\Big|, \\
&~~~~~~~~~~~~~~~~~~~~~~~~~~~~~~~~~~ \Big|\min_{l \in [ \max \{s-q, \kk-q\}, \min \{s, \kk\}] \cup \{s\}} \{\Delta^{l, R \cup \{\kk\}}_{\min} z^{s, \pi}_{k_s}\}\Big| \Bigg\}. 
\end{align*}
\end{proof}

%% file: Diff-w-z.tex

\pfof{Lemma~\ref{lem::diff::w::cord}}
The proof of the bounds on $\sum_{k'} \left(\var{w}{\kk, \pi}_{k'} - \var{w}{\kk, \pi(k', t)}_{k'}\right)^2$ and $\sum_{k'} \left(\var{z}{\kk, \pi}_{k'} - \var{z
}{\kk, \pi(k', t)}_{k'}\right)^2 $ are similar. Here, we only give the proof for $\sum_{k'} \left(\var{w}{\kk, \pi}_{k'} - \var{w}{\kk, \pi(k', t)}_{k'}\right)^2$.
\hide{Since
\begin{align*}
\var{w}{\kk, \pi} = (\bbeta_t, 1 - \bbeta_t) (\var{y}{\kk, \pi}, \var{z}{\kk, \pi}) \tran 
= (\bbeta_t, 1 - \bbeta_t) \var{B}{t} \left[(\var{u}{\kk - q - 1, \pi}, \var{v}{\kk - q - 1, \pi})\tran + \sum_{l = \kk - q - 1}^{t - 1} {\var{B}{l+1}}^{-1} \var{D}{l, \pi})\right],
\end{align*}
and $\var{D}{l, \pi}_{k'} = \Bigg(\begin{matrix} 
[n \psi_{l+1} \phi_l + (1 - \psi_{l+1})] \Delta \var{z}{l, \pi}_{k'}\\
\Delta \var{z}{l, \pi}_{k'}
\end{matrix} \Bigg)$, the column $k'$ of matrix $\var{D}{l, \pi}$, is non-zero if $k' = k_l$. Therefore,}
Remember that
\begin{align*}
\var{w}{\kk, \pi}_{k'} &= (\bbeta_t, 1 - \bbeta_t) (\var{z}{\kk, \pi}_{k'}, \var{y}{\kk, \pi}_{k'}) \tran \\
&= (\bbeta_t, 1 - \bbeta_t) \var{B}{t} \left[(\var{u}{\kk - q , \pi}_{k'}, \var{v}{\kk - q , \pi}_{k'})\tran + \sum_{l \in [\kk - q , t - 1] \text{ and } k_l = k'} {\var{B}{l+1}}^{-1} \var{D}{l, \pi}  \Delta \var{z}{l, \pi}_{k'}\right].
\end{align*}
Remember that $\Delta \var{z}{l, \pi}_{k'} = 0$ for $k' \neq k_l$.

Applying the Cauchy-Schwarz inequality gives:
\begin{align*}
\left(\var{w}{\kk, \pi}_{k'} - \var{w}{\kk, \pi(k', t)}_{k'}\right)^2 ~
&  = ~\left(\sum_{l \in [\kk - q , t - 1] \text{ and } k_l = k'} (\bbeta_t, 1 - \bbeta_t) \var{B}{t} {\var{B}{l+1}}^{-1}  \var{D}{l}\left( \Delta \var{z}{l, \pi}_{k'} - \Delta \var{z}{l, \pi(k', t)}_{k'}\right)\right)^2 \\
& \leq ~q \sum_{l \in [\kk - q , t - 1] \text{ and } k_l = k'} \left((\bbeta_t, 1 - \bbeta_t)\var{B}{t} {\var{B}{l+1}}^{-1}  \var{D}{l}\left( \Delta \var{z}{l, \pi}_{k'} - \Delta \var{z}{l, \pi(k', t)}_{k'}\right)\right)^2.
\end{align*}
\hide{
where $\var{D}{l, \pi}_{k'} = \Bigg(\begin{matrix} 
[n \psi_{l+1} \phi_l + (1 - \psi_{l+1})] \Delta \var{z}{l, \pi}_{k'}\\
\Delta \var{z}{l, \pi}_{k'}
\end{matrix} \Bigg)$.}

Since $\var{B}{t}$ are good, we know that $\left|(\bbeta_t, 1 - \bbeta_t) \var{B}{t} {\var{B}{l+1}}^{-1} \Bigg(\begin{matrix} 
[n \psi_{l+1} \phi_l + (1 - \psi_{l+1})]\\
1
\end{matrix} \Bigg)\right| \leq 2$.
So,  
\begin{align*}
\left(\var{w}{\kk, \pi}_{k'} - \var{w}{\kk, \pi(k', t)}_{k'}\right)^2 
~ \leq ~ 4q  \sum_{l \in [\kk - q , t - 1] \text{ and } k_l = k'} \left( \Delta \var{z}{l, \pi}_{k'} - \Delta \var{z}{l, \pi(k', t)}_{k'}\right)^2.
\end{align*}

We know that $ \Delta \var{z}{l, \pi}_{k_l} \in [\overline{\Delta}_{\min} \var{z}{l, \pi}_{k_l}, \overline{\Delta}_{\max} \var{z}{l, \pi}_{k_l}]$, $ \Delta \var{z}{l, \pi(k', t)}_{k_l} \in [\overline{\Delta}_{\min} \var{z}{l, \pi(k', t)}_{k_l}, \overline{\Delta}_{\max} \var{z}{l, \pi(k', t)}_{k_l}]$; also the intervals  $[\overline{\Delta}_{\min} \var{z}{l, \pi}_{k_l}, \overline{\Delta}_{\max} \var{z}{l, \pi}_{k_l}]$ and  $[\overline{\Delta}_{\min} \var{z}{l, \pi(k', t)}_{k_l}, \overline{\Delta}_{\max} \var{z}{l, \pi(k', t)}_{k_l}]$ overlap\\
(as $[\overline{\Delta}_{\min} \var{z}{l, \pi(k', t)}_{k_l}, \overline{\Delta}_{\max} \var{z}{l, \pi(k', t)}_{k_l}] \supseteq [\Delta^{l, \{t\}}_{\min} \var{z}{l, \pi(k', t)}_{k_l}, \Delta^{l, \{t\}}_{\max} \var{z}{l, \pi(k', t)}_{k_l}] =  [\Delta^{l, \{t\}}_{\min} \var{z}{l, \pi}_{k_l}, \Delta^{l, \{t\}}_{\max} \var{z}{l, \pi}_{k_l}] \subseteq [\overline{\Delta}_{\min} \var{z}{l, \pi}_{k_l}, \overline{\Delta}_{\max} \var{z}{l, \pi}_{k_l}]$). Therefore
\begin{align*}
&\left(\var{w}{\kk, \pi}_{k'} - \var{w}{\kk, \pi(k', t)}_{k'}\right)^2 \\
&\largespace \leq 8q  \sum_{l \in [\kk - q , t - 1] \text{ and } k_l = k'} \left[\left(\overline{\Delta}_{\min} \var{z}{l, \pi}_{k_l}- \overline{\Delta}_{\max} \var{z}{l, \pi}_{k_l}\right)^2 + \left(\overline{\Delta}_{\min} \var{z}{l, \pi(k', t)}_{k_l} - \overline{\Delta}_{\max} \var{z}{l, \pi(k', t)}_{k_l}\right)^2 \right].
\end{align*}
Summing over $k'$ yields
\begin{align*}
&\sum_{k'} \left(\var{w}{\kk, \pi}_{k'} - \var{w}{\kk, \pi(k', t)}_{k'}\right)^2 \\
&\largespace \leq 8q \sum_{l \in [\kk - q , t - 1] } \left[\left(\overline{\Delta}_{\min} \var{z}{l, \pi}_{k_l}- \overline{\Delta}_{\max} \var{z}{l, \pi}_{k_l}\right)^2 + \left(\overline{\Delta}_{\min} \var{z}{l, \pi(k', t)}_{k_l} - \overline{\Delta}_{\max} \var{z}{l, \pi(k', t)}_{k_l}\right)^2 \right].
\end{align*}

\end{proof}

%% file: Diff_y.tex

\pfof{Lemma~\ref{lem::diff::true::y::cord}}\hide{
\begin{align*}
\var{y}{\kk, \pi} &= (1, 0) (\var{y}{\kk, \pi}, \var{z}{\kk, \pi}) \tran = (1, 0) \var{B}{t} \left[(\var{u}{\kk - q - 1, \pi}, \var{v}{\kk - q - 1, \pi})\tran + \sum_{l = \kk - q - 1}^{t - 1} {\var{B}{l+1}}^{-1} \var{D}{l, \pi})\right],
\end{align*}
and $\var{D}{l, \pi}_{k'}$ is non-zero only if $k' = k_l$. Therefore,}
Remember
\begin{align*}
\var{y}{\kk, \pi}_{k} &= (1, 0) (\var{y}{\kk, \pi}_{k}, \var{z}{\kk, \pi}_{k}) \tran 
= (1, 0) \var{B}{t} \left[(\var{u}{\kk - q , \pi}_{k}, \var{v}{\kk - q , \pi}_{k})\tran + \sum_{l \in [\kk - q , t - 1] \text{ and } k_l = k} {\var{B}{l+1}}^{-1} \var{D}{l} \Delta \var{z}{l, \pi}_{k} \right].
\end{align*}

Applying the Cauchy-Schwarz inequality gives:
\begin{align*}
&\left(\var{y}{\kk, \pi}_{k} - \var{y}{\kk, \pi(k', t)}_{k}\right)^2 
 = \left(\sum_{l \in [\kk - q , t - 1] \text{ and } k_l = k} (1, 0) \var{B}{t} {\var{B}{l+1}}^{-1} \var{D}{l} \left( \Delta \var{z}{l, \pi}_{k} - \Delta \var{z}{l, \pi(k', t)}_{k}\right)\right)^2 \\
& \leq q \sum_{l \in [\kk - q , t - 1] \text{ and } k_l = k} \left((1, 0) \var{B}{t} {\var{B}{l+1}}^{-1} \var{D}{l} \left( \Delta \var{z}{l, \pi}_{k} - \Delta \var{z}{l, \pi(k', t)}_{k}\right) \right)^2,
\end{align*}
where $\var{D}{l} = \Bigg(\begin{matrix} 
[n \psi_{l+1} \phi_l + (1 - \psi_{l+1})] \\
1
\end{matrix} \Bigg)$.

Since the $\var{B}{t}$ are good, we know that the absolute value of the first element of\\ $\var{B}{t} {\var{B}{l+1}}^{-1} \Bigg(\begin{matrix} 
[n \psi_{l+1} \phi_l + (1 - \psi_{l+1})]\\
1
\end{matrix} \Bigg)$ is less than $ \frac{3}{2} n \aalpha_t$.
And so  
\begin{align*}
\left(\var{y}{\kk, \pi}_{k} - \var{y}{\kk, \pi(k', t)}_{k}\right)^2 ~
~ \leq \frac{9}{4} n^2 \phi_t^2 q \sum_{l \in [\kk - q , t - 1] \text{ and } k_l = k} \left( \Delta \var{z}{l, \pi}_{k} - \Delta \var{z}{l, \pi(k', t)}_{k}\right)^2.
\end{align*}
We know that $ \Delta \var{z}{l, \pi}_{k_l} \in [\overline{\Delta}_{\min} \var{z}{l, \pi}_{k_l}, \overline{\Delta}_{\max} \var{z}{l, \pi}_{k_l}]$, $ \Delta \var{z}{l, \pi(k, t)}_{k_l} \in [\overline{\Delta}_{\min} \var{z}{l, \pi(k, t)}_{k_l}, \overline{\Delta}_{\max} \var{z}{l, \pi(k, t)}_{k_l}]$ and,  also, that $[\overline{\Delta}_{\min} \var{z}{l, \pi}_{k_l}, \overline{\Delta}_{\max} \var{z}{l, \pi}_{k_l}]$ and  $[\overline{\Delta}_{\min} \var{z}{l, \pi(k, t)}_{k_l}, \overline{\Delta}_{\max} \var{z}{l, \pi(k, t)}_{k_l}]$ overlap (as $[\overline{\Delta}_{\min} \var{z}{l, \pi(k, t)}_{k_l}, \overline{\Delta}_{\max} \var{z}{l, \pi(k, t)}_{k_l}] \supseteq [\Delta^{l, \{t\}}_{\min} \var{z}{l, \pi(k, t)}_{k_l}, \Delta^{l, \{t\}}_{\max} \var{z}{l, \pi(k, t)}_{k_l}] =  [\Delta^{l, \{t\}}_{\min} \var{z}{l, \pi}_{k_l}, \Delta^{l, \{t\}}_{\max} \var{z}{l, \pi}_{k_l}] \subseteq [\overline{\Delta}_{\min} \var{z}{l, \pi}_{k_l}, \overline{\Delta}_{\max} \var{z}{l, \pi}_{k_l}]$). Therefore,
\begin{align*}
&\left(\var{y}{\kk, \pi}_{k} - \var{y}{\kk, \pi(k', t)}_{k}\right)^2 \\
&\hspace*{0.5in} \leq \frac{9}{2} n^2 \phi_t^2 q \sum_{l \in [\kk - q , t - 1] \text{ and } k_l = k} \left[\left(\overline{\Delta}_{\min} \var{z}{l, \pi}_{k_l}- \overline{\Delta}_{\max} \var{z}{l, \pi}_{k_l}\right)^2 + \left(\overline{\Delta}_{\min} \var{z}{l, \pi(k', t)}_{k_l} - \overline{\Delta}_{\max} \var{z}{l, \pi(k', t)}_{k_l}\right)^2 \right].
\end{align*}
Also,
\begin{align*}
&\left(\grad{k'}{f(\var{y}{\kk, \pi})} - \grad{k'}{f(\var{y}{\kk, \pi(k', t)})}\right)^2 \\
& \hspace*{0.5in} \leq \frac{9}{2} n^2 \phi_t^2 q \sum_{l \in [\kk - q , t - 1]} L^2_{k_l, k'} \left[\left(\overline{\Delta}_{\min} \var{z}{l, \pi}_{k_l}- \overline{\Delta}_{\max} \var{z}{l, \pi}_{k_l}\right)^2 + \left(\overline{\Delta}_{\min} \var{z}{l, \pi(k', t)}_{k_l} - \overline{\Delta}_{\max} \var{z}{l, \pi(k', t)}_{k_l}\right)^2 \right].
\end{align*}
Summing over $k'$ yields
\begin{align*}
&\sum_{k'} \left(\grad{k'}{f(\var{y}{\kk, \pi})} - \grad{k'}{f(\var{y}{\kk, \pi(k', t)})}\right)^2 \\
& \smallspace \leq \frac{9}{2} n^2 \phi_t^2 q \sum_{l \in [\kk - q , t - 1]} \Bigg[  \left[\Lresbar^2 \left(\overline{\Delta}_{\min} \var{z}{l, \pi}_{k_l}- \overline{\Delta}_{\max} \var{z}{l, \pi}_{k_l}\right)^2\right] \\
& \largespace\largespace\smallspace +\left. \sum_{k'} L_{k_l, k'} \left[\left(\overline{\Delta}_{\min} \var{z}{l, \pi(k', t)}_{k_l} - \overline{\Delta}_{\max} \var{z}{l, \pi(k', t)}_{k_l}\right)^2 \right]\right].
\end{align*}
In Lemma~\ref{lem::new::amor} in Appendix~\ref{Appendix::new::amor}, we show that 
\begin{align*}
&\E{}{q \sum_{l = [t -q, t-1]} L_{k_l, k_t}^2 (\overline{\Delta}_{\max} z_{\iit_l}^{l, \pi} - \overline{\Delta}_{\min} z_{\iit_l}^{l, \pi})^2} \\
&\largespace \leq \mathbb{E}\Bigg[ {\frac{4r}{3(1-r)}}  \Big(\overline{\Delta}_{\min}  z_{\iit_t}^{t, \pi} -  \overline{\Delta}_{\min} z_{\iit_t}^{t, \pi}\Big)^2  +   \frac{8r}{9(1-r)}  (\Delta z^{t, \pi}_{k_t})^2 \\
&\largespace \smallspace \smallspace+ \sum_{s \in [t - 3q, t+ q] \setminus \{t\}} {\frac{r}{486 q(1-r)}}   \Big(\overline{\Delta}_{\max}  z_{\iit_s}^{s, \pi} -  \overline{\Delta}_{\min} z_{\iit_s}^{s, \pi}\Big)^2 \\
 &\largespace \smallspace \smallspace + \sum_{s \in [t - 3q, t+ q] \setminus \{t\}} {\frac{r}{486 q(1-r)}} (\Delta z^{s, \pi}_{k_s} )^2\Bigg].
\end{align*}
Therefore,
\begin{align*}
&\mathbb{E}\left[\sum_{k'} \left(\grad{k'}{f(\var{y}{\kk, \pi})} - \grad{k'}{f(\var{y}{\kk, \pi(k', t)})}\right)^2\right] \\
& \hspace*{0.5in} \leq \frac{9}{2} n^2 \phi_t^2 q \sum_{l \in [\kk - q , t - 1]}  \Lresbar^2 \E{}{\left(\overline{\Delta}_{\min} \var{z}{l, \pi}_{k_l}- \overline{\Delta}_{\max} \var{z}{l, \pi}_{k_l}\right)^2}\\
&\largespace + \frac{9}{2} n^2 \phi_t^2 n \mathbb{E}\Bigg[ {\frac{4r}{3(1-r)}}  \Big(\overline{\Delta}_{\min}  z_{\iit_t}^{t, \pi} -  \overline{\Delta}_{\min} z_{\iit_t}^{t, \pi}\Big)^2  +   {\frac{8r}{9(1-r)}}  (\Delta z^{t, \pi}_{k_t})^2 \\
&\largespace \largespace\smallspace + \sum_{s \in [t - 3q, t+ q] \setminus \{t\}} {\frac{r}{486q(1-r)}}   \Big(\overline{\Delta}_{\max}  z_{\iit_s}^{s, \pi} -  \overline{\Delta}_{\min} z_{\iit_s}^{s, \pi}\Big)^2 \\
 &\largespace \largespace\smallspace+ \sum_{s \in [t - 3q, t+ q] \setminus \{t\}} {\frac{r}{486q(1-r)}} (\Delta z^{s, \pi}_{k_s} )^2\Bigg].
\end{align*}

\end{proof}

%% file: TechnicalLemmas.tex
\section{Some Technical Lemmas}
\label{Appendix::D}
\begin{defn} \label{defn::good}
We say $\var{B}{t}$ are good if the absolute value of the first elements of these four matrices:
\begin{align*}
&\var{B}{t} {\var{B}{s+1}}^{-1} \Bigg  (\begin{matrix} 
[n \psi_{s+1} \phi_s + (1 - \psi_{s+1})] \\
1
\end{matrix} \Bigg), &\var{B}{t} {\var{B}{s+1}}^{-1} \Bigg  (\begin{matrix} 
[n \psi_{s+1} \phi_s + (1 - \psi_{s+1})] \\
0
\end{matrix} \Bigg),  \\
&\var{B}{t} {\var{B}{s+1}}^{-1} \Bigg  (\begin{matrix} 
0 \\
1
\end{matrix} \Bigg), &~~~~\mbox{and}~~~~~~~~~~~\var{B}{t} {\var{B}{s+1}}^{-1} \Bigg  (\begin{matrix} 
0 \\
0
\end{matrix} \Bigg) .
\end{align*}
are smaller than $\frac{3}{2} n \phi_t$, and the second element of \begin{align*}
&\var{B}{t} {\var{B}{s+1}}^{-1} \Bigg  (\begin{matrix} 
[n \psi_{s+1} \phi_s + (1 - \psi_{s+1})] \\
1
\end{matrix} \Bigg)
\end{align*} is smaller than $2$,  for any $s$ and $t$ such that $|s - t| \leq 2q$.
\end{defn}
\pfof{Lemma~\ref{lem::delta::act}}
For any $t$ and any $s \in [t - 2q, t + 2q]$,
\begin{align*}
a_t \Big(\prod_{l = \kk + 1}^{T-1}  (1 - \parb \phi_l)\Big) \leq \frac{1}{\Phi_a} a_s \Big(\prod_{l = s + 1}^{T-1}  (1 - \parb \phi_l)\Big).
\end{align*}
Therefore, by \eqref{ineq::tran::1} and \eqref{ineq::tran::2},
\begin{align*}
\sum_t a_t \Big(\prod_{l = \kk + 1}^{T-1}  (1 - \parb \phi_l)\Big)(\Delta_{\kk}^{\tFE})^2  & \leq \sum_t \frac{54 n^2 \phi_{\kk}^2 q \Lresbar^2}{n (\Gamma_{\kk})^2} \sum_{s \in [\kk - 2q, \kk + 2q] \setminus \{\kk\}} \Big[ a_t \Big(\prod_{l = t + 1}^{T-1}  (1 - \parb \phi_l)\Big)\left((\Delta_{s}^{\tFE})^2 + E_s^{\Delta}\right) \Big] \\
&\leq  \sum_t  \Big[ \frac{54 n^2 \phi_{\kk}^2 q \Lresbar^2}{n (\Gamma_{k})^2} \cdot \frac{4q}{\Phi_a} a_t  \Big(\prod_{l = t + 1}^{T-1}  (1 - \parb \phi_l)\Big)\left((\Delta_{t}^{\tFE})^2 + E_t^{\Delta}\right) \Big] \\
&\leq \frac{\Xi}{\Phi_a} \sum_t \Big[ a_t \Big(\prod_{l = t + 1}^{T-1}  (1 - \parb \phi_l)\Big)\left((\Delta_{t}^{\tFE})^2 + E_t^{\Delta}\right) \Big].
\end{align*}
On rearranging, the result follows.
\end{proof}

\begin{lemma}\label{lem::prop::B}
If $- \delta \leq \delta_1, \delta_2 \leq \delta$ then
$
|\delta_1 a - \delta_2 b| \leq \max \{2 \delta |a - b|,  2 \delta|a|, 2 \delta|b|\}$.
\end{lemma}

\begin{lemma}\label{lem::para1}
 $\para(a - b)^2 + a^2 \geq (1 - \frac{1}{\para+1}) b^2$ for any $\para>0$, $a$ and $b$.
\end{lemma}
\begin{lemma} \label{lem::convergence}
Suppose $\prod_{\kk = 0 \cdots T-1} (1 - \phi_{\kk})$ has a convergence rate of $f(T)$ for any $T$, which means that $\frac{\prod_{k = 0}^{T-1} (1 - \phi_{\kk})}{f(T)} \leq 1$; then, for $\tau \geq 1$, the convergence rate of $\prod_{\kk = 0}^{T-1} (1 - \parb \phi_{\kk})$ is $f(T)^{\frac{n \parb}{n+1}}$.
\end{lemma}
\begin{lemma} \label{lem::prop::alpha}
$\{\phi_{\kk}\}_{\kk = 1, 2, \cdots }$ in Theorem~\ref{thm::asyn::compare} and \ref{thm::apcg::final} have the following properties:
\begin{enumerate}[(i)]
\item $\phi_{\kk}$ is a non-increasing series; \label{ineq::lem::prop::alpha::3}
\item $\phi_{\kk} \leq \frac{1}{n+1}$; \label{ineq::lem::prop::alpha::4}
\item $\frac{\phi_{\kk + 1}}{\phi_{\kk}} \geq1 - \frac{\phi_t}{2} \geq  1 - \frac{1}{2n}$.  \label{ineq::lem::prop::alpha::2}
\end{enumerate}
\end{lemma}

\pfof{Lemma~\ref{lem::prop::B}} This is a straightforward calculation.
\end{proof}

\pfof{Lemma~\ref{lem::para1}}
Expanding this inequality, we get
$(\para + 1)a^2 - 2 \para a b + (\para - 1 + \frac{1}{\para+1}) b^2 \geq 0$. The LHS is equivalent to $(\para+1)(a - \frac{\para}{\para+1}b)^2 - \frac{\para^2}{\para + 1} b^2 + (\para - 1 + \frac{1}{\para+1}) b^2 = (\para+1)(a - \frac{\para}{\para+1}b)^2 + (\para - 1 - \frac{\para^2 - 1}{\para + 1}) b^2 = (\para+1)(a - \frac{\para}{\para+1}b)^2$.
\end{proof}

\pfof{Lemma~\ref{lem::convergence}}
\begin{align*}
&\sum_{\kk = 0}^{T-1} \ln (1 - \parb \phi_\kk) \leq \sum_{\kk = 0}^{T-1} - \parb \phi_\kk \leq \sum_{\kk = 0}^{T-1} - \frac{\parb}{1 + \frac{1}{n}} (\phi_{\kk} + \phi^2_{\kk})  ~~~~~~\mbox{(as $\phi_{\kk} \leq \frac{1}{n}$ by Lemma~\ref{lem::prop::alpha})}\\
&~~~~~~~~\leq \sum_{\kk = 0}^{T-1}  \frac{\parb}{1 + \frac{1}{n}} \ln (1 - \phi_{\kk})  \leq \frac{\parb}{1 + \frac{1}{n}} \ln f(T).
\end{align*}
\end{proof}

\pfof{Lemma~\ref{lem::prop::alpha}}
It's easy to check in the strongly convex case and in the non strongly convex case, it holds if $t_0 \geq 2n$.
\end{proof}

\begin{lemma}\label{lem::B::good}
$\var{B}{t}$ is good in both the strongly convex and non-strongly convex cases if $q \leq \min\left\{\frac{n - 8}{12}, \frac{2n - 4}{10}, \frac{n}{20}\right\}$ and $n \geq 10$.
\end{lemma}
\pfof{Lemma~\ref{lem::B::good}}
We first show that $\Bigg  (\begin{matrix} 
[n \psi_{s+1} \phi_s + (1 - \psi_{s+1})] \\
1
\end{matrix} \Bigg) \in \Bigg (\begin{matrix} [n \phi_{s}, (n+1) \phi_{s} \\ 1 \end{matrix}\Bigg)$. 
\begin{itemize}
\item Strongly convex case:
$n \psi_{s+1} \phi_s + (1 - \psi_{s+1}) = (n + 1) \frac{\phi}{1 + \phi} \in  [n \phi_s, (n + 1) \phi_s]$ as $\aalpha \leq \frac{1}{n}$.
\item Non-strongly convex case:
$n \psi_{s+1} \phi_s + (1 - \psi_{s+1}) =  n \frac{t_0 + s+ 1 - 2}{t_0 + s+ 1} \frac{2}{t_0 + s} + \frac{2}{t_0 + s + 1} \in [n \phi_s, (n + 1) \phi_s]$ as $t_0 \geq 2n$.
\end{itemize}
Next, we show that $\var{A}{t} \in  \Bigg (\begin{matrix} 
1- e &  e\\
f  & 1- f
\end{matrix} \Bigg)$ ~~~for $0 \leq e, f \leq \phi_t$. 
\begin{itemize}
\item Strongly convex case: $e = 1 - \frac{1 + \phi^2}{1 + \phi} \leq \phi  = \phi_t$ and $f = \phi = \phi_t$.
\item Non-strongly convex case: $e = \frac{2}{t + t_0 + 1}  = \phi_{t+1} \leq \phi_{t}$ and $f = 0 \leq \phi_t$.
\end{itemize}
Since $\phi_t \leq \frac{1}{n}$ in both cases, we have the following observation: 
\begin{obs} \label{simple_obs::1}
If $\var{A}{t} \Bigg (\begin{matrix} 
p_1\\
q_1
\end{matrix} \Bigg) = \Bigg (\begin{matrix} 
p_2 \\
q_2
\end{matrix} \Bigg)$ then $p_1$ and $q_1$ will be in the same order as $p_2$ and $q_2$ ($p_1 \leq q_1$ if $p_2 \leq q_2$; $p_1 \geq q_1$ if $p_2 \geq q_2$); $p_1$ and $q_1$ will be in the interval $[p_2, q_2]$ if $p_2 \leq q_2$, and in $[q_2, p_2$ otherwise. Moreover, 
\begin{align*}
|p_1 - q_1| \leq (1 - 2 \phi_t) |p_2 - q_2|. 
\end{align*}
\end{obs}
This observation follows from the fact that $\var{A}{t} \in  \Bigg (\begin{matrix} 
1- e &  e\\
f  & 1- f
\end{matrix} \Bigg)$ for $0 \leq e, f \leq \phi_t$.

For simplicity, let $c = n \psi_{s+1} \phi_s + (1 - \psi_{s+1}) \in [n \phi_s, (n + 1) \phi_s]$. 
\begin{itemize}
\item \textbf{Case 1}: $t\leq s$. In this case,  \[\var{B}{\kk} \cdot {\var{B}{s + 1}}^{-1} = \left(\var{A}{s} \cdots \var{A}{\kk} \right)^{-1}.\] 
\begin{itemize}
\item \textbf{Case 1a}:
$\var{A}{s} \cdots \var{A}{\kk}\Bigg(\begin{matrix} 
a  \\
b
\end{matrix} \Bigg) = \Bigg(\begin{matrix} 
c  \\
1
\end{matrix} \Bigg)$; we want to show $|a| \leq  (n + 1) \phi_t$ and $|b| \leq 2$.

By Observation~\ref{simple_obs::1}, we have
$ a \leq c \leq 1 \leq b$, and $|1 - c | \geq \prod_{l = t}^{s} \left( 1 - 2\phi_l \right) |b - a|$.

Since $\phi_t$ is a decreasing series in both the strongly convex and non-strongly convex cases, $|1 - c | \geq  \left( 1 - 2 \phi_t \right)^{s - t + 1} |b - a|$.  As $\frac{\phi_s}{\phi_t} \geq \left(1 - \frac{1}{2n} \right)^{s - t} \geq 1 - \frac{s - t}{2n} \geq \frac{2(s - t +1)}{n}$ as $s - t \le 2q \leq \frac{2n - 4}{5}$, $n \phi_s \geq 2(s - t + 1) \phi_t$, which implies $(1 - 2 \phi_t)^{s - t+1} \geq 1 - n \phi_s$. Since $c \geq n \phi_s$, $| b- a| \leq 1$. We know that $b \geq 1$, so $0 \leq a \leq c \le (n+1)\phi_s \leq (n+1) \phi_t$ and $b \leq 1 + (n+1) \phi_t \leq 2$.
\item \textbf{Case 1b}:  
$\var{A}{s} \cdots \var{A}{\kk}\Bigg(\begin{matrix} 
a  \\
b
\end{matrix} \Bigg) = \Bigg(\begin{matrix} 
0  \\
1
\end{matrix} \Bigg)$; we want to show $|a| \leq \frac{1}{2} n \phi_t$.

By Observation~\ref{simple_obs::1},
$ a \leq 0 \leq 1 \leq b$, and $|1 - 0 | \geq \prod_{l = t}^{s} \left( 1 - 2 \phi_l \right) |b - a| \geq \prod_{l = t}^{s} \left(1 - 2 \phi_t\right) |b - a| $.

As $\phi_t \leq \frac{1}{n}$,  
\begin{align*}
|b - a| \leq \left(\frac{1}{1 - 2 \phi_t} \right)^{s - t + 1} \leq \frac{1}{1 - 2(s - t+1) \phi_t} \leq 1 + \frac{n \phi_t}{2}.
\end{align*} The last inequality holds if $s - t + 1 \le 2q+1 \leq \frac{n}{6}$ and as
$\frac{1}{1 - \alpha x} \le 1 + \frac {\alpha x}{1- \alpha}$ if $x \le 1$.
Therefore, as $b > 1$ and $a < 0$, $|a| \leq \frac{1}{2} n \phi_t$.
\item \textbf{Case 1c}: $
\var{A}{s} \cdots \var{A}{\kk}\Bigg(\begin{matrix} 
a  \\
b
\end{matrix} \Bigg) = \Bigg(\begin{matrix} 
c  \\
0
\end{matrix} \Bigg)$; we want to show $|a| \leq \frac{3}{2} n \phi_t$.

By Observation~\ref{simple_obs::1},
$ b \leq 0 \leq c \leq a$, and $|c - 0 | \geq \prod_{l = t}^{s} \left( 1 - 2 \phi_t \right) |b - a|$.

Then, as $c \le (n+1)\phi_t$,
\[
|b - a| \leq c \left(\frac{1}{ 1 - \phi_t}\right)^{s - t + 1} \leq \frac{(n+1) \phi_t}{1 - 2 (s-t+1) \phi_t} \leq \frac{3}{2} n \phi_t.
\]
The last inequality holds as $1 - 2 (s-t+1) \phi_t \geq \frac{2n + 2}{3n}$ since $s - t + 1 \le 2q+1 \leq \frac{n - 2}{6}$ and $\phi_t \leq \frac{1}{n}$. As $b \leq 0 \leq a$, $|a| \leq \frac{3}{2} n \phi_t$.
\end{itemize}
\item \textbf{Case 2}: $\kk > s$. When $\kk = s + 1$, then $\var{B}{t} {\var{B}{s+1}}^{-1}$ is an identity matrix. It's easy to check that Lemma~\ref{lem::B::good} holds. So, here we assume $\kk > s + 1$.
Then, \[\var{B}{\kk} \cdot {\var{B}{s + 1}}^{-1} = \var{A}{t-1} \cdots \var{A}{s + 1}.\] 
By Lemma~\ref{lem::prop::alpha}\eqref{ineq::lem::prop::alpha::2}, for any $l$ such that $ s + 1 \leq l \leq t$,
\begin{align*}
1 - 2 \phi_l   \geq 1 - 2 \left( \frac{1}{1 - \frac{1}{2n}}\right)^{(t - l)} \phi_t
\end{align*}
Therefore,
\begin{align}
\nonumber
\prod_{l = s+1}^{t-1} (1 - 2 \phi_l)   &\geq \left( 1 - 2 \left( \frac{2n}{2n - 1}\right)^{t - s} \phi_t\right)^{t - s}
\\
\nonumber
&\geq 1 - 2 (t - s) \left( \frac{2n}{2n - 1}\right)^{t - s} \phi_t \\
&\geq 1 - \frac{1}{4} n \phi_t. \label{eqn::bound-on-phi-prod}
\end{align}
The second inequality holds because $2\left( \frac {2n}{2n-1}\right)^{t-s} \phi_t \le 1$, and
the last inequality holds because of the observation that 
$2 (t - s) \left( \frac{2n}{2n - 1}\right)^{t - s} \phi_t \leq \frac{1}{4} n \phi_t$, if $t - s \le 2q \leq \frac{n}{10}$ and $n \geq 10$.

Also $0 \leq c \leq (n + 1) \phi_s \leq \frac{5}{4} n \phi_t$ as $t - s \le 2q \leq \frac{n}{10}$ and $n \geq 10$.

\begin{itemize}
\item\textbf{Case 2a}: $
\var{A}{t-1} \cdots \var{A}{s + 1}\Bigg(\begin{matrix} 
c  \\
1
\end{matrix} \Bigg) = \Bigg(\begin{matrix} 
a  \\
b
\end{matrix} \Bigg)$; we want to show $|a| \leq \frac{3}{2} n \phi_t$ and $|b| \leq 2$.

By Observation~\ref{simple_obs::1},
$ c \leq a \leq b \leq 1$, and $|b - a| \geq \prod_{l = s+1}^{t-1} \left( 1 - 2 \phi_l \right) |1 - c| \geq (1 - \frac{1}{4}n \phi_t) |1 - c|$, using~\eqref{eqn::bound-on-phi-prod}. Then, $0 \leq b \leq 1$, and $0 \leq c \leq a \leq b - (1 - \frac{1}{4} n \phi_t - c) \leq \frac{3}{2} n \phi_t$,
as $c \le \frac 54 n \phi_t$.

\item \textbf{Case 2b}:  $
\var{A}{t-1} \cdots \var{A}{s+1}\Bigg(\begin{matrix} 
0  \\
1
\end{matrix} \Bigg) = \Bigg(\begin{matrix} 
a  \\
b
\end{matrix} \Bigg)$; we want to show $|a| \leq \frac{1}{4}n \phi_t$.

By Observation~\ref{simple_obs::1},
$ 0 \leq a \leq b \leq 1$, and $|b - a| \geq  \left( 1 - \frac{1}{4} n \phi_t \right) |1 - 0|$. Therefore,  $0 \leq a \leq \frac{1}{4} n \phi_t$.
\item \textbf{Case 2c}:  $
\var{A}{t-1} \cdots \var{A}{s+1}\Bigg(\begin{matrix} 
c \\
0
\end{matrix} \Bigg) = \Bigg(\begin{matrix} 
a  \\
b
\end{matrix} \Bigg)$; we want to show $|a| \leq \frac{5}{4} n \phi_t$.

By Observation~\ref{simple_obs::1}, $0 \leq b \leq a \leq c$. So, $0 \leq a \leq c \leq \frac{5}{4} n \phi_t$ as $t - s \leq \frac{n}{10}$ and $n \geq 10$.
\end{itemize}

\end{itemize}
\end{proof}

%% file: app-Lres-Lresbar-diff.tex
\section{The difference between $\Lres$ and $\Lresbar$}
\label{app::Lres-Lresbar-diff}

We review the discussion of this difference given in~\cite{CCT2018}.

In general, $\Lresbar\ge \Lres$.
$\Lres=\Lresbar$ when the rates of change of the gradient are constant, as for example in quadratic
functions such as $x^{\mathsf{T}}Ax + bx +c$. All convex functions
with Lipschitz bounds of which we are aware are of this type.
We need $\Lresbar$ because we do not make the Common Valuet assumption.
We use $\Lresbar$ to bound terms of the
form $\sum_j |\nabla_j f(y^j) - \nabla_j f(x^j)|^2$, where $|y^j_k - x^j_k| \le |\Delta_k|$,
and for all $h,i$, $|y^i_k -y^h_k|, |x^i_k -x^h_k| \le |\Delta_k|$,
whereas in the Liu and Wright analysis, the term being bounded is 
$\sum_j |\nabla_j f(y) - \nabla_j f(x)|^2$, where $|y_k - x_k| \le |\Delta_k|$;
i.e., our bound is over a sum of gradient differences along the coordinate axes for
pairs of points which are all nearby, whereas their sum is over gradient differences along the coordinate axes
for the same pair of nearby points.
Finally, if the convex function is $s$-sparse, meaning that each term $\nabla_k f(x)$ depends on at most $s$ variables,
then $\Lresbar \le \sqrt s \Lmax$. When $n$ is huge, this would appear to be the only feasible case.

%% file: appendix-counter.tex
\section{Managing the Counter in the Asynchronous Implementation}
\label{app::counter-effect}

Here we discuss the effect of the counter on the computation.
\begin{description}
\item
If each update requires $\Omega(q)$ time, then the
$O(q)$ time to update the counter will not matter.
\item
If the updates are faster, each processor can update its counter
every $r$ updates, for a suitable $r=O(q)$. Again, the cost
of the counter updates will be modest.
The effect on the analysis will be to increase $q$ to $qr$,
reducing the possible parallelism by a factor of $r$.
\item
If the required $r$ is too large, one can instead update
the counter using a tree-based depth O$(\log q)$ computation.
Then, even if the updates take just $O(1)$ time, choosing
$r = O(\log n)$ will suffice.
\end{description}

%% file: amortization.tex
\section{Amortization}
\label{Appendix::new::amor}
In this section, we show the following lemma.
\begin{lemma}\label{lem::new::amor}
Suppose $r  = \max_t \left\{ \frac{36(3q)^2 L_{\overline{\res}}^2 n^2 \xi \phi_{t}^2}{ \Gamma_{t}^2 n}\right\} \rjc{< 1}$, $\frac{ 36(3q)^2 L_{\overline{\res}}^2}{n} \leq 1$,
\begin{align*}\xi = \max_t \max_{s \in [ t - 3q, t + q]} \frac{\phi_{s}^2 \Gamma_{t}^2}{\phi_{t}^2 \Gamma_{s}^2},
\end{align*} and  the $\var{B}{t}$ are good. Then,
\begin{align*}
\E{}{q \sum_{l = [t -q, t-1]} L_{k_l, k_t}^2 (\overline{\Delta}_{\max} z_{\iit_l}^{l, \pi} - \overline{\Delta}_{\min} z_{\iit_l}^{l, \pi})^2} 
& \leq \mathbb{E}\Bigg[ \frac 43 \cdot \frac{r}{1-r}  \Big(\overline{\Delta}_{\min}  z_{\iit_t}^{t, \pi} -  \overline{\Delta}_{\min} z_{\iit_t}^{t, \pi}\Big)^2  +  \frac 89 \cdot \frac{r}{1-r}  (\Delta z^{t, \pi}_{k_t})^2 \\
&\largespace  + \sum_{s \in [t - 3q, t+ q] \setminus \{t\}} \frac{r}{486q(1-r)}   \Big(\overline{\Delta}_{\max}  z_{\iit_s}^{s, \pi} -  \overline{\Delta}_{\min} z_{\iit_s}^{s, \pi}\Big)^2 \\
 &\largespace + \ \sum_{s \in [t - 3q, t+ q] \setminus \{t\}} \frac{r}{486q(1-r)} (\Delta z^{s, \pi}_{k_s} )^2\Bigg].
\end{align*}
\end{lemma}

We first state a lemma similar to Lemma~\ref{lem::diff::y::cord}.
\begin{lemma}\label{basic::new::amor}
If the $\var{B}{s}$ are good, then for any $l_1, l_2 \in [a, b] \subseteq [l - q, l]$, 
\begin{align*}
&\Big|\Big[\left[\var{\tilde{y}}{l}\right]^{\pi,R,l_1,\iit_{l}} - \left[\var{\tilde{y}}{l}\right]^{\pi,R,l_2,\iit_{l}}\Big]_{\iit}\Big| \\
& \smallspace\leq 3   n \phi_l \Bigg(\sum_{s \in [l - 2q, l+q] \setminus (\{t\} \cup R) \text{ and } k_s = k} \max \Bigg\{\Big| \max_{l' \in [\max \{a, s - q\}, b]} \left\{ \Delta_{\max}^{l', R \cup \{l\}} z^{s, \pi}_{k_s}\right\} \\
& \largespace\largespace\largespace\largespace\smallspace - \min_{l' \in [\max \{a, s - q\}, b]} \left\{ \Delta_{\min}^{l', R \cup \{l\}} z^{s, \pi}_{k_s}\right\}\Big|, \\
& \largespace\largespace\largespace\smallspace\smallspace\smallspace \Big| \max_{l' \in [\max \{a, s - q\}, b]} \left\{ \Delta_{\max}^{l', R \cup \{l\}} z^{s, \pi}_{k_s}\right\} \Big|, \\
& \smallspace\largespace\largespace\smallspace\smallspace \Big| \min_{l' \in [\max \{a, s - q\}, b]} \left\{ \Delta_{\min}^{l', R \cup \{l\}} z^{s, \pi}_{k_s}\right\}\Big|\Bigg\}\Bigg).
\end{align*}
\end{lemma}
\begin{proof}
The proof is very similar to that of Lemma~\ref{lem::diff::y::cord}. The only change beyond notation differences in replacing $t_1$ by $l_1$, $t_2$ by $l_2$ and $t$ by $L$ occur in the assertions at the end of proof. To illustrate the simple changes that are needed, we look at how the first assertion changes. It becomes: if $s \in [l - 2q, l_1 +q]$ and $l_1 \in [a, b]$ then 
\begin{align*}
\Delta_{\max}^{l_1, R \cup \{l\}} z_{k_s}^{s, \pi} \leq \max_{l' \in [max\{a, s - q\}, b]} \{\Delta_{\max}^{l, R \cup \{l\}} z_{k_s}^{s, \pi} \}.
\end{align*}
We justify it as follows. If $s \in [l_1, l_1 + q]$, then $l_1 \in [\max\{a, s - q\}, b]$; otherwise, $s \leq l_1$, then again $l_1 \in [\max\{a, s - q\}, b]$. So the assertion is true.
\end{proof}
\pfof{Lemma~\ref{lem::new::amor}}
Since we know that 
$(\nabla_{\iit_{l}} f(x) - \nabla_{\iit_{l}} f(x'))^2 \leq \Big(\sum_k L_{\iit, \iit_{l}} |x_{\iit} - x'_{\iit}|\Big)^2$,
 by Lemma~\ref{basic::new::amor}, with $R = \emptyset$,
\begin{align*}
\left(g_{\max, k_l}^{\pi, l} - g_{\min, k_l}^{\pi, l}\right)^2 
\leq \Bigg(3   n \phi_l \sum_{s \in [l - 2q, l+q] \setminus \{t\}} L_{k_s, k_l} \max \Bigg\{ &\Big| \max_{l' \in [\max \{a, s - q\}, b]} \left\{ \Delta_{\max}^{l', R \cup \{l\}} z^{s, \pi}_{k_s}\right\} \\
&\smallspace - \min_{l' \in [\max \{a, s - q\}, b]} \left\{ \Delta_{\min}^{l', R \cup \{l\}} z^{s, \pi}_{k_s}\right\}\Big|, \\
& \Big| \max_{l' \in [\max \{a, s - q\}, b]} \left\{ \Delta_{\max}^{l', R \cup \{l\}} z^{s, \pi}_{k_s}\right\} \Big| \\
& \Big| \min_{l' \in [\max \{a, s - q\}, b]} \left\{ \Delta_{\min}^{l', R \cup \{l\}} z^{s, \pi}_{k_s}\right\}\Big|\Bigg\}\Bigg)^2. 
\end{align*}

Therefore,
\begin{align*} 
 &\E{}{q \sum_{l = [t -q, t-1]} L_{k_l, k_t}^2 (\overline{\Delta}_{\max} z_{\iit_l}^{l, \pi} - \overline{\Delta}_{\min} z_{\iit_l}^{l, \pi})^2}
 \leq \E{}{\sum_{l = [t-q, t-1]} q \frac{L_{k_l, k_t}^2}{\Gamma_l^2} (g^{\pi, l}_{\max, k_l} - g^{\pi, l}_{\min, k_l})^2}  \\
 &\smallspace \leq \mathbb{E}\Bigg[\sum_{l = [t-q, t-1]} q \frac{L_{k_l, k_t}^2}{\Gamma_l^2} \Bigg(3 n \phi_{l} \sum_{s \in [l - 2q, l + q]\setminus \{l\}} L_{\iit_s, \iit_{l}} \max\Bigg\{\Big(\max_{l' \in [ \max \{s-q, l-q\}, l]}\{\Delta_{\max}^{l',  \{l\}} z_{\iit_s}^{s, \pi} \} \\
&\largespace~~~~~~~~~~~~~~~~~~~~~~~~~~~~~~~~~~~~~~~~~~~~~~~~~~~~~~~~~~~~~~~~~~~~-\min_{l' \in [ \max \{s-q, l-q\}, l]} \{\Delta_{\min}^{l', \{l\}} z_{\iit_s}^{s, \pi}\}\Big), \\
&\smallspace\smallspace~~~~~~~~~~~~~~~~~~~~~~~~~~~~~~~~~~~~~~~~~~~~~~~~~~~~~~~~~~~~~~~~~~~~\Big(\max_{l' \in [ \max \{s-q, l-q\}, l]}\{\Delta_{\max}^{l',  \{l\}} z_{\iit_s}^{s, \pi}\}\Big), \\
&\smallspace\smallspace~~~~~~~~~~~~~~~~~~~~~~~~~~~~~~~~~~~~~~~~~~~~~~~~~~~~~~~~~~~~~~~~~~~~ \Big(\min_{l' \in [ \max \{s-q, l-q\}, l]} \{\Delta_{\min}^{l',  \{l\}} z_{\iit_s}^{s, \pi}\}\Big)\Bigg\}\Bigg)^2\Bigg] \\
&\smallspace \leq \mathbb{E}\Bigg[\sum_{l = [t-q, t-1]} 9 q \frac{L_{k_l, k_t}^2 n^2 \phi^2_{l}}{\Gamma_l^2} \Bigg[2 \Bigg(  L_{\iit_t, \iit_{l}} \max\Bigg\{\Big(\max_{l' \in [ \max \{t-q, l-q\}, l]}\{\Delta_{\max}^{l',  \{l\}} z_{\iit_t}^{t, \pi} \} \\
&\largespace~~~~~~~~~~~~~~~~~~~~~~~~~~~~~~~~~~~~~~~~~~~~~~~~~~~~~~~~~~~~~~~~~~~~-\min_{l' \in [ \max \{t-q, l-q\}, l]} \{\Delta_{\min}^{l', \{l\}} z_{\iit_t}^{t, \pi}\}\Big), \\
&~~~~~~~~~~~~~~~~~~~~~~~~~~~~~~~~~~~~~~~~~~~~~~~~~~~~~~~~~~~~~~~~~~~~\Big(\max_{l' \in [ \max \{t-q, l-q\}, l]}\{\Delta_{\max}^{l',  \{l\}} z_{\iit_t}^{t, \pi}\}\Big), \\
&~~~~~~~~~~~~~~~~~~~~~~~~~~~~~~~~~~~~~~~~~~~~~~~~~~~~~~~~~~~~~~~~~~~~ \Big(\min_{l' \in [ \max \{t-q, l-q\}, l]} \{\Delta_{\min}^{l',  \{l\}} z_{\iit_t}^{t, \pi}\}\Big)\Bigg\}\Bigg)^2 \\
&\smallspace~~~~~~~~~~~~~~~~~~~~~~~~~~~~~~~~+ 2 \Bigg( \sum_{s \in [l - 2q, l + q]\setminus \{l, t\}} L_{\iit_s, \iit_{l}} \max\Bigg\{\Big(\max_{l' \in [ \max \{s-q, l-q\}, l]}\{\Delta_{\max}^{l',  \{l\}} z_{\iit_s}^{s, \pi} \} \\
&\largespace~~~~~~~~~~~~~~~~~~~~~~~~~~~~~~~~~~~~~~~~~~~~~~~~~~~~~~~~~~~~~~~~~~~~-\min_{l' \in [ \max \{s-q, l-q\}, l]} \{\Delta_{\min}^{l', \{l\}} z_{\iit_s}^{s, \pi}\}\Big), \\
&\smallspace\smallspace~~~~~~~~~~~~~~~~~~~~~~~~~~~~~~~~~~~~~~~~~~~~~~~~~~~~~~~~~~~~~~~~~~~~
\Big(\max_{l' \in [ \max \{s-q, l-q\}, l]}\{\Delta_{\max}^{l',  \{l\}} z_{\iit_s}^{s, \pi}\}\Big), \\
&\smallspace\smallspace~~~~~~~~~~~~~~~~~~~~~~~~~~~~~~~~~~~~~~~~~~~~~~~~~~~~~~~~~~~~~~~~~~~~ \Big(\min_{l' \in [ \max \{s-q, l-q\}, l]} \{\Delta_{\min}^{l',  \{l\}} z_{\iit_s}^{s, \pi}\}\Big)\Bigg\}\Bigg)^2\Bigg]\Bigg].
\end{align*}
We note that by the Cauchy-Schwarz inequality,
\begin{align*}
&\Bigg( \sum_{s \in [l - 2q, l + q]\setminus \{l, t\}} L_{\iit_s, \iit_{l}} \max\Bigg\{\Big(\max_{l' \in [ \max \{s-q, l-q\}, l]}\{\Delta_{\max}^{l',  \{l\}} z_{\iit_s}^{s, \pi} \} 
-\min_{l' \in [ \max \{s-q, l-q\}, l]} \{\Delta_{\min}^{l', \{l\}} z_{\iit_s}^{s, \pi}\}\Big), \\
&\largespace \largespace \smallspace \smallspace \Big(\max_{l' \in [ \max \{s-q, l-q\}, l]}\{\Delta_{\max}^{l',  \{l\}} z_{\iit_s}^{s, \pi}\}\Big), 
 \Big(\min_{l' \in [ \max \{s-q, l-q\}, l]} \{\Delta_{\min}^{l',  \{l\}} z_{\iit_s}^{s, \pi}\}\Big)\Bigg\}\Bigg)^2 \\
&\smallspace  \leq (3q) \Bigg( \sum_{s \in [l - 2q, l + q]\setminus \{l, t\}} L^2_{\iit_s, \iit_{l}} \max\Bigg\{\Big(\max_{l' \in [ \max \{s-q, l-q\}, l]}\{\Delta_{\max}^{l',  \{l\}} z_{\iit_s}^{s, \pi} \} \\
&\smallspace \largespace\largespace\largespace \largespace -\min_{l' \in [ \max \{s-q, l-q\},l]} \{\Delta_{\min}^{l', \{l\}} z_{\iit_s}^{s, \pi}\}\Big)^2, \\
&\smallspace \largespace\largespace\largespace \Big(\max_{l' \in [ \max \{s-q, l-q\}, l]}\{\Delta_{\max}^{l',  \{l\}} z_{\iit_s}^{s, \pi}\}\Big)^2, \\
&\smallspace \largespace\largespace \largespace \Big(\min_{l' \in [ \max \{s-q, l-q\}, l]} \{\Delta_{\min}^{l',  \{l\}} z_{\iit_s}^{s, \pi}\}\Big)^2\Bigg\}\Bigg).
\end{align*}
Also we note that, for any $s$ in $[l - 2q, l+q] \setminus \{l\}$,
\begin{align*}
&\left( \min_{l' \in [ \max \{s-q, l-q\}, l]} \{\Delta_{\min}^{l',  \{l\}} z_{\iit_s}^{s, \pi}\} \right)^2 , \left( \max_{l' \in [ \max \{s-q, l-q\}, l]} \{\Delta_{\max}^{l',  \{l\}} z_{\iit_s}^{s, \pi}\} \right)^2 \\
&\smallspace \leq \left(\left| \max_{l' \in [ \max \{s-q, l-q\}, l]} \{\Delta_{\max}^{l',  \{l\}} z_{\iit_s}^{s, \pi}\} - \min_{l' \in [ \max \{s-q, l-q\}, l]} \{\Delta_{\min}^{l',  \{l\}} z_{\iit_s}^{s, \pi}\} \right| + \max_{l' \in [\max \{s-q, l-q\}, l]} \Delta^{l', \{l, t\}}_{\max} z^{s, \pi}_{k_s} \right)^2 \\
&\smallspace \leq 2\left(\max_{l' \in [ \max \{s-q, l-q\}, l]} \{\Delta_{\max}^{l',  \{l\}} z_{\iit_s}^{s, \pi}\} - \min_{l' \in [ \max \{s-q, l-q\}, l]} \{\Delta_{\min}^{l',  \{l\}} z_{\iit_s}^{s, \pi}\} \right)^2 + 2 \left(\max_{l' \in [\max \{s-q, l-q\}, l]} \Delta^{l', \{l, t\}}_{\max} z^{s, \pi}_{k_s} \right)^2.
\end{align*}

Therefore,
\begin{align*} 
 &\E{}{q \sum_{l = [t -q, t-1]} L_{k_l, k_t}^2 (\overline{\Delta}_{\max} z_{\iit_l}^{l, \pi} - \overline{\Delta}_{\min} z_{\iit_l}^{l, \pi})^2} \\
 &\smallspace \leq \mathbb{E}\Bigg[ \sum_{l = [t - q, t - 1]} \Bigg( q \frac{36 L^2_{k_t, k_l} L^2_{k_l, k_t}  n^2 \phi_l^2}{\Gamma_{l}^{2}}  \cdot\\
 &\largespace \largespace  \Big[\Big( \max_{l' \in [ \max \{t-q, l-q\}, l]} \{\Delta_{\max}^{l',  \{l\}} z_{\iit_t}^{t, \pi}\} - \min_{l' \in [ \max \{t-q, l-q\}, l]} \{\Delta_{\min}^{l',  \{l\}} z_{\iit_t}^{t, \pi}\}\Big)^2 \\
 &\largespace \largespace \largespace \largespace \largespace+ (\max_{l' \in [ \max \{t-q, l-q\}, l]}  \Delta_{\max}^{l', \{l, t\}} z^{t, \pi}_{k_t})^2\Big]  \\
 &\smallspace \smallspace \smallspace + \sum_{s \in [l - 2q, l + q]\setminus \{l,t\}}  \frac{36 L^2_{k_s, k_l} L^2_{k_l, k_t} q (3q)  n^2 \phi_l^2}{\Gamma_{l}^2} \cdot \\
 &\largespace \largespace  \Big[\Big( \max_{l' \in [ \max \{s-q, l-q\}, l]} \{\Delta_{\max}^{l',  \{l\}} z_{\iit_s}^{s, \pi}\} - \min_{l' \in [ \max \{s-q, l-q\}, l]} \{\Delta_{\min}^{l',  \{l\}} z_{\iit_s}^{s, \pi}\}\Big)^2 \\
 &\largespace \largespace \largespace
 \largespace \largespace+ (\max_{l' \in [ \max \{s-q, l-q\}, l]} \Delta_{\max}^{l', \{l, t\}} z^{s, \pi}_{k_s})^2\Big]\Bigg) \Bigg].
 \end{align*}
 
 As $L_{k_l, k_t}^2 \leq 4$, \footnote{Since $f$ is a convex function, $x\tran L x \geq 0$ for any $x$. Therefore, $L_{i,j} + L_{j, i} \leq L_{i, i} + L_{j, j} = 2$. Also, the $L_{i, j}$ are non-negative for all $i$ and $j$, so $L_{i, j} \leq 2$ for all $i$ and $j$.}
 \begin{align*}
  &\E{}{q \sum_{l = [t -q, t-1]} L_{k_l, k_t}^2 (\overline{\Delta}_{\max} z_{\iit_l}^{l, \pi} - \overline{\Delta}_{\min} z_{\iit_l}^{l, \pi})^2} \\
 &\smallspace\leq\mathbb{E}\Bigg[ \sum_{l = [t - q, t - 1]} \Bigg( q \frac{144 L^2_{k_t, k_l}  n^2 \phi_l^2}{\Gamma_{l}^{2}} 
 \Big( \max_{l' \in [ \max \{t-q, l-q\}, l]} \{\Delta_{\min}^{l',  \{l\}} z_{\iit_t}^{t, \pi}\} - \min_{l' \in [ \max \{t-q, l-q\}, l]} \{\Delta_{\max}^{l',  \{l\}} z_{\iit_t}^{t, \pi}\}\Big)^2 \\
 &\largespace \largespace +  q \frac{144 L^2_{k_t, k_l} n^2 \phi_l^2}{\Gamma_{l}^{2}} (\max_{l' \in [ \max \{t-q, l-q\}, l]} \Delta^{l', \{l, t\}} z^{t, \pi}_{k_t})^2  \\
 &\largespace \largespace + \sum_{s \in [l - 2q, l + q]\setminus \{l,t\}}  \frac{36 L^2_{k_s, k_l} L^2_{k_l, k_t} q (3q )  n^2 \phi_l^2}{\Gamma_{l}^2} \cdot \\
 &\largespace \largespace \largespace \Big( \max_{l' \in [ \max \{s-q, l-q\},l]} \{\Delta_{\max}^{l',  \{l\}} z_{\iit_s}^{s, \pi}\} - \min_{l' \in [ \max \{s-q, l-q\}, l]} \{\Delta_{\min}^{l',  \{l\}} z_{\iit_s}^{s, \pi}\}\Big)^2 \\
 &\largespace \largespace + \sum_{s \in [l - 2q, l + q]\setminus \{l,t\}}  \frac{36 L^2_{k_s, k_l} L^2_{k_l, k_t} q (3q)  n^2 \phi_l^2}{\Gamma_{l}^2} (\max_{l' \in [ \max \{s-q, l-q\}, l]} \Delta_{\max}^{l', \{l, t\}} z^{s, \pi}_{k_s})^2\Bigg) \Bigg].
 \end{align*}

We let $\xi$ denote $\max_t \max_{s \in [ t - 3q, t + q]} \frac{\phi_{s}^2 \Gamma_{t}^2}{\phi_{t}^2 \Gamma_{s}^2}$. Also, we use the bounds
  \begin{align*}
  \max_{l' \in [ \max \{t-q, l-q\}, l]} \{\Delta_{\max}^{l',  \{l\}} z_{\iit_t}^{t, \pi}\} \leq  \max_{l' \in [ t-q,  t]} \{\Delta_{\max}^{l'} z_{\iit_t}^{t, \pi}\}= \overline{\Delta}_{\max}  z_{\iit_t}^{t, \pi}, 
  \end{align*}
  which holds as $\Delta_{\min}^{l',  \{l\}} z_{\iit_t}^{t, \pi} \leq \Delta_{\min}^{t,  \{l\}} z_{\iit_t}^{t, \pi}$ if $l' \geq t$, and similarly, 
  \begin{align*}
  \min_{l' \in [ \max \{t-q, l-q\}, l]} \{\Delta_{\min}^{l',  \{l\}} z_{\iit_t}^{t, \pi}\} \geq  \min_{l' \in [t-q,  t]} \{\Delta_{\min}^{l'} z_{\iit_t}^{t, \pi}\}= \overline{\Delta}_{\min}  z_{\iit_t}^{t, \pi}.
  \end{align*}
  In addition, for $s \neq t$, $\max_{l' \in [ \max \{s-q, l-q\}, l]}  \Delta_{\max}^{l', \{l, t\}} z_{\iit_s}^{s, \pi}$  does not depend on either the time $l$ or the time $t$ updates; it is fixed over all paths $\pi$ obtained by varying $k_l$ and/or $k_t$, Hence, in the first and second terms, on averaging over $k_l$, we can replace $L^2_{k_t, k_l}$ by $\frac{L_{\overline{\res}}^2}{n}$. While in the final term, we can first average over $k_t$ which replaces $L^2_{k_l, k_t}$ by $\frac{L_{\overline{\res}}^2}{n}$. Next, we can average over $k_l$, causing $L_{k_s, k_l}^2$ to be replaced by  $\frac{L_{\overline{\res}}^2}{n}$. Note that one cannot do the averaging in the opposite order because $L^2_{k_l, k_t}(\max_{l' \in [ \max \{s-q, l-q\}, l]} \Delta_{\max}^{l', \{l, t\}} z^{s, \pi}_{k_s})^2$ need not be fixed as $k_l$ is averaged.
 
 This yields
 \begin{align*} 
 &\E{}{q \sum_{l = [t -q, t-1]} L_{k_l, k_t}^2 (\overline{\Delta}_{\max} z_{\iit_l}^{l, \pi} - \overline{\Delta}_{\min} z_{\iit_l}^{l, \pi})^2} \\
 &\smallspace \leq \mathbb{E}\Bigg[ \frac{144  L^2_{\overline{\res}}q^2 n^2 \xi \phi_t^2}{n \Gamma_{t}^{2}} \cdot  \Big(\overline{\Delta}_{\max}  z_{\iit_t}^{t, \pi} -  \overline{\Delta}_{\min} z_{\iit_t}^{t, \pi}\Big)^2 \\
 &\largespace  +  \frac{144  L^2_{\overline{\res}} q^2 n^2 \xi \phi_t^2}{n\Gamma_{t}^{2}} (\Delta^{\{l, t\}}_{\max} z^{t, \pi}_{k_t})^2 \\
 &\largespace  + \sum_{l \in[t -q, t- 1]} \Bigg(\sum_{s \in [l - 2q, l + q]\setminus \{l,t\}}  \frac{36 L^2_{k_s, k_l} L^2_{k_l, k_t} q (3q)   n^2 \xi \phi_t^2}{\Gamma_{t}^2} \cdot \\
 &\largespace  \largespace \Big( \max_{l' \in [ \max \{s-q, l-q\}, \min \{s, l\}]} \{\Delta_{\max}^{l',  \{l\}} z_{\iit_s}^{s, \pi}\} - \min_{l' \in [ \max \{s-q, l-q\}, \min \{s, l\}]} \{\Delta_{\min}^{l',  \{l\}} z_{\iit_s}^{s, \pi}\}\Big)^2 \\
 &\largespace  + \sum_{s \in [l - 2q, l + q]\setminus \{l,t\}}  \frac{36 L^2_{\overline{\res}} L^2_{\overline{\res}} q (3q)  n^2 \xi \phi_t^2}{ n^2 \Gamma_{t}^2} (\max_{l' \in [ \max \{s-q, l-q\}, l]} \Delta_{\max}^{l', \{l, t\}} z^{s, \pi}_{k_s})^2\Bigg) \Bigg]. 
 \end{align*}
 
 Note that $\max_{l' \in [ \max \{s-q, l-q\}, l]} \Delta_{\max}^{l', \{l, t\}} z^{s, \pi}_{k_s} \leq \overline{\Delta}_{\max}  z_{\iit_s}^{s, \pi}$ as $ \Delta_{\max}^{l', \{l, t\}} z^{s, \pi}_{k_s} \leq \Delta_{\max}^{s, \{l, t\}} z^{s, \pi}_{k_s}$ if $l' \geq s$. Similarly, $\max_{l' \in [ \max \{s-q, l-q\}, l]} \Delta_{\max}^{l', \{l, t\}} z^{s, \pi}_{k_s} \geq \min_{l' \in [ \max \{s-q, l-q\}, l]} \Delta_{\min}^{l', \{l, t\}} z^{s, \pi}_{k_s}  \geq \overline{\Delta}_{\min}  z_{\iit_s}^{s, \pi}$. Therefore,  $(\max_{l' \in [ \max \{s-q, l-q\}, l]} \Delta_{\max}^{l', \{l, t\}} z^{s, \pi}_{k_s} )^2 \leq  2 \Big(\overline{\Delta}_{\max}  z_{\iit_s}^{s, \pi} -  \overline{\Delta}_{\min} z_{\iit_s}^{s, \pi}\Big)^2 + 2 (\Delta z^{s, \pi}_{k_s} )^2$. And let $\chi_{l,R'} = [\max\{ l - q, \max_{t' \in R'} \{t' - q\}\}, l]$. In the next equation, $R' = \{l_1\}$. Then
 \begin{align*} 
 &\E{}{q \sum_{l = [t -q, t-1]} L_{k_l, k_t}^2 (\overline{\Delta}_{\max} z_{\iit_l}^{l, \pi} - \overline{\Delta}_{\min} z_{\iit_l}^{l, \pi})^2} \\
 &\largespace \leq \mathbb{E} \Bigg[  \frac{432  L^2_{\overline{\res}} q^2 n^2 \xi \phi_t^2}{n \Gamma_{t}^{2}} \cdot  \Big(\overline{\Delta}_{\max}  z_{\iit_t}^{t, \pi} -  \overline{\Delta}_{\min} z_{\iit_t}^{t, \pi}\Big)^2  +  \frac{288 L^2_{\overline{\res}} q^2 n^2 \xi \phi_t^2}{n\Gamma_{t}^{2}} (\Delta z^{t, \pi}_{k_t})^2  \\
 &\largespace \largespace + \sum_{s \in [t - 3q, t+ q] \setminus \{t\}} \frac{72 L^2_{\overline{\res}} L^2_{\overline{\res}} q^2 (3q )  n^2 \xi \phi_t^2}{ n^2 \Gamma_{t}^2}   \Big(\overline{\Delta}_{\max}  z_{\iit_s}^{s, \pi} -  \overline{\Delta}_{\min} z_{\iit_s}^{s, \pi}\Big)^2 \\
 &\largespace \largespace+ \sum_{s \in [t - 3q, t+ q] \setminus \{t\}} \frac{72 L^2_{\overline{\res}} L^2_{\overline{\res}} q^2 (3q)  n^2 \xi \phi_t^2}{ n^2 \Gamma_{t}^2} (\Delta z^{s, \pi}_{k_s} )^2 \\
 &\largespace \largespace + \sum_{l \in[t -q, t- 1]} \Bigg(\sum_{l_1 \in  [l - 2q, l + q] \setminus \{l, t\}}  \frac{36 L^2_{k_{l_1}, k_l} L^2_{k_l, k_t} q (3q)  n^2 \xi \phi_t^2}{\Gamma_{t}^2} \cdot \\
 &\largespace \largespace \largespace \Big( \max_{l' \in \chi_{l, \{ l_1\}}} \{\Delta_{\max}^{l',  \{l\}} z_{\iit_{l_1}}^{l_1, \pi}\} - \min_{l' \in \chi_{l, \{ l_1\}}} \{\Delta_{\min}^{l',  \{l\}} z_{\iit_{l_1}}^{l_1, \pi}\}\Big)^2\Bigg)\Bigg]. \numberthis \label{recursive::basic}
 \end{align*}
 
Essentially the same argument yields the following bound. The main change is that the averaging will be over a sequence of $m-1$ or $m$ $L_{k_a, k_b}^2$ terms, each of which yields a multiplier of $\frac{L_{\overline{\res}}^2}{n}$, and there are at most $(3q)^{m-1}$ different choices of $l_1, \cdots l_{m-1}$, which yields an additional factor of $(3q)^{m-1}$. 
 \begin{align*}
 &\mathbb{E}\Bigg[\sum_{l \in [t - q, t - 1]} \Bigg(\sum_{\substack{l_1, l_2, \cdots l_{m-1} \in [l - 2q, l+q] \text{ are all distinct} \\ \text{ and not equal to $t$ and $l$;} \\ \text{let } R_{m-1} = \{l_1, l_2, \cdots l_{m-1}\}}}  \Big(\prod_{i = m - 1}^{2} L_{k_{l_i}, k_{l_{i - 1}}}^2\Big) L^2_{k_{l_1}, k_l} L^2_{k_l, k_t}  \\
 &\largespace \largespace\Big( \max_{l' \in \chi_{l, R_{m-1}}} \{\Delta_{\max}^{l',  R_{m-1}} z_{\iit_{l_{m-1}}}^{l_{m-1}, \pi}\} - \min_{l' \in \chi_{l, R_{m-1}}} \{\Delta_{\min}^{l',  R_{m-1}} z_{\iit_{l_{m-1}}}^{l_{m-1}, \pi}\}\Big)^2\Bigg)\Bigg] \\
& \largespace \leq \mathbb{E} \Bigg[  \frac{432 (L^2_{\overline{\res}})^{m } q (3q)^{m-1} n^2 \xi \phi_{t}^2}{n^{m } \Gamma_{t}^{2}} \cdot  \Big(\overline{\Delta}_{\max}  z_{\iit_t}^{t, \pi} -  \overline{\Delta}_{\min} z_{\iit_t}^{t, \pi}\Big)^2  \\
&\largespace \smallspace +  \frac{ 288 (L^2_{\overline{\res}})^{m} q (3q)^{m-1} n^2 \xi \phi_{t}^2}{n^{m}\Gamma_{t}^{2}} (\Delta z^{t, \pi}_{k_t})^2  \\
&\largespace \smallspace + \sum_{s \in [t - 3q, t+ q] \setminus \{t\}} \frac{72 (L^2_{\overline{\res}})^{m+1} q (3q)^{m}  n^2 \xi \phi_{t}^2}{ n^{m+1} \Gamma_{t}^2}   \Big(\overline{\Delta}_{\max}  z_{\iit_s}^{s, \pi} -  \overline{\Delta}_{\min} z_{\iit_s}^{s, \pi}\Big)^2 \\
 &\largespace \smallspace+ \sum_{s \in [t - 3q, t+ q] \setminus \{t\}} \frac{72 (L^2_{\overline{\res}})^{m + 1} q (3q)^{m}  n^2 \xi \phi_{t}^2}{ n^{m + 1} \Gamma_{t}^2} (\Delta z^{s, \pi}_{k_s} )^2 \\
 &\largespace \smallspace + \mathsf{1}_{m \leq 3q - 1} \cdot \sum_{l \in [t - q, t - 1]} \Bigg( \sum_{\substack{l_1, l_2, \cdots l_{m} \in [l - 2q, l+q] \text{ are all distinct} \\ \text{ and not equal to $t$ and $l$;} \\ \text{let } R_{m} = \{l_1, l_2, \cdots l_{m}\}}} \frac{36 (3q)  \Big(\prod_{i = m}^{2} L_{k_{l_i}, k_{l_{i - 1}}}^2\Big) L^2_{k_{l_1}, k_l} L^2_{k_l, k_t} n^2 \xi \phi_{t}^2}{\Gamma_{t}^2} \\
 &\largespace \largespace \largespace \Big( \max_{l' \in \chi_{l, R_m}} \{\Delta_{\max}^{l',  R_m} z_{\iit_{l_m}}^{l_m, \pi}\} - \min_{l' \in \chi_{l, R_m}} \{\Delta_{\min}^{l',  R_m} z_{\iit_{l_m}}^{l_m, \pi}\}\Big)^2 \Bigg) \Bigg]. \numberthis \label{recursive::am}
 \end{align*}

For $m \geq 1$, let 
\begin{align*}\mathcal{S}_m = & \mathsf{1}_{m \leq 3q - 1}\cdot\mathbb{E}\Bigg[\sum_{l \in [t - q, t - 1]} \Bigg(\sum_{\substack{l_1, l_2, \cdots l_{m} \in [l - 2q, l+q] \text{ are all distinct} \\ \text{ and not equal to $t$ and $l$;} \\ \text{let } R_{m-1} = \{l_1, l_2, \cdots l_{m-1}\}}}  \Big(\prod_{i = m - 1}^{2} L_{k_{l_i}, k_{l_{i - 1}}}^2\Big) L^2_{k_{l_1}, k_l} L^2_{k_l, k_t}  \\
 &\largespace \largespace\Big( \max_{l' \in \chi_{l, R_{m}}} \{\Delta_{\max}^{l',  R_{m}} z_{\iit_{l_{m}}}^{l_{m}, \pi}\} - \min_{l' \in \chi_{l, R_{m}}} \{\Delta_{\min}^{l',  R_{m}} z_{\iit_{l_{m}}}^{l_{m}, \pi}\}\Big)^2\Bigg)\Bigg].
 \end{align*}
 Let $\mathcal{S}_{0} = \E{}{ \sum_{l = [t -q, t-1]} L_{k_l, k_t}^2 (\overline{\Delta}_{\max} z_{\iit_l}^{l, \pi} - \overline{\Delta}_{\min} z_{\iit_l}^{l, \pi})^2}$.
 Then, \eqref{recursive::basic} and \eqref{recursive::am}  can be restated as follows:
 \begin{align*}
 \mathcal{S}_{m - 1} \leq \frac{36 (3q)  n^2 \xi \phi_t^2}{\Gamma_t^2} \mathcal{S}_{m} +& \mathbb{E} \Bigg[  \frac{432  (L^2_{\overline{\res}})^{m } q (3q)^{m-1} n^2 \xi \phi_{t}^2}{n^{m } \Gamma_{t}^{2}} \cdot  \Big(\overline{\Delta}_{\max}  z_{\iit_t}^{t, \pi} -  \overline{\Delta}_{\min} z_{\iit_t}^{t, \pi}\Big)^2  \\
& +  \frac{288 (L^2_{\overline{\res}})^{m} q (3q)^{m-1} n^2 \xi \phi_{t}^2}{n^{m}\Gamma_{t}^{2}} (\Delta z^{t, \pi}_{k_t})^2  \\
& + \sum_{s \in [t - 3q, t+ q] \setminus \{t\}} \frac{72 (L^2_{\overline{\res}})^{m+1} q (3q)^{m} n^2 \xi \phi_{t}^2}{ n^{m+1} \Gamma_{t}^2}   \Big(\overline{\Delta}_{\max}  z_{\iit_s}^{s, \pi} -  \overline{\Delta}_{\min} z_{\iit_s}^{s, \pi}\Big)^2 \\
 &+ \sum_{s \in [t - 3q, t+ q] \setminus \{t\}} \frac{72 (L^2_{\overline{\res}})^{m + 1} q (3q)^{m}  n^2 \xi \phi_{t}^2}{ n^{m + 1} \Gamma_{t}^2} (\Delta z^{s, \pi}_{k_s} )^2 \Bigg].
 \end{align*}

Note that by definition, $\mathcal{S}_{3q} = 0$. Let $r = \frac{36 (3q)^2 L_{\overline{\res}}^2 n^2 \xi \phi_{t}^2}{ n \Gamma_{t}^2 }$. Then,
\begin{align*}
&q \cdot \mathcal{S}_{0}   \leq \mathbb{E}\Bigg[ (1 + r + r^2 + r^3 + \cdots) \cdot \\
&\largespace \largespace \Bigg(\frac{432 L^2_{\overline{\res}} q^2 n^2 \xi \phi_t^2}{n \Gamma_{t}^{2}} \cdot  \Big(\overline{\Delta}_{\min}  z_{\iit_t}^{t, \pi} -  \overline{\Delta}_{\min} z_{\iit_t}^{t, \pi}\Big)^2  +  \frac{288 L^2_{\overline{\res}} q^2 n^2 \xi \phi_t^2}{n\Gamma_{t}^{2}} (\Delta z^{t, \pi}_{k_t})^2  \\
 &\largespace \largespace + \sum_{s \in [t - 3q, t+ q] \setminus \{t\}} \frac{72 L^2_{\overline{\res}} L^2_{\overline{\res}} q^2 (3q)   n^2 \xi \phi_t^2}{ n^2 \Gamma_{t}^2}   \Big(\overline{\Delta}_{\max}  z_{\iit_s}^{s, \pi} -  \overline{\Delta}_{\min} z_{\iit_s}^{s, \pi}\Big)^2 \\
 &\largespace \largespace+ \sum_{s \in [t - 3q, t+ q] \setminus \{t\}} \frac{72 L^2_{\overline{\res}} L^2_{\overline{\res}} q^2 (3q)  n^2 \xi \phi_t^2}{ n^2 \Gamma_{t}^2} (\Delta z^{s, \pi}_{k_s} )^2\Bigg)\Bigg].
\end{align*}

So long as $r < 1$, $1 + r + r^2 + r^3 + \cdots \leq \frac 1{1-r}$, and as $\frac{324 q^2 L_{\overline{\res}}^2}{n}  \leq 1$, replacing $\mathcal{S}_0$ by  \[\E{}{ \sum_{l = [t -q, t-1]} L_{k_l, k_t}^2 (\overline{\Delta}_{\max} z_{\iit_l}^{l, \pi} - \overline{\Delta}_{\min} z_{\iit_l}^{l, \pi})^2}\] yields
\begin{align*}
&\E{}{q \sum_{l = [t -q, t-1]} L_{k_l, k_t}^2 (\overline{\Delta}_{\max} z_{\iit_l}^{l, \pi} - \overline{\Delta}_{\min} z_{\iit_l}^{l, \pi})^2} \\
&\largespace \leq \mathbb{E}\Bigg[ \frac 43 \cdot \frac {r}{1-r}  \Big(\overline{\Delta}_{\min}  z_{\iit_t}^{t, \pi} -  \overline{\Delta}_{\min} z_{\iit_t}^{t, \pi}\Big)^2  +   \frac 89 \cdot \frac {r}{1-r}  (\Delta z^{t, \pi}_{k_t})^2 \\
&\largespace \largespace +  \sum_{s \in [t - 3q, t+ q] \setminus \{t\}} \frac{r}{486 q(1-r)}   \Big(\overline{\Delta}_{\max}  z_{\iit_s}^{s, \pi} -  \overline{\Delta}_{\min} z_{\iit_s}^{s, \pi}\Big)^2 \\
 &\largespace \largespace+ \sum_{s \in [t - 3q, t+ q] \setminus \{t\}} \frac{r}{486 q(1-r)} (\Delta z^{s, \pi}_{k_s} )^2\Bigg)\Bigg].
\end{align*}
 
\hide{We show the following lemma.
\begin{lemma}\label{lem::new::amor}
Suppose $r  = \max_t \{ \frac{(3q)^2 L_{\overline{\res}}^2 36 n^2 \xi \phi_{t}^2}{ \Gamma_{t}^2 n}\}$ is less than $\frac{1}{2}$, $\frac{(3q)^2 L_{\overline{\res}}^2 36}{n} \leq 1$ and  $\var{B}{t}$ are good, then
\begin{align*}
&\E{}{q \sum_{l = [t -q, t-1]} L_{k_l, k_t}^2 (\overline{\Delta}_{\max} z_{\iit_l}^{l, \pi} - \overline{\Delta}_{\min} z_{\iit_l}^{l, \pi})^2} \\
&\largespace \leq \mathbb{E}\Bigg[ 2 \Bigg(3 r  \Big(\overline{\Delta}_{\min}  z_{\iit_t}^{t, \pi} -  \overline{\Delta}_{\min} z_{\iit_t}^{t, \pi}\Big)^2  +  2 r (\Delta z^{t, \pi}_{k_t})^2 \\
&\largespace \largespace + \sum_{s \in [t - 3q, t+ q] \setminus \{t\}} \frac{2r}{q}   \Big(\overline{\Delta}_{\max}  z_{\iit_s}^{s, \pi} -  \overline{\Delta}_{\min} z_{\iit_s}^{s, \pi}\Big)^2 \\
 &\largespace \largespace+ \sum_{s \in [t - 3q, t+ q] \setminus \{t\}} \frac{2r}{q} (\Delta z^{s, \pi}_{k_s} )^2\Bigg)\Bigg].
\end{align*}
\end{lemma}}
\end{proof}

%% file: reAPCG.bbl
\begin{thebibliography}{10}

\bibitem{allen2016even}
Zeyuan Allen-Zhu, Zheng Qu, Peter Richt{\'a}rik, and Yang Yuan.
\newblock Even faster accelerated coordinate descent using non-uniform
  sampling.
\newblock In {\em International Conference on Machine Learning}, pages
  1110--1119, 2016.

\bibitem{avron2015revisiting}
Haim Avron, Alex Druinsky, and Anshul Gupta.
\newblock Revisiting asynchronous linear solvers: Provable convergence rate
  through randomization.
\newblock {\em Journal of the ACM (JACM)}, 62(6):51, 2015.

\bibitem{BertsakisTsiTsi1969}
Dimitri~P. Bertsekas and John~N. Tsitsiklis.
\newblock {\em Parallel and Distributed Computation:Numerical Methods}.
\newblock Prentice Hall, 1989.

\bibitem{bradley2011parallel}
Joseph~K Bradley, Aapo Kyrola, Danny Bickson, and Carlos Guestrin.
\newblock Parallel coordinate descent for l1-regularized loss minimization.
\newblock {\em arXiv preprint arXiv:1105.5379}, 2011.

\bibitem{ChazenMir1969}
D.~Chazan and W.~Miranker.
\newblock Chaotic relaxation.
\newblock {\em Linear Algebra and its Applications}, 2(2):199--222, 1969.

\bibitem{2016arXiv161209171K}
Yun~Kuen Cheung and Richard Cole.
\newblock {A Unified Approach to Analyzing Asynchronous Coordinate Descent and
  Tatonnement}.
\newblock {\em ArXiv e-prints}, December 2016.

\bibitem{CCT2018}
Yun~Kuen Cheung, Richard Cole, and Yixin Tao.
\newblock A unified approach to analyzing asynchronous coordinate descent ---
  standard and partitioned.
\newblock {\em manuscript}, 2018.

\bibitem{DevolderGN2014}
Olivier Devolder, Fran\c{c}ois Glineur, and Yurii Nesterov.
\newblock First-order methods of smooth convex optimization with inexact
  oracle.
\newblock {\em Math. Program.}, 146(1-2):37--75, August 2014.

\bibitem{fang2018accelerating}
Cong Fang, Yameng Huang, and Zhouchen Lin.
\newblock Accelerating asynchronous algorithms for convex optimization by
  momentum compensation.
\newblock {\em arXiv preprint arXiv:1802.09747}, 2018.

\bibitem{fercoq2015accelerated}
Olivier Fercoq and Peter Richt{\'a}rik.
\newblock Accelerated, parallel, and proximal coordinate descent.
\newblock {\em SIAM Journal on Optimization}, 25(4):1997--2023, 2015.

\bibitem{hannah2018a2bcd}
Robert Hannah, Fei Feng, and Wotao Yin.
\newblock A2bcd: An asynchronous accelerated block coordinate descent algorithm
  with optimal complexity.
\newblock 2018.

\bibitem{lee2013efficient}
Yin~Tat Lee and Aaron Sidford.
\newblock Efficient accelerated coordinate descent methods and faster
  algorithms for solving linear systems.
\newblock {\em arXiv preprint arXiv:1305.1922}, 2013.

\bibitem{LinLuXiao2015}
Qihang Lin, Zhaosong Lu, and Lin Xiao.
\newblock An accelerated randomized proximal coordinate gradient method and its
  application to regularized empirical risk minimization.
\newblock {\em SIAM J. Optim.}, 25:2244–2273, 2015.

\bibitem{LiuW2015}
Ji~Liu and Stephen~J. Wright.
\newblock Asynchronous stochastic coordinate descent: Parallelism and
  convergence properties.
\newblock {\em {SIAM} Journal on Optimization}, 25(1):351--376, 2015.

\bibitem{LWRBS2015}
Ji~Liu, Stephen~J. Wright, Christopher R{\'{e}}, Victor Bittorf, and Srikrishna
  Sridhar.
\newblock An asynchronous parallel stochastic coordinate descent algorithm.
\newblock {\em Journal of Machine Learning Research}, 16:285--322, 2015.

\bibitem{MPPRRJ2015}
Horia Mania, Xinghao Pan, Dimitris~S. Papailiopoulos, Benjamin Recht, Kannan
  Ramchandran, and Michael~I. Jordan.
\newblock Perturbed iterate analysis for asynchronous stochastic optimization.
\newblock {\em CoRR}, abs/1507.06970, 2015.

\bibitem{nesterov2012efficiency}
Yu~Nesterov.
\newblock Efficiency of coordinate descent methods on huge-scale optimization
  problems.
\newblock {\em SIAM Journal on Optimization}, 22(2):341--362, 2012.

\bibitem{nesterov1983method}
Yurii~E Nesterov.
\newblock A method for solving the convex programming problem with convergence
  rate o (1/k\^{} 2).
\newblock In {\em Dokl. Akad. Nauk SSSR}, volume 269, pages 543--547, 1983.

\bibitem{richtarik2016parallel}
Peter Richt{\'a}rik and Martin Tak{\'a}{\v{c}}.
\newblock Parallel coordinate descent methods for big data optimization.
\newblock {\em Mathematical Programming}, 156(1-2):433--484, 2016.

\bibitem{Sun2017}
Tao Sun, Robert Hannah, and Wotao Yin.
\newblock Asynchronous coordinate descent under more realistic assumptions.
\newblock In I.~Guyon, U.~V. Luxburg, S.~Bengio, H.~Wallach, R.~Fergus,
  S.~Vishwanathan, and R.~Garnett, editors, {\em Advances in Neural Information
  Processing Systems 30}, pages 6182--6190. Curran Associates, Inc., 2017.

\bibitem{wright2015coordinate}
Stephen~J Wright.
\newblock Coordinate descent algorithms.
\newblock {\em Mathematical Programming}, 151(1):3--34, 2015.

\end{thebibliography}
